\documentclass[11pt]{amsart}
\usepackage[T1]{fontenc}
\usepackage{float}
\usepackage{graphicx}
\usepackage[printonlyused]{acronym}
\usepackage{hyperref}
\usepackage{xcolor}
\usepackage[margin=0.8in]{geometry}
\usepackage{amsmath}
\usepackage{amsthm}
\usepackage{multicol}
\usepackage{bm}
\usepackage[]{graphicx}
\usepackage{subfig}
\usepackage{color}
\usepackage{amssymb, epsfig}
\usepackage{url}
\usepackage{hyperref}
\usepackage{bbm}    
\usepackage{mathrsfs}
\newtheorem{theorem}{Theorem}[section]
\newtheorem{lemma}[theorem]{Lemma}
\newtheorem{Corollary}[theorem]{Corollary}

\newtheorem{proposition}[theorem]{Proposition}

\usepackage[utf8]{inputenc}
\usepackage{xcolor}
\usepackage{bbm}
\usepackage[style=numeric, citestyle=numeric-comp]{biblatex}
\DefineBibliographyStrings{english}{
  in = {} 
}

\DeclareFieldFormat[article]{title}{#1}
\DeclareFieldFormat[book]{title}{#1}
\DeclareFieldFormat[inbook]{title}{#1}
\DeclareFieldFormat[incollection]{title}{#1}
\DeclareFieldFormat[inproceedings]{title}{#1}
\DeclareFieldFormat[thesis]{title}{#1}
\DeclareFieldFormat[misc]{title}{#1} 
\DefineBibliographyStrings{english}{
  in = {} 
}

\DeclareFieldFormat{labelnumber}{[#1]} 

\DeclareFieldFormat{labelnumber}{#1}
\DeclareFieldFormat{pages}{#1} 
\DeclareFieldFormat[article]{volume}{\textbf{#1}}

\title{ Long-time behavior for systems of Fisher-KPP type with interacting components}
\author{Alexandra Stavrianidi$^*$}
\thanks{$^*$Institute for Analysis and Numerics, University of M\"unster, Germany; alex.st@uni-muenster.de}
\date{}
\begin{document}


\maketitle
\begin{abstract}
We study the long-time behavior of a triangular system of Fisher--KPP type with $k$ interacting components, associated with a reducible multitype branching Brownian motion with $k$ types of particles. For this cascading system, we prove convergence in shape of each component to the minimal-speed Fisher--KPP traveling wave and determine the front asymptotics up to the constant order. This yields a PDE proof of Conjecture 1.2 from \cite{belloumpaper} on the convergence in distribution of the centered maximum particle in a cascading branching Brownian motion. We also derive asymptotic front-location estimates for such systems with general Fisher--KPP nonlinearities.

\end{abstract}

\section{Introduction}

We study the long-time behavior of a triangular system of Fisher–KPP type equations with interacting components,

\begin{equation} \label{cascadinggeneralf}
\begin{split}
&v^{i}_{t}= v^{i}_{xx}+f(v^{i})+\alpha v^{i+1}\left(1-v^{i}\right) \text{ for }  i \in \{1,\ldots,k-1\}, \quad{ t>0, x \in \mathbb{R},} \\
& v^{k}_{t}= v^{k}_{xx}+f(v^{k}),  \quad{ t>0, x \in \mathbb{R}},\\
&v^{i}(0,x)=v^{i}_{0}(x)  \text{ for }  i \in \{1,\ldots,k\}, x \in \mathbb{R}.
\end{split}
\end{equation}

where the reaction term $f \in C^{2}([0,1])$ is of Fisher-KPP type, i.e.

\begin{equation} \label{fkppconditions}
f(0)=f(1)=0, \quad f^{\prime}(0)>0, \quad f^{\prime}(1)<0, \quad  0 < f(u) \le f'(0)u, \quad u \in (0,1).
\end{equation}
We consider initial conditions $v^{i}_0(x)$ that are compact perturbations of the Heaviside step function, i.e., there exist $x_1,x_2$ so that $v^{i}(0,x)=1$ for all $x \leq x_1$, and $v^{i}(0,x)=0$ for all $x \geq x_2$ with $x_1 \leq x_2$ and $0\leq v^{i}(0,x) \leq 1$ for all $x \in \mathbb{R}$. Then we have $0<v^{i}(t,x)<1 $ for all $t>0$ and we interpret $v^{i}$ as population densities, allele frequencies, or chemical concentrations. The system is monotone and models a mutualistic interaction between the species, as the presence of $v^{i+1}$ enhances the growth of $v^{i}$ for $1 \leq i \leq k-1$. 

The motivation to study this system comes from a reducible multitype cascading branching Brownian motion (BBM) introduced in \cite{mallein_belloum, belloumpaper} in which particles of type $i$ produce particles of type $i+1$ for $1 \leq i \leq k-1$, while particles of type $k$ only produce particles of type $k$. In the special case $f(v)=v-v^2$, system \eqref{cascadinggeneralf} describes the evolution of the McKean functionals associated with this cascading BBM. The main goal of this paper is to characterize the long-time behavior of system \eqref{cascadinggeneralf}, and provide a PDE proof of Conjecture 1.2 from \cite{belloumpaper} on the convergence in distribution of the centered maximum particle in this cascading BBM.

We begin by providing some preliminaries on the Fisher-KPP equation and its connection with branching Brownian motion and then we state our main results.

\subsection{Preliminaries on the Fisher-KPP equation and branching Brownian motion}

The equation for $v^{k}$ in system \eqref{cascadinggeneralf} is the classic Fisher-KPP equation, introduced in 1937 by Fisher to model the spread of advantageous genes \cite{Fisher}. Kolmogorov, Petrovskii and Piskunov studied the long-time behavior of this equation \cite{KPPpaper}, and since then the dynamics have been understood in great detail. The stable colonized state $1$ invades the unstable uncolonized state $0$, as the region where $u \approx 1$ displaces the region $u \approx 0$. In one dimension, this is described mathematically by the following theorem on the existence of traveling waves connecting the two equilibria.

\begin{theorem}{(\cite{KPPpaper})} \\
The Fisher-KPP equation admits traveling wave solutions of the form $v(t,x)=U_{c}(x-ct)>0$ satisfying the boundary conditions $
U_{c}(-\infty)=1, U_{c}(+ \infty)=0
$ for all speeds $c \geq c_{*}=2 \sqrt{f'(0)}$. Moreover, $U_{c}$ is decreasing and unique up to translation.
\end{theorem}
Traveling wave solutions of the Fisher-KPP equation satisfy the ODE problem
\begin{equation}
\begin{split}
&U_{c}''+cU_{c}'+f(U_{c})=0 \\
&U_{c}(-\infty)=1, U_{c}(+\infty)=0.
\end{split}
\end{equation}

Their asymptotics (after fixing the translation) are the following: 

For $c>c_{*}$, there exists $k_{c}>0$ and $\delta>0$ such that for $\lambda_{c}=\frac{c-\sqrt{c^{2}-4f'(0)}}{2},$

\begin{equation*}
U_{c}(x)=k_{c} e^{-\lambda_{c} x}+O\left(e^{-\left(\lambda_{c}+\delta\right) x}\right) \text { as } x \rightarrow+\infty.
\end{equation*}

and for the minimal speed $c_{*}$, there exists $k_{*} \in \mathbb{R}$ such that for $\lambda_{*}=\frac{c_{*}}{2},$

\begin{equation} \label{travelingwaveasymptotics}
U_{c_{*}}(x)=\left(x+k_{*}\right) e^{-\lambda_{*} x}+O\left(e^{-\left(\lambda_{*}+\delta\right) x}\right) \text { as } x \rightarrow+\infty. 
\end{equation}

The minimal speed traveling wave is special because it attracts the solutions to Fisher-KPP with front-like initial data.

\begin{theorem} {(\cite{KPPpaper})} \label{convergenceinshape}
\newline 
Let $v(t, x)$ be the solution to the Fisher-KPP equation with  $v(0,x)=\mathbbm{1}_{\{x \leq 0\}}$. There exists a frame $m(t)$ with $m(t)=c_{*}t+o(t)$ such that $
v(t, x+m(t)) \rightarrow U_{c_{*}}\left(x\right) \text { as } t \rightarrow+\infty, \text { uniformly.}$
\end{theorem}

This theorem was proven for more general front-like initial data in \cite{rothe}. The frame $m(t)$ where the solution is bounded away from $0$ and $1$ at time $t$ is referred to as the equation front. The asymptotics of $m(t)$ are of great interest as they help us track the progress of the invasion. In \cite{KPPpaper} the asymptotics were derived up to linear order, but in \cite{Bramson, bramson2} Bramson showed that the invasion lags logarithmically in time, \textit{i.e.} there exists a constant $x_{\infty}$ that depends on the initial data such that $m(t)=c_{*}t-\frac{3}{2 \lambda_{*}} \ln t+x_{\infty}+o(1)$.  This proof relied on the connection between the Fisher-KPP equation and branching Brownian Motion.
The beautiful connection between Fisher-KPP and branching Brownian motion was discovered independently by McKean, Skorohod, Ikeda, Nagasawa and Watanabe in  \cite{McKean, skorohod,watanabe}.

Branching Brownian motion (BBM) is a process that begins with a single particle at position $x$ at time $0$, performing a Brownian motion and carrying an exponentially distributed clock. When that clock rings, the particle splits into $k$-many particles with probability $p_{k}$, and each “child” particle performs an independent copy of this process. At time $t$, there is a cloud of exponentially many particles on average and questions arise on how BBM invades its environment. Consider the BBM functional $v(t, x)=\mathbb{E}_{x}\left(\prod_{k=1}^{\mathcal{N}_{t}} g\left(X_{k}(t)\right)\right)$ for a $g$ bounded, where $\mathcal{N}_t$ denotes the set of particles alive at time $t$ and $X_{k}(t)$ denote their positions. Then, $u(t,x)=1-v(t,x)$ solves the Fisher-KPP equation with initial condition $u(0,x)=1-g(x)$ and nonlinearity 
$f(u)= \beta (1-u) - \beta \sum_{k=1}^{\infty} p_{k}(1-u)^{k}$, where $\beta$ is the branching rate. These are the so-called McKean nonlinearities, and they are a strict subclass of Fisher-KPP nonlinearities \cite{votingmodels2}.

Prescribing  $u(0,x)=\mathbbm{1}_{\{x \leq 0\}}$ as the initial condition for $u$, so that $g(x)= \mathbbm{1}_{\{x \geq 0\}}$  we have that $u(t,x)= \mathbb{P}_{0}\left( \max_{k \in \mathcal{N}_t} X_{k}(t) \geq x\right)$. Therefore the cumulative distribution function of the rightmost particle $M(t)=\max_{k \in \mathcal{N}_t} X_{k}(t)$ for a BBM starting at $0$ satisfies the Fisher-KPP equation with the Heaviside initial condition. Hence Theorem $\ref{convergenceinshape}$ implies that $M(t)-m(t)$ converges in law to a random variable whose distribution is given by the traveling wave $U_{c_*}$. Moreover, $m(t)$ is the median location of the maximum particle. Consequently, probabilistic methods can also be applied to determine the asymptotics of $m(t)$, as Bramson did to obtain the $\frac{3}{2}\ln(t)$ delay. This $\frac{3}{2}\ln(t)$ is known as the Bramson correction and it is universal, as it appears in numerous other contexts, such as in the extreme value theory of log-correlated fields \cite{arguin}.

Aiming to study the long-time behavior of the Fisher-KPP equation with analytic tools, Hamel, Nolen, Roquejoffre, and Ryzhik provided a PDE proof of the Bramson correction \cite{short_proof} which holds for all  Fisher-KPP nonlinearities. Later, Nolen, Roquejoffre, and Ryzhik refined the asymptotics of $m(t)$ up to order $O(\ \frac {1}{t^{1-\gamma}})$ for any $\gamma>0$ and uncovered the algebraic rate of convergence to the traveling wave \cite{refined_asymptotics}. In 2019, Graham determined the asymptotics of $m(t)$ up to $O(\frac{1}{t})$, and showed that below that order, the asymptotics of $u$ can't be described as shifts of the wave \cite{precise_asymptotics}. The sharpest asymptotics (normalizing $f'(0)=1$) are

\begin{equation}
m(t):=2 t-\frac{3}{2} \ln t+ C_{1}-\frac{3 \sqrt{\pi}}{\sqrt{t}}+ \frac{9}{8}(5-6 \log 2) \frac{\log t}{t}+\frac{C_{2}}{t}
\end{equation}

where the constants $C_1,C_2$ depend on the initial data, while the universal coefficients were previously predicted by Ebert and van Saarloos \cite{ebertvansaarloos, vansaarloos} and Berestycki, Brunet, and Derrida \cite{berestyckibrunetderrida}. Our paper extends this line of research to Fisher-KPP systems with interacting components.

\subsection{Main results}

The main result of the paper concerns the location of the front of the $v^{i}$ component of system \eqref{cascadinggeneralf} with $f(v)=v-v^2$.

\begin{theorem} \label{convergencetothewavecascading}
Let $\{v^{i}\}_{i=1}^{k}$ solve system~\eqref{cascadinggeneralf} with $f(v)=v-v^2$.  
Then the component $v^{i}$ converges in shape to the minimal-speed Fisher--KPP traveling wave $U_{c_{*}}(x)$, $c_{*}=2$ in the following sense:  
there exists a constant $x^{i}_{\infty} \in \mathbb{R}$, depending on the initial data, and a moving frame $m^{i}(t)$ such that
\[
m^{i}(t) = 2t - \tfrac{3}{2}\ln t + (k-i)\ln t -x^{i}_{\infty}+ o(1),
\quad \text{as } t \to \infty,
\]
and
\[
\lim_{t \to \infty} v^{i}(t, x + m^{i}(t)) = U_{c_{*}}(x)
\quad \text{uniformly in } x \text{ on compact sets}.
\]

In particular, $x^{i}_{\infty}=-\ln(\alpha^{i}_{\infty})=-\ln(\frac{\alpha^{k-i}}{(k-i)!})-x^{k}_{\infty}$, where $x^{k}_{\infty}$ is the constant correction of the frame of the single Fisher-KPP equation.
\end{theorem}

Interestingly, the coupling parameter $\alpha$ does not enter the logarithmic correction for the front of $v^{i}$, but appears in the choice of the specific translation of the traveling wave that $v^{i}$ converges to. The long-time behavior of the $v^{1}$ component of the cascading system is connected to the asymptotic location of the maximum particle of the associated cascading branching Brownian motion. Theorem \ref{convergencetothewavecascading} provides a PDE proof for Conjecture 1.2 from \cite{belloumpaper}.
Theorem \ref{convergencetothewavecascading} can be generalized to $C^2$ Fisher-KPP nonlinearities with $f'(v) \leq f'(0), v \in (0,1).$
\\ 

For general Fisher-KPP nonlinearities satisfying \eqref{fkppconditions} we also show results on the asymptotic location of the fronts. In the case $k=2$, system \eqref{cascadinggeneralf} becomes

\begin{equation} \label{mysystem}
\begin{split}
& u_{t}= u_{x x}+f(u)+\alpha v (1-u), \quad{t>0, x \in \mathbb{R}}, \\
& v_{t} =v_{x x}+f(v), \quad{ t>0, x \in \mathbb{R}}, \\
& u(0,x)=u_{0}(x), v(0,x)=v_{0}(x), \quad{ x \in \mathbb{R}}.
\end{split}
\end{equation}

\begin{theorem} \label{levelsets}
Let $(u,v)$ be the solution to the system~\eqref{mysystem}, and let $E_{m}(t)=\{x>0:\, u(t,x)=m\}$. 
For every $m \in (0,1)$, there exists $C \ge 0$ such that
\[
E_{m}(t) \subset 
\Big[c_{*} t-\tfrac{1}{2\lambda_{*}} \ln t - C,\,
      c_{*} t-\tfrac{1}{2\lambda_{*}} \ln t + C \Big]
\quad \text{for all sufficiently large } t,
\]
where $c_{*}=2\sqrt{f'(0)}$ and $\lambda_{*}=\sqrt{f'(0)}$.
\end{theorem}

We note that this system can be associated with a two-type BBM with particles that have the same diffusivities and branching rates.  The asymptotics that are uncovered in \cite{belloumpaper} for the median location of the maximum particle of this two-type BBM are in line with our Theorem \ref{levelsets}  on the location of the front of $u$ in \eqref{mysystem}. We note, however, that our result holds for all initial conditions that are compact perturbations of the Heaviside function and for all Fisher-KPP nonlinearities, including those that cannot be associated with a branching process.

We also obtain an upper bound on the front location of each component $v^{i}$ in the cascading system~\eqref{cascadinggeneralf} with a general Fisher-KPP nonlinearity satisfying \eqref{fkppconditions}.

\begin{theorem} \label{cascadingupperbound}
Let $\{v^{i}\}_{i=1}^{k}$ be the solution to the system~\eqref{cascadinggeneralf}.  
For any $m \in (0,1)$, let $E^{i}_{m}(t)=\{x>0:\, v^{i}(t,x)=m\}$.  
Then there exists $C \geq 0$ such that
\begin{equation} \label{maxlevelsetcascading}
\max E^{i}_{m}(t) 
\le c_{*}t - \frac{\big(\tfrac{3}{2}-k+i\big)}{\lambda_{*}} \ln t + C, 
\quad \text{for all sufficiently large } t,
\end{equation}
where $c_{*}=2\sqrt{f'(0)}$ and $\lambda_{*}=\sqrt{f'(0)}$.
\end{theorem}

We note that the corresponding lower bound cannot be obtained using the same argument as in the case $k=2$.  
The technical obstruction is that $k-i-\tfrac{3}{2}>0$ for $k \ge 3$ and $i \leq k-2$, which invalidates the subsolution construction.

The proof strategy relies on analyzing the linearized system with Dirichlet boundary conditions in appropriately chosen moving frames, allowing for the construction of subsolutions and supersolutions for the nonlinear problem.  
The moving boundary is chosen so that the solution to the linearized problem neither grows nor decays in time.  
Through spectral analysis, we obtain precise asymptotics for the linearized system, which are then matched with the Fisher--KPP traveling-wave asymptotics to pinpoint the front locations of the nonlinear system.  
The arguments build on techniques developed by Hamel, Nolen, Roquejoffre, Ryzhik, and Graham \cite{convergencetothewave,precise_asymptotics,short_proof,refined_asymptotics}.

\subsection{Connections with related work on Fisher-KPP systems with interacting components}
Relevant work on Fisher-KPP systems with interacting components includes Holzer's investigation \cite{holzer_proof} of the anomalous spreading phenomenon exhibited in the system 

\begin{equation} \label{holzer}
\begin{split}
 u_{t} & =\sigma^2 u_{x x}+\beta u(1-u)+\alpha v \\
v_{t} & =v_{x x}+v(1-v) 
\end{split}
\end{equation}

In the uncoupled regime ($\alpha=0$) we have two independent Fisher-KPP equations so that $u,v$ spread with the Fisher-KPP speeds $c_u=2\sqrt{ \beta \sigma^2}$, and $c_{v}=2$ respectively. For  $\alpha>0$, and for $(\sigma^2, \beta)$ in a certain regime, the presence of species $v$ enhances the growth of $u$ so that in the long time limit, the two populations separate. Holzer showed that for $(\sigma^2, \beta)$ in said regime, $u$ spreads at a speed that exceeds the maximum of $c_u$ and $c_v$, at the so-called anomalous speed, $c_{anom}$. 
This fact is remarkable, given that $v$ is close to $0$ at $c_{anom}t$, (the front of $v$ is located far behind, at $2t-\frac{3}{2}\log t$), but it still manages to force $u$ enough to move its front at $c_{anom}t$. Faye and Peltier \cite{faye_anomalous_for_monostable} showed that this anomalous spreading persists for more general monostable nonlinearities $f$ and more general interaction terms.

Our system \eqref{mysystem} corresponds to the point $\sigma^2=\beta=1$ of the phase diagram of \cite{holzer_proof} where the boundaries intersect. It manifests the same phenomenon of population separation, but it is more subtle, as the separation appears in the sublinear order of the front asymptotics. Indeed, $u$ spreads at the Fisher--KPP minimal speed $c_{*}$.  
However, the forcing by $v$ has a subtle effect on the propagation of $u$, increasing the logarithmic correction of its front from $-\frac{3}{2 \lambda_{*}}\ln(t)$ to $-\frac{1}{2 \lambda_{*}}\ln(t)$. The fronts of $u$ and $v$ are then $\frac{1}{\lambda_{*}}\ln t$ distance apart, and separate in the long-time limit.

To the best of our knowledge, this work provides the first PDE determination of a sublinear front correction of $\tfrac{1}{2}\ln t$ in a system of reaction--diffusion equations with interacting components.  
In fact, the usual Bramson correction of $\frac{3}{2}\ln(t)$ has been shown to be universal among a large class of systems of reaction-diffusion equations with interacting components \cite{montiepaper}. Interesting works on systems with interacting components that exhibit the usual Bramson correction  include \cite{threecomponentreactiondiffusion, longtimesismodel} which study a farmer-hunter-gatherer (three component) model and a SIS epidemiological model respectively. Moreover, for irreducible multi-type branching Brownian motions, convergence to traveling wave solutions with the classical Bramson correction has been established in \cite{extremalprocessofirreduciblebbm}.  In that sense, this triangular system is special.

Other interesting works on coupled systems of Fisher-KPP type that exhibit anomalous spreading phenomena include \cite{lotkavolterrainvasion}, \cite{hartungadvancingpopulation}, \cite{holzelandscheel}.

\noindent\textbf{Organization of the paper.}
The paper is organized as follows. In Section \ref{Section 2} we establish the connection between the system~\eqref{cascadinggeneralf} with $f(v)=v-v^2$ and a cascading branching Brownian motion, providing a probabilistic interpretation of Theorem~\ref{convergencetothewavecascading}.  
Section \ref{Section 3} contains some PDE heuristics for the derivation of the logarithmic correction. Section \ref{section 4} contains the proofs of Theorems \ref{levelsets} and \ref{cascadingupperbound} on the long-time front asymptotics for the systems~\eqref{mysystem} and~\eqref{cascadinggeneralf} with general Fisher-KPP nonlinearities.  
Section \ref{Section 4} is devoted to the proof of Theorem~\ref{convergencetothewavecascading} with $f(v)=v-v^2$ and its implication for the cascading branching Brownian motion.

\section{Connection with multitype branching Brownian motion} \label{Section 2}
We will now describe a connection between the cascading system of equations \eqref{cascadinggeneralf} and a cascading branching Brownian motion, as defined in \cite{belloumpaper}.

We consider $k$-many types of particles on $\mathbb{R}$. All types of particles move according to a Brownian motion with variance $\sigma^2=2$.  Particles of type $i \in \{1,\ldots,k-1\} $ branch at rate $1$ into two particles of type $i$ and at rate $\alpha$ into one particle of type $i$ and one particle of type $i+1$.
Particles of type $k$ only branch at rate $1$ into two particles of type $k$.\\

\begin{figure}[h!]
\centering
\begin{minipage}{0.20\textwidth}
  \includegraphics[width=\linewidth]{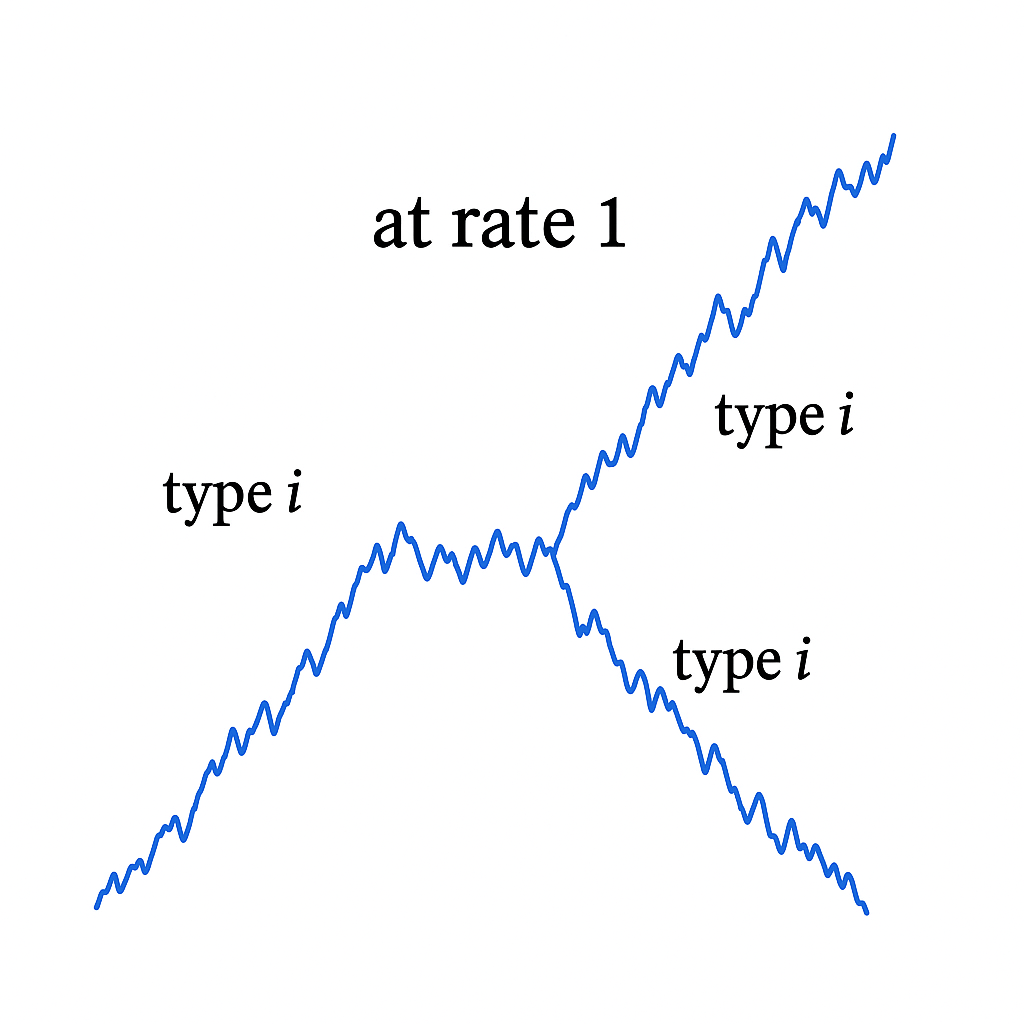}
\end{minipage}\hfill
\begin{minipage}{0.20\textwidth}
  \includegraphics[width=\linewidth]{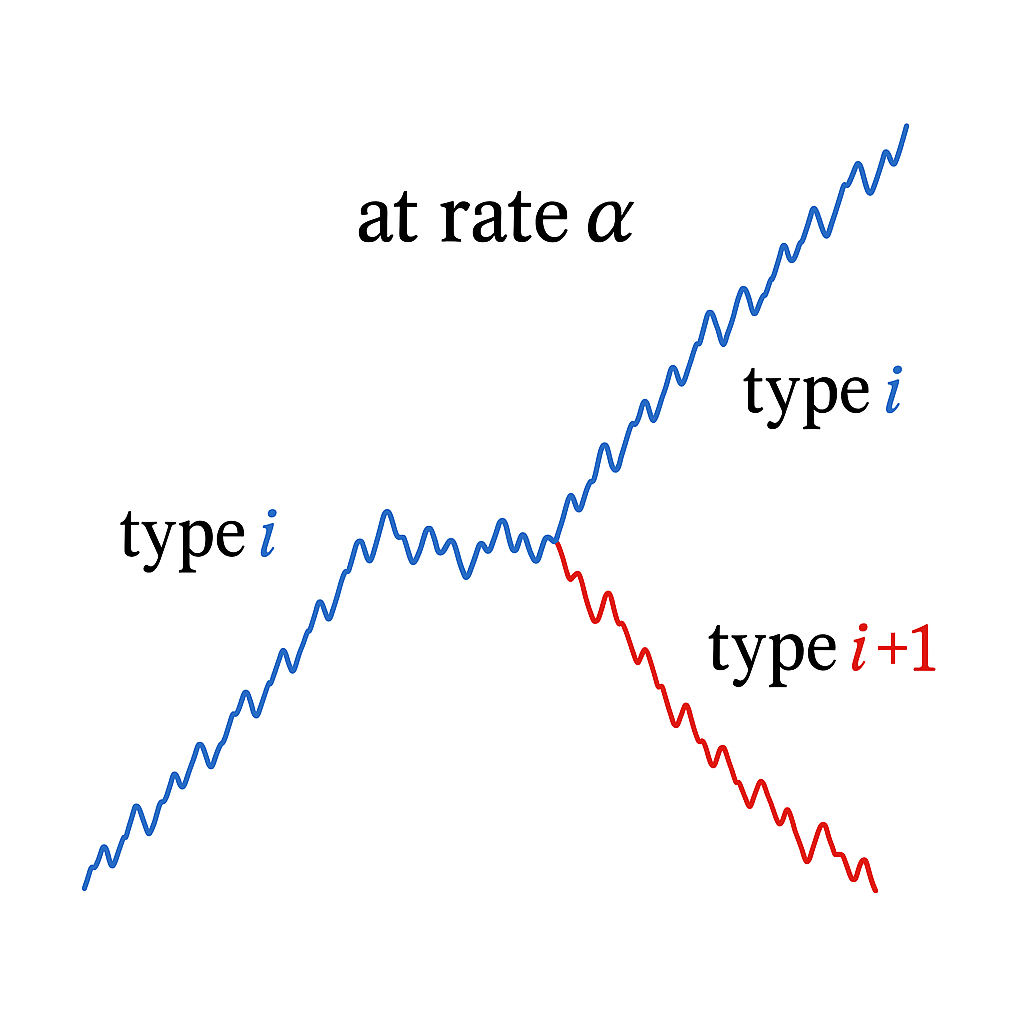}
\end{minipage}\hfill
\begin{minipage}{0.20\textwidth}
  \includegraphics[width=\linewidth]{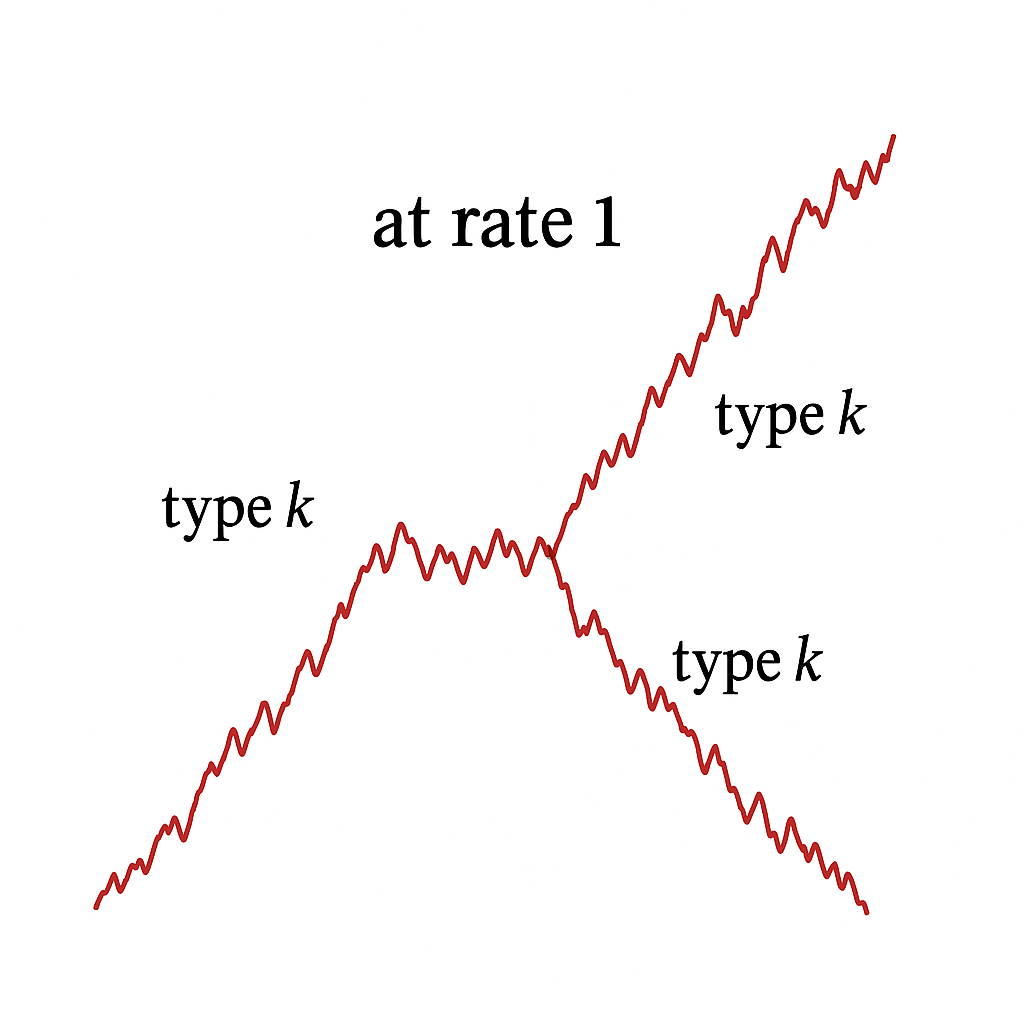}
\end{minipage}
\end{figure}

For all $t \geq 0$, we denote by $\mathcal{N}_{t}^{(i)}, i \in \{1,\ldots,k\}$ the set of particles of type $i$ at time $t$, and by $\mathcal{N}_{t}$ the set of all particles at time $t$.  We look at the following McKean functionals, where $\mathbb{E}^{(i,x)}$ denotes expectation for the process started with a single particle of type $i$ at position $x$.

$$
\begin{aligned}
& u^{i}(t, x)=\mathbb{E}^{(i,x)}\left(\prod_{u \in \mathcal{N}_{t}^{(i)}} f^{i}\left(X_{u}(t)\right) \prod_{u \in \mathcal{N}_{t}^{(i+1)}} f^{i+1}\left(X_{u}(t)\right) \ldots \prod_{u \in \mathcal{N}_{t}^{(k)}} f^{k}\left(X_{u}(t)\right)     \right). \\
\end{aligned}
$$

Here $u^{i}(0,x)=f^{i}(x)$ is the initial condition for each component. Since particles of type $k$ only branch into particles of type $k$, $u^{k}(t, x)=\mathbb{E}^{(k,x)}\left(\prod_{u \in \mathcal{N}_{t}^{(k)}} f^{k}\left(X_{u}(t)\right) \right).$ As a result, $u^{k}(t,x)$ satisfies the classic Fisher-KPP equation

$$
\begin{aligned}
u^{k}_{t}=u^{k}_{xx}+u^{k}(1-u^{k}) .
\end{aligned}
$$

Each particle of type $i \in \{1,\ldots,k-1\}$ branches into two particles of type $i$ when clock $\tau_1$ rings, and into one particle of type $i$ and one particle of type $i+1$ when clock $\tau_2$ rings with $P(\tau_1>t)=e^{-t}, P(\tau_2>t)=e^{-\alpha t}.$ Each particle of type $k$ branches into two particles of type $k$ when $\tau_1$ rings.  We write a renewal equation for $u^{i}(t,x)$ for $i \in \{1,\ldots,k-1\}$ conditioning on the first branching event. Until time $t$, we have one of the following cases.

\begin{enumerate}
\item{ The initial particle of type $i$ has not branched yet. 
This happens if none of $\tau_1,\tau_2$ have rang, which has probability $P(\tau_1>t)P(\tau_2>t)=e^{-(\alpha+1)t}$.}
\item{ The initial particle of type $i$  branched at some time $s<t$ into two particles of type $i$. This happens if the clock $\tau_1$ rang first and then we have two independent branching Brownian motions, both starting with particles of type $i$ at position $x+B_{s}$ and running for time $t-s$.
$$
\begin{aligned}
\mathbb{P}\left( \{\tau_{1} \in d s \} \cap \{\tau_{1}<\tau_{2} \}\right)=\mathbb{P}\left(\tau_{1} \in d s\right) \mathbb{P} \left( \tau_{1}<\tau_{2}  | \tau_{1}\in d s \right)=e^{-(1+\alpha)s}d s
\end{aligned}
$$}
\item{The initial particle of type $i$ branched into one particle of type $i$ and one particle of type $i+1$. 
This happens if the clock $\tau_2$ rang first at some time $s<t$, and then we have two independent branching Brownian motions, one with a particle of type $i$, the other with a particle of type $i+1$, starting at position $x+B_s$ and running for time $t-s$.
\newline 
$$
\begin{aligned}
\mathbb{P}\left(\{\tau_{2} \in d s \} \cap \{ \tau_{2}<\tau_{1} \} \right)=\mathbb{P}\left(\tau_{2} \in d s\right) \mathbb{P} \left( \tau_{2}<\tau_{1}  | \tau_{2}\in d s \right)=\alpha e^{-(1+\alpha)s}d s
\end{aligned}
$$
}
\end{enumerate}
\medskip
We get the following renewal equation.

\begin{equation}
\begin{split}
& u^{i}(t, x) =\mathbb{E}^{(i,x)}\left(f^{i}\left(B_{t}\right)\right) \mathbb{P}\left(\tau_{1}>t\right)\mathbb{P}\left(\tau_{2}>t\right) +\int_{0}^{t} \mathbb{E}^{(i,x)}\left(\left[u^{i}\left(t-s, B_{s}\right)\right]^{2}\right) \mathbb{P}\left( \{\tau_{1} \in d s \} \cap \{\tau_{1}<\tau_{2} \}\right) \\
&+\int_{0}^{t} \mathbb{E}^{(i,x)}\left(u^{i}\left(t-s,B_{s}\right)u^{i+1}\left(t-s, B_{s}\right)\right) \mathbb{P}\left(\{\tau_{2} \in d s \} \cap \{ \tau_{2}<\tau_{1} \} \right)  \\
& =\mathbb{E}^{(i,x)} \left(f^{i}\left(B_{t}\right)\right) e^{-(1+\alpha)t} + \int_{0}^{t} \mathbb{E}^{(i,x)} \left(\left[u^{i} \left(t-s,B_{s}\right)\right]^{2}+\alpha u^{i}\left(t-s,B_{s}\right)u^{i+1}\left(t-s, B_{s}\right) \right) e^{-(1+\alpha)s} d s .
\end{split}
\end{equation}

Let $\left(e^{\mathcal{L} t}\right)_{t \ge 0}$ denote the semigroup generated by $\mathcal{L} =\partial_{xx} - (1+\alpha).$ We rewrite the above equation as 

\begin{equation}
\begin{split}
&u^{i}(t,x) =\left[e^{\mathcal{L} t} f^{i}(\cdot)\right](x)+ \int_{0}^{t} \left[e^{\mathcal{L} s} \left(\left( u^{i}\right)^{2}+\alpha u^{i}u^{i+1}\right)(t-s, \cdot) \right](x) d s.
\end{split}
\end{equation}

By the Duhamel formula we obtain that $u^{i}$ satisfies the initial value problem

\begin{equation}
\begin{split}
& u^{i}_{t}= u^{i}_{xx}-(1+\alpha)u^{i} +\left(u^{i} \right)^{2} +\alpha u^{i}u^{i+1}\\
& u^{i}(0, x)=f^{i}(x)
\end{split}
\end{equation}

For $v^{i}=1-u^{i}$ our system of equations then becomes

\begin{equation} \label{cascadingsystemequations}
\begin{split}
&v^{i}_{t}= v^{i}_{xx}+v^{i}-\left(v^{i}\right)^2 +\alpha v^{i+1}\left(1-v^{i}\right) \text{ for }  i \in \{1,\ldots,k-1\} \\
& v^{k}_{t}= v^{k}_{xx}+v^{k}-\left(v^{k}\right)^2 \\
&v^{i}(0,x)=1-f^{i}(x)
\end{split}
\end{equation} 

We have thus derived system \eqref{cascadinggeneralf} with nonlinearity $v-v^2$. Moreover,  for $f^{i}(x)=\mathbbm{1}_{x>0},$

\begin{equation} \label{cdfcascadingbbmequation}
\begin{split}
& v^{1}(t, x)=1-\mathbb{P}^{(1,x)}\left\{\text { all } X_{u}(t) \geq 0\right\} =\mathbb{P}^{(1,x)}\left\{\min _{u \in \mathcal{N}_{t}} X_{u}(t) \leq 0\right\}= \mathbb{P}^{(1,0)}\left\{\max _{u \in \mathcal{N}_{t}} X_{u}(t) \geq x\right\}.   \\
\end{split}
\end{equation}

Therefore $v^{1}(t,x)$ is the cumulative distribution function of the maximum particle of the cascading BBM starting from a particle of type $1$ at the origin.  Conjecture 1.2 from \cite{belloumpaper} predicts that the maximum particle of this process centered at
$m(t)=2t-\frac{3}{2}\ln(t)+(k-1)\ln(t)$ converges in distribution to a randomized Gumbel. Theorem \ref{convergencetothewavecascading} on the location of the front of $v^{1}$  confirms this. 

Note that for $k=2$ we have a BBM with $2$ types of particles, and we recover the system \eqref{mysystem}. This two-type BBM has been studied by Belloum in \cite{belloumpaper} with probabilistic techniques. It is an edge case of the system studied by Belloum and Mallein and later by Ma and Ren in \cite{mallein_belloum, ren_and_ma}, with parameters $\sigma^2=1, \beta=1$ which correspond to the two types of particles having the same diffusivity and branching rate. The point $(\sigma^2,\beta)=(1,1)$ is the intersection of the three boundaries of the phase diagrams of \cite{mallein_belloum} and \cite{holzer_numerical} (see Figure 1 of \cite{mallein_belloum}, Figure 1 of \cite{holzer_numerical}). The median asymptotics obtained in \cite{belloumpaper} for the maximum of this two-type branching Brownian motion agree with the front location predicted by Theorem \ref{levelsets} for the u-component of \eqref{mysystem}. Our result applies in a broader PDE setting: it allows general compact perturbations of the Heaviside initial condition and arbitrary Fisher--KPP nonlinearities, including those outside the McKean class and hence without an underlying branching representation.

\section{PDE heuristics } \label{Section 3}
Let us begin by describing some PDE heuristics based on simple computations that justify the $\frac{1}{2} \ln (t)$ correction that appears in Theorem \ref{levelsets} for system \eqref{mysystem}. 

\subsection{PDE heuristics}

In the region $\left\{x \geq c_{*} t\right\}$, $u,v$ are both small, so the solution $(u,v)$ of system \eqref{mysystem} is close to the solution of the linearized system with zero Dirichlet boundary condition imposed at $c_{*}t$:
$$
\begin{aligned}
&\widetilde{u}_{t}=\widetilde{u}_{x x}+f'(0)\widetilde{u}+\alpha \widetilde{v}, \quad t>0, \quad x > c_{*} t \\
&\widetilde{v}_{t}=\widetilde{v}_{x x}+ f'(0)\widetilde{v}, \quad t>0, \quad x > c_{*}t \\
&\widetilde{u}(t,c_{*}t)=\widetilde{v}(t,c_{*}t)=0 \\
\end{aligned}
$$
In particular, $ \omega(t,x)=e^{\lambda_{*}x}\widetilde{v}\left(t, x+c_{*}t\right)$ solves the heat equation on the positive half line:
$$
\omega_{t}=\omega_{x x}, x>0, \text{ } \omega(t,0)=0.
$$
and $\zeta(t, x)=e^{ \lambda_{*} x} \widetilde{u}\left(t, x+c_{*} t\right),$ solves the forced heat equation on the positive half line:
$$
\zeta_{t}=\zeta_{x x}+\alpha \omega(t,x), x>0, \zeta(t,0)=0.
$$
It can be shown that $\zeta(t, x)$ is of order $1$ at the position $x=O(\sqrt{t})$ as $t \rightarrow+\infty$, which implies that

$$
\widetilde{u}\left(t, c_{*} t+\sqrt{t}\right) \sim  e^{- \lambda_{*} \sqrt{t}}.
$$
Indeed, since $\omega$ solves the heat equation on the positive half-line with a compactly supported initial condition, we have the approximation

\begin{equation}
\begin{split}
&\omega(x, t)=\frac{1}{\sqrt{4 \pi t}} \int_{0}^{\infty}\left[e^{-\frac{(x-y)^{2}}{4 t}}-e^{-\frac{(x+y)^{2}}{ 4 t}}\right] \omega_0(y) d y \approx \frac{e^{-\frac{x^2}{4t}}}{\sqrt{4 \pi t}} \int_{0}^{\infty} e^{-\frac{y^2}{4t}}
\left[e^{\frac{2xy}{4 t}} -e^{-\frac{2xy}{ 4 t}}\right] \omega_0(y) d y \approx \\
& \approx C \frac{e^{-\frac{x^2}{4t}}}{\sqrt{4 \pi t}} \int_{0}^{\infty} e^{-\frac{y^2}{4t}}
\frac{xy}{t} \omega_0(y) d y \approx C \frac{x e^{-\frac{x^2}{4t}}}{t^{\frac{3}{2}}}.
\end{split}
\end{equation}

Using this approximation for $\omega$ we can write for $\zeta$

$$
\zeta(t, x) \approx K \int_{0}^{t} \int_{0}^{\infty} \frac{1}{\sqrt{t-s}}\left(e^{-\frac{|x-y|^{2}}{4 (t-s)}}-e^{-\frac{|x+y|^{2}}{4 (t-s)}}\right) \frac{y e^{-y^{2} / 4 s}}{s^{3 / 2}} d y d s=I_{1}-I_{2} .
$$

Let us write out the exponentials:

$$
\Phi_{1}(x, y)=\frac{|x-y|^{2}}{(t-s)}+\frac{y^{2}}{s}= \frac{|x|^{2}}{(t-s)}-\frac{|x|^{2}}{(t-s)^{2} \varphi(t, s)}+\varphi(t, s)\left(y-Y_{1}(t, s, x)\right)^{2}
$$

where we set $
\varphi(t, s)=\frac{1}{(t-s)}+\frac{1}{s}=\frac{t}{s(t-s)}
$ and  $Y_{1}(t, s, x)=\frac{x}{(t-s)\varphi(t, s)}$. \\

Thus, $I_{1}$ can be written as

\small
\begin{equation}
\begin{split}
I_{1} & =K \int_{0}^{t} \int_{0}^{\infty} \frac{1}{s^{3 / 2} \sqrt{t-s}} y e^{-\frac{\Phi_{1}(x, y)}{4}} d y d s \\
& =K \int_{0}^{t} \frac{\exp \left(-\frac{|x|^{2}}{4 (t-s)}+\frac{|x|^{2}}{4(t-s)^{2} \varphi(t, s)}\right) }{s^{3 / 2} \sqrt{t-s}} d s \int_{0}^{\infty} y \exp \left(-\frac{\varphi(t, s)\left(y-Y_{1}(t, s, x)\right)^{2}}{4}\right) d y \\
&= K \int_{0}^{t} \frac{\exp \left(-\frac{|x|^{2}}{4 (t-s)}+\frac{|x|^{2}}{4(t-s)^{2} \varphi(t, s)}\right)}{s^{3 / 2} \sqrt{t-s}}  d s \int_{-Y_{1}(t, s, x)}^{\infty}(y+Y_{1}(t, s, x)) \exp \left(-\frac{\varphi(t, s) y^{2}}{4}\right) d y \\
& =I_{11}+I_{12} .
\end{split}
\end{equation}

We analyze $I_{11}$ and $I_{12}$ further

\begin{equation}
\begin{split}
I_{11} & = K \int_{0}^{t} \frac{1}{s^{3 / 2} \sqrt{t-s}} \exp \left(-\frac{|x|^{2}}{4 (t-s)}+\frac{|x|^{2}}{4(t-s)^{2} \varphi(t, s)}\right)  \int_{-\sqrt{\varphi(t, s)} Y_{1}(t, s, x)}^{\infty} y \exp \left(-\frac{y^{2}}{4}\right) \frac{d y d s}{\varphi(t, s)} \\
& =2 K \int_{0}^{t} \frac{1}{s^{3 / 2} \sqrt{t-s}} \exp \left(-\frac{|x|^{2}}{4 (t-s)}+\frac{|x|^{2}}{4 (t-s)^{2} \varphi(t, s)}\right) \frac{1}{\varphi(t, s)} \exp \left(-\frac{\varphi(t,s) Y_{1}^{2}(t, s, x)}{4}\right) d s \\
& =2 K \int_{0}^{t} \frac{1}{s^{3 / 2} \varphi(t, s) \sqrt{t-s}} \exp \left(-\frac{|x|^{2}}{4 (t-s)}\right) d s .
\end{split}
\end{equation}

\begin{equation}
\begin{split}
I_{12} & =K \int_{0}^{t} \frac{ Y_{1}(t, s, x)}{s^{3 / 2} \sqrt{t-s}} \exp \left(-\frac{|x|^{2}}{ 4(t-s)}+\frac{\varphi(t, s) Y_{1}^{2}(t, s, x)}{4}\right) \int_{-Y_{1}(t, s, x)}^{\infty} \exp \left(-\frac{\varphi(t, s) y^{2}}{4}\right) d y d s \\
& =K \int_{0}^{t} \frac{ Y_{1}(t, s, x)}{s^{3 / 2} \sqrt{(t-s) \varphi(t, s)}} \exp \left(-\frac{|x|^{2}}{4 (t-s)}\right) \mathcal{E}\left(-\sqrt{\varphi(t, s)} Y_{1}(t, s, x)\right) d s .
\end{split}
\end{equation}

Here, we have set $
\mathcal{E}(y)=e^{y^{2} / 4} \int_{y}^{\infty} e^{-z^{2} / 4} d z.
$

Next, let us look at the term $I_{2}$. The corresponding phase is

$$
\Phi_{2}(x, y)=\frac{|x+y|^{2}}{(t-s)}+\frac{y^{2}}{s}= \frac{|x|^{2}}{(t-s)}-\frac{|x|^{2}}{(t-s)^{2} \varphi(t, s)}+\varphi(t, s)\left(y-Y_{2}(t, s, x)\right)^{2} .
$$
with $Y_{2}(t, s, x)=\frac{-x}{(t-s) \varphi(t, s)}.$ Thus,

\begin{equation}
\begin{split}
I_{2} & = K \int_{0}^{t} \frac{\exp \left(-\frac{|x|^{2}}{4 (t-s)}+\frac{|x|^{2}}{4 (t-s)^{2} \varphi(t, s)}\right)}{s^{3 / 2} \sqrt{t-s}} d s \int_{-Y_{2}(t, s, x)}^{\infty}\left(y+Y_{2}(t, s, x)\right) \exp \left(-\frac{\varphi(t, s) y^{2}}{4}\right) d y \\
& =I_{21}+I_{22}
\end{split}
\end{equation}

We analyze $I_{21}$ and $I_{22}$

\begin{equation}
\begin{split}
I_{21} & =2 K \int_{0}^{t} \frac{1}{s^{3 / 2} \varphi(t, s) \sqrt{t-s}} \exp \left(-\frac{|x|^{2}}{4(t-s)}\right) d s .
\end{split}
\end{equation}

\begin{equation}
\begin{split}
I_{22} & = K \int_{0}^{t} \frac{ Y_{2}(t, s, x)}{s^{3 / 2} \sqrt{(t-s) \varphi(t, s)}} \exp \left(-\frac{|x|^{2}}{4 (t-s)}\right) \mathcal{E}\left(-\sqrt{\varphi(t, s)} Y_{2}(t, s, x)\right)
\end{split}
\end{equation}

We note that $ I_{11}=I_{21} $, so these two terms cancel exactly. Therefore, going back to the expression for $\zeta(t,x)$:

\begin{equation} \label{zetaapprox}
\begin{split}
& \zeta(t, x)=I_{12}-I_{22}=K \int_{0}^{t} \frac{ Y_{1}(t, s, x)}{s^{3 / 2} \sqrt{(t-s) \varphi(t, s)}} \exp \left(-\frac{|x|^{2}}{4 (t-s)}\right) \mathcal{E}\left(-\sqrt{\varphi(t, s)} Y_{1}(t, s, x)\right) d s \\
& -K \int_{0}^{t} \frac{ Y_{2}(t, s, x)}{s^{3 / 2} \sqrt{(t-s) \varphi(t, s)}} \exp \left(-\frac{|x|^{2}}{4 (t-s)}\right) \mathcal{E}\left(-\sqrt{\varphi(t, s)} Y_{2}(t, s, x)\right) \\
& =K \int_{0}^{t} \frac{1}{s^{3 / 2} \varphi(t, s) \sqrt{(t-s)}} \exp \left(-\frac{|x|^{2}}{4 (t-s)}\right)\left(\mathcal{G}\left(\sqrt{\varphi(t, s)} Y_{1}(t, s, x)\right)-\mathcal{G}\left(\sqrt{\varphi(t, s)} Y_{2}(t, s, x)\right)\right) d s .
\end{split}
\end{equation}

Here, we have defined $
\mathcal{G}(y)=y \mathcal{E}(-y)=y e^{y^{2} / 4} \int_{-y}^{\infty} e^{-z^{2} / 4} d z.
$ We note that 
\begin{equation} \nonumber
\begin{split}
& \sqrt{\varphi(t, s)} Y_{1}(t, s, x)-\sqrt{\varphi(t, s)} Y_{2}(t, s, x)=\frac{2 \sqrt{s}}{\sqrt{(t-s)t}} x .
\end{split}
\end{equation}

For $x \sim O(1)$, we approximate

$$
\mathcal{G}\left(\sqrt{\varphi(t, s)} Y_{1}(t, s, x)\right)-\mathcal{G}\left(\sqrt{\varphi(t, s)} Y_{2}(t, s, x)\right) \approx \mathcal{G}^{\prime}\left(\sqrt{\varphi(t, s)} Y_{1}(t, s, x)\right) \frac{2 x}{(t-s) \sqrt{\varphi(t, s)}}
$$
and for $
\sqrt{\varphi(t, s)} Y_{1}(t, s, x)=\frac{x}{(t-s) \sqrt{\varphi(t, s)}} = \frac{\sqrt{s}}{\sqrt{(t-s)t}} x
$, we approximate $
\mathcal{G}^{\prime}(y) \approx C e^{\frac{y^{2}}{4}}.
$

Using this approximation in $\ref{zetaapprox}$ gives us

\begin{equation} \label{finalzetaapproximation}
\begin{split}
\zeta(t, x) & =K \int_{0}^{t} \frac{1}{s^{3 / 2} \varphi(t, s) \sqrt{(t-s)}} \exp \left(-\frac{|x|^{2}}{4 (t-s)}\right)\left(\mathcal{G}\left(\sqrt{\varphi(t, s)} Y_{1}(t, s, x)\right)-\mathcal{G}\left(\sqrt{\varphi(t, s)} Y_{2}(t, s, x)\right)\right) d s \\
& \approx C K \int_{0}^{t} \frac{1}{s^{3 / 2} \varphi(t, s) \sqrt{(t-s)}} \exp \left(-\frac{|x|^{2}}{4(t-s)}\right) \frac{2 x}{(t-s) \sqrt{\varphi(t, s)}}  \exp \left(\frac{\varphi(t, s) Y_{1}^{2}(t, s, x)}{4}\right) d s  \\
&= C K \int_{0}^{t} \frac{s^{3 /2} (t-s)^{ 3 /2 }}{s^{3 / 2} t^{3 / 2} } \exp \left(-\frac{|x|^{2}}{4(t-s)}\right) \frac{2 x}{(t-s)^{3/2} } \exp \left(\frac{|x|^2 s}{4t(t-s)}\right) d s \\
&= C K \int_{0}^{t} \frac{2x}{ t^{3 / 2} } \exp \left(-\frac{|x|^{2}}{4(t-s)}\right) \exp \left(\frac{|x|^2 s}{4t(t-s)}\right) d s \\
&=C K \frac{2x}{ t^{1 / 2} }  \exp \left(-\frac{|x|^{2}}{4t}\right) 
\end{split}
\end{equation}

From $\ref{finalzetaapproximation}$ we see that $\zeta$ is of order $1$ at the position $x=O (\sqrt{t})$. Thus 
$
\widetilde{u}\left(t, c_{*} t+\sqrt{t}\right) \sim  e^{- \lambda_{*} \sqrt{t}}.
$

To match this value, the front must be shifted to the position $m(t)= c_{*}t-\frac{1}{2 \lambda_{*}} \ln t$,
\small
$$
\left.U_{c_{*}}\left(x-c_{*} t+\frac{1}{2 \lambda_{*}} \ln t \right)\right|_{x=c_{*} t+\sqrt{t}} \sim \left(\sqrt{t}+\frac{1}{2 \lambda_{*}} \ln t\right) e^{-\lambda_{*}\left(\sqrt{t}+\frac{1}{2 \lambda_{*}}\ln t \right)} \sim  e^{- \lambda_{*} \sqrt{t}}
$$
\normalsize
\newline
which provides the heuristic estimate $c_{*} t-\frac{1}{2 \lambda_{*}} \ln t$ as a lower bound of the location of $u$.


\subsection{The linearized Dirichlet problem in the log shifted frames}

Inspired by the work on the classic Fisher-KPP equation, we expect that the solution to the nonlinear system $\eqref{mysystem}$ will behave as the solution of the linearized system with Dirichlet boundary conditions imposed at appropriately chosen logarithmic frames $\xi(t)$, with $t_0>0$ and $r$

\begin{equation}
\begin{split}
  & \xi_{u}(t)=c_{*}t-\frac{r}{\lambda_{*}} \ln \left(t+t_{0}\right), \xi_{v}(t)=c_{*}t-\frac{3}{2 \lambda_{*}} \ln \left(t+t_{0}\right) \\
\end{split}
\end{equation}

The discussion that follows explains why we need to choose $r=\frac{1}{2}.$

We do the change of variables $u(t,x)=\widetilde{u}(t, x-\xi_{u}(t)), v(t,x)=\widetilde{v}(t,x-\xi_{v}(t))$

\begin{equation}
\begin{split}
\widetilde{u}_{t}-c_{*}\widetilde{u}_{x} + \frac{r}{\lambda_{*}(t+t_{0})}\widetilde{u}_{x} & =\widetilde{u}_{xx}+f(\widetilde{u})+\alpha \widetilde{v}(t,x+\frac{\left(\frac{3}{2}-r \right)}{\lambda_{*}} \ln (t+t_0))(1-\widetilde{u}), \\
\widetilde{v}_{t}-c_{*} \widetilde{v}_{x} + \frac{3}{2 \lambda_{*} (t+t_{0})} \widetilde{v}_{x} & =\widetilde{v}_{xx}+f(\widetilde{v}) .
\end{split}
\end{equation}

The linearized Dirichlet problem in the logarithmically shifted frames is

\begin{equation} \label{linearizeddirichletinlog}
\begin{split}
&U_{t}-c_{*} U_{x} + \frac{r}{\lambda_{*}(t+t_{0})}U_{x}=U_{xx}+f'(0)U+\alpha V(t,x+\frac{\left(\frac{3}{2}-r \right)}{\lambda_{*}} \ln (t+t_0)), \\
&V_{t}-c_{*} V_{x} + \frac{3}{2\lambda_{*}(t+t_{0})} V_{x}=V_{xx}+f'(0)V . \\
& U(t,0)=0, V(t,0)=0 
\end{split}
\end{equation}

We do another change of variables to get rid of the exponential factor

\begin{equation}
\begin{split}
& U(t, x)=e^{-\lambda_{*}x} z(t, x), V(t, x)=e^{-\lambda_{*} x} w(t, x) . \\
\end{split}
\end{equation}

Using the dispersion relation $c_{*} \lambda_{*}= (\lambda_{*})^2 +f'(0)$ and $c_{*}=2 \lambda_{*}$, we get the equations for $z,w$

\begin{equation}
\begin{split}
&z_{t}+\frac{r}{\lambda_{*}(t+t_{0})} (z_{x}-\lambda_{*}z )=z_{x x}+\frac{\alpha }{(t+t_0)^{\frac{3}{2}-r} }w(t,x+\frac{\left(\frac{3}{2}-r\right)}{\lambda_{*}}\ln (t+t_0)) \\
&w_{t}+\frac{3}{2 \lambda_{*} (t+t_{0})}(w_{x}-\lambda_{*}w)=w_{x x} \\
&z(t,0)=0, w(t,0)=0.
\end{split}
\end{equation}

We introduce  the self-similar variables $\tau=\ln \left(t+t_{0}\right)-\ln t_{0}, \eta=\frac{x}{\sqrt{t+t_{0}}}$ and set 

\begin{equation}
\begin{split}
&z(t,x)=p(\tau,\eta)=p\Big(\ln \left(t+t_{0}\right)-\ln t_{0},\frac{x}{\sqrt{t+t_{0}}}\Big), w(t,x)=q(\tau,\eta)=q\Big(\ln \left(t+t_{0}\right)-\ln t_{0},\frac{x}{\sqrt{t+t_{0}}}\Big).
\end{split}
\end{equation}

We obtain the equations for $p,q$

\begin{equation} \label{seeheuristicwithr}
\begin{split}
&p_{\tau}-\frac{\eta}{2}p_{\eta}+\frac{r}{\lambda_{*}}\left( \frac{e^{-\frac{\tau}{2}}}{t_{0}^{\frac{1}{2}}} p_{\eta}-\lambda_{*}p \right)=p_{\eta \eta}+\frac{\alpha}{(e^{\tau}t_0)^{\frac{1}{2}-r}} q\left(\tau, \eta+ \frac{\left(\frac{3}{2}-r\right)}{\lambda_{*}}\frac{\tau+\ln(t_0)}{\sqrt{t_0 e^{\tau}}}\right)
\\
&q_{\tau}-\frac{\eta}{2}q_{\eta}+\frac{3}{2 \lambda_{*}} \left( \frac{e^{-\frac{\tau}{2}}}{t_{0}^{\frac{1}{2}}} q_{\eta} -\lambda_{*} q \right)=q_{\eta \eta}
\end{split}
\end{equation}

From Lemma 2.2 in \cite{short_proof} we have for $q(\tau,\eta)$, with $\varepsilon=\frac{3}{2\lambda_{*}t_0^{\frac{1}{2}}}$

$$
q(\tau, \eta)=e^{\frac{\tau}{2}}\eta \left(\frac{e^{-\frac{\eta^{2}}{4}}}{2 \sqrt{\pi}} \int_{0}^{+\infty} \xi v_{0}(\xi) d \xi+O(\varepsilon)+O\left(e^{-\frac{\tau}{2}}\right)\right)
$$

Transforming into the original variables gives us

$$
V(t, x)=\frac{x e^{-\lambda_{*}x}}{\sqrt{t_0}}\left[C e^{-\frac{x^{2}}{4\left(t+t_{0}\right)}}+h(t, x)\right]
$$

so $V(t,x)$ remains bounded from above and below away from zero on the interval $1 \leq x \leq 2$.

We can now see from $\eqref{seeheuristicwithr}$ that we need to choose $r=\frac{1}{2}$ to make the power of $e^{\tau}t_0$ in the denominator of the forcing vanish
and have the $q$ forcing from the right side balance with the extra $p$ term from the left side of the equation. Then $p$ has the same asymptotics as $q$, and 
$U(t,x)$ is also bounded from above and from below away from zero on the interval $1 \leq x \leq 2$.

\section{Long-time front asymptotics} \label{section 4}
\subsection{An upper bound for the location of the front of the cascading system}
We begin by proving Theorem \ref{cascadingupperbound}.
As suggested by the heuristics of the previous section, we move each $v^{i}$ of the cascading system \eqref{cascadinggeneralf} into the moving frame $$\xi_{i}(t)=c_{*} t-\frac{\left(\frac{3}{2}+i-k\right)}{\lambda_{*}}\ln(t+t_0), \text{ setting }  v^{i}(t,x)=\widetilde{v^{i}}(t,x-\xi_{i}(t))$$

\begin{equation} \label{logmovingframecascading}
\begin{split}
&(\widetilde{v^{i}})_{t}-c_{*}(\widetilde{v^{i}})_{x}+\frac{\left( \frac{3}{2}+i-k\right)}{\lambda_{*} (t+t_0)} (\widetilde{v^{i}})_{x} = (\widetilde{v^{i}})_{xx}+f(\widetilde{v^{i}}) +\alpha \widetilde{v^{i+1}}(t,x+\frac{\ln(t+t_0)}{\lambda_{*}}) \left(1-\widetilde{v^{i}}\right), 1 \leq i \leq k-1 \\
& (\widetilde{v^{k}})_{t}-c_{*}(\widetilde{v^{k}})_{x}+\frac{\frac{3}{2}}{\lambda_{*}(t+t_0)}(\widetilde{v^{k}})_{x}= (\widetilde{v^{k}})_{xx}+f(\widetilde{v^{k}}).\\
\end{split}
\end{equation}

We consider the linearized Dirichlet problem with zero boundary conditions imposed at those logarithmically shifted frames.

\begin{equation} \label{cascadinglinearizeddirichletlogframes}
\begin{split}
&V^{i}_{t}-c_{*}V^{i}_{x}+\frac{\left( \frac{3}{2}+i-k\right)}{\lambda_{*}(t+t_0)} V^{i}_{x}= V^{i}_{xx}+f'(0)V^{i}+\alpha V^{i+1}(t,x+\frac{\ln(t+t_0)}{\lambda_{*}})  \text{ for } 1 \leq i \leq k-1, x >0 \\
& V^{k}_{t}-c_{*}V^{k}_{x}+\frac{\frac{3}{2}}{\lambda_{*}(t+t_0)} V^{k}_{x}= V^{k}_{xx}+f'(0)V^{k}, x >0 \\
&V^{i}(t,0)=0 \text{ for } 1 \leq i \leq k.\\
\end{split}
\end{equation} 

We will need the following Proposition.

\begin{proposition}\label{asymptoticslinearlemma}
    
Let $(V^{i})_{i=1}^{k}$ be the solution to the linearized Dirichlet problem in the logarithmically shifted reference frames $\eqref{cascadinglinearizeddirichletlogframes}$ with zero boundary conditions and compactly supported initial data on $(0,+\infty)$ with $V^{i}_{0} \not \equiv 0$. There exists a constant $C>0$ that depends on the initial conditions  and a constant $t_{0}>0$ that depends on $C$ such that

$$
V^{i}(t, x)=\frac{x e^{-\lambda_{*}x}}{\sqrt{t_{0}}}\left[C e^{-\frac{x^{2}}{4\left(t+t_{0}\right)}}+h^{i}(t, x)\right]
$$

where, for each $\sigma>0$,

$$
\limsup _{t \rightarrow+\infty} \sup _{0 \leq x \leq \sigma \sqrt{t+1}}|h^{i}(t, x)|<\frac{C}{2}.
$$
\end{proposition} 

Proposition \ref{asymptoticslinearlemma} provides us with precise estimates for the linearized Dirichlet problem in the log shifted frames \eqref{cascadinglinearizeddirichletlogframes}. Let us postpone its proof for a moment and see how we can use it to construct a supersolution for the $v^{i}$ component of the nonlinear system \eqref{cascadinggeneralf} and obtain an upper bound for the location of the front of $v^{i}$.

In the log-shifted frames our nonlinear cascading system is

\begin{equation} \label{logmovingframecascadingagain}
\begin{split}
&(\widetilde{v^{i}})_{t}-c_{*}(\widetilde{v^{i}})_{x}+\frac{\left( \frac{3}{2}+i-k\right)}{\lambda_{*} (t+t_0)} (\widetilde{v^{i}})_{x} = (\widetilde{v^{i}})_{xx}+f(\widetilde{v^{i}}) +\alpha \widetilde{v^{i+1}}(t,x+\frac{\ln(t+t_0)}{\lambda_{*}}) \left(1-\widetilde{v^{i}}\right), 1 \leq i \leq k-1 \\
& (\widetilde{v^{k}})_{t}-c_{*}(\widetilde{v^{k}})_{x}+\frac{\frac{3}{2}}{\lambda_{*}(t+t_0)}(\widetilde{v^{k}})_{x}= (\widetilde{v^{k}})_{xx}+f(\widetilde{v^{k}})\\
\end{split}
\end{equation}

We consider the solution to the linearized system in the logarithmically shifted frames with zero Dirichlet boundary conditions and initial conditions equal to the indicator function on the interval $[0,2A]$, where we take $A>x_2$. Recall that for the initial conditions we have $\widetilde{v^{i}}(0,x)=0$ for $x>A$ and $\widetilde{v^{i}}(0,x) \leq 1$ for all $x>0$, so that $\widetilde{v^{i}}(0,x) \leq V^{i}(0,x)=\mathbbm{1}_{[0,2A]}$ for $x \geq 0$.

\begin{equation} \label{cascadinglinearizeddirichletlogframesagain}
\begin{split}
&V^{i}_{t}-c_{*}V^{i}_{x}+\frac{\left( \frac{3}{2}+i-k\right)}{\lambda_{*}(t+t_0)} V^{i}_{x}= V^{i}_{xx}+f'(0)V^{i}+\alpha V^{i+1}(t,x+\frac{\ln(t+t_0)}{\lambda_{*}})  \text{ for } 1 \leq i \leq k-1, x>0 \\
& V^{k}_{t}-c_{*}V^{k}_{x}+\frac{\frac{3}{2}}{\lambda_{*}(t+t_0)} V^{k}_{x}= V^{k}_{xx}+f'(0)V^{k}, x>0\\
&V^{i}(t,0)=0 \text{ for }  1 \leq i \leq k .\\
&V^{i}(0,x)= \mathbbm{1}_{[0,2A]} \text{ for }  1 \leq i \leq k .
\end{split}
\end{equation} 

We will suppose that there exist $B>0, t_0>0$ large enough such that for $x \geq 0, t > 0$
\begin{equation} \label{condition}
BV^{i+1}(t, x+ \frac{\ln (t+t_0)}{\lambda_{*}})\geq \widetilde{v^{i+1}}(t, x+\frac{\ln(t+t_0)}{\lambda_{*}})
\end{equation}

and show that then this holds for $i$, i.e. 

\begin{equation} \label{condition}
BV^{i}(t, x+ \frac{\ln (t+t_0)}{\lambda_{*}})\geq \widetilde{v^{i}}(t, x+\frac{\ln(t+t_0)}{\lambda_{*}}).
\end{equation}

We set 
$$
\begin{aligned}
&N(w)=w_{t}-c_{*}w_{x}+ \frac{\left( \frac{3}{2}+i-k\right)}{\lambda_{*}(t+t_0)}w_{x}-w_{xx}-f(w)-\alpha \widetilde{v^{i+1}}(t,x+\frac{\ln(t+t_0)}{\lambda_{*}}) \\
&+\alpha \widetilde{v^{i+1}}(t, x+ \frac{\ln(t+t_0)}{\lambda_{*}})w
\end{aligned}
$$

A function $w(t,x)$ is a supersolution for \eqref{logmovingframecascadingagain} if $N(w) \geq 0$.

We have that any multiple $\{B V^{i}\}_{i=1}^{k}$, (where $B>0$ is to be chosen later to be large enough), solves the linearized system \eqref{cascadinglinearizeddirichletlogframesagain} and thus

$$
\begin{aligned}
&N(BV^{i})=B (V^{i})_{t}-c_{*}B ({V^{i}})_{x}+\frac{\left( \frac{3}{2}+i-k\right)}{\lambda_{*}(t+t_0)} B (V^{i})_{x} -B (V^{i})_{xx}-f(B V^{i} ) \\
&-\alpha \widetilde{v^{i+1}}(t,x+\frac{\ln (t+t_0)}{\lambda_{*}})+\alpha \widetilde{v^{i+1}}(t,x+\frac{\ln(t+t_0)}{\lambda_{*}}) B V^{i} = f'(0)B V^{i} -f( B V^{i}) \\
&+\alpha \left(B V^{i+1}(t, x+ \frac{\ln (t+t_0)}{\lambda_{*}})-\widetilde{v^{i+1}}(t, x+\frac{\ln(t+t_0)}{\lambda_{*}})\right)+\alpha \widetilde{v^{i+1}}(t,x+\frac{\ln(t+t_0)}{\lambda_{*}}) B V^{i}
\end{aligned}
$$

Since our $f(u)$ is a Fisher-KPP type nonlinearity, we have that
$$
f(B V^{i}) \leq f'(0) B V^{i}
$$
We have that then $N(B V^{i})\geq 0$ as long as 
\begin{equation} \label{condition}
BV^{i+1}(t, x+ \frac{\ln (t+t_0)}{\lambda_{*}})\geq \widetilde{v^{i+1}}(t, x+\frac{\ln(t+t_0)}{\lambda_{*}})
\end{equation}

We have supposed that there exist $B,t_0$ to have this be true for $i+1$. Thus we get $N(BV^{i})>0$  and thus $BV^{i}$ is a supersolution for $\widetilde{v^{i}}$ for $x>0$. However, we can't compare immediately $\widetilde{v^{i}}(t,x)$ with $BV^{i}(t,x)$ because we have $BV^{i}(t,0)=0$ while $\widetilde{v^{i}}(t,0)>0$. In Lemma \eqref{asymptoticslinearlemma} we showed that there exists a constant $C>0$ that depends on the initial conditions such that
$$
V^{i}(t, x)=\frac{1}{\sqrt{t_0}} x e^{-\lambda_{*} x}\left[C e^{-\frac{x^{2}}{ 4\left(t+t_{0}\right)}}+h^{i}(t, x)\right]
$$
where, for each $\sigma>0$,

$$
\limsup _{t \rightarrow+\infty} \sup _{0 \leq x \leq \sigma \sqrt{t+1}}|h^{i}(t, x)|<\frac{C}{2}.
$$

Thus, we can choose $t_0$ sufficiently large and then choose $B$
sufficiently large (depending on $t_0$, for instance $B=M\sqrt{t_0}$ for large $M$) such that
$
BV^{i}(t,A) \geq 1 \text{ for all } t>0.$
Then $BV^{i}$ intersects $1$ for a second time at a point $x(t)$ whose
location is uniformly bounded in time, say by $D$. By enlarging $t_0$
further if necessary and choosing $B$ accordingly, we may also ensure that
$
\frac{\ln(t_0)}{\lambda_{*}} \geq D.$ We define for $x \in \mathbb{R}$

$$
\overline{V^{i}}(t, x)=\left\{\begin{array}{cl}
1, & \text { if } x \leq A \\
\min \left(1, B V^{i}(t, x)\right), & \text { if } x \geq A
\end{array}\right.
$$

We get that $\overline{V^{i}(t,x)} \geq \widetilde{v^{i}}(t,x)$. In particular, since $\overline{V^{i}(t,x)}=BV^{i}(t,x)$ for $x>D$, we get $\overline{V^{i}}(t,x +\frac{\ln(t+t_0)}{\lambda_{*}})=B V^{i}(t,x+ \frac{\ln(t+t_0)}{\lambda_{*}})$ for $x>0, t>0$.
This lets us conclude the desired inequality for $i$, i.e. we have for $x>0, t>0$

\begin{equation} \label{condition}
BV^{i}(t, x+ \frac{\ln (t+t_0)}{\lambda_{*}})\geq \widetilde{v^{i}}(t, x+\frac{\ln(t+t_0)}{\lambda_{*}}).
\end{equation}

As shown in \cite{short_proof}, this inequality holds for $V^{k}$ and $\widetilde{v^{k}}$. By induction, we can conclude that this holds for all $1 \leq i \leq k$.

Having this inequality hold for $i+1$ ensures that $N(BV^{i})>0$, and that $\widetilde{v^{i}}(t,x) \leq \overline{V^{i}}(t,x)$, for $t>0, x \in \mathbb{R}$, which gives us 

$$
v^{i}\left(t,x+c_{*}t-\frac{\left(\frac{3}{2}-k+i \right)}{\lambda_{*}} \ln(1+\frac{t}{t_0})\right) \leq \overline{V^{i}}(t,x)
$$

for all $t>0, x \in \mathbb{R}$ which implies that

\begin{equation} \label{limsupmaxcascading}
\limsup _{t \rightarrow+\infty}\left(\max _{x \geq c_{*}t-\frac{\left(\frac{3}{2}-k+i \right)}{\lambda_{*}}\ln t +y} v^{i}(t, x)\right) \rightarrow 0 \quad \text { as } y \rightarrow+\infty \text {. }
\end{equation}

and for every $\sigma>0$, there is a positive constant $\rho>0$ such that

$$
v^{i}\left(t, c_{*}t-\frac{\left(\frac{3}{2}-k+i \right)}{\lambda_{*}}\ln t +y \right) \leq \rho(y+1) e^{- \lambda_{*} y} \quad \text { for all } t \geq 1 \text { and } 0 \leq y \leq \sigma \sqrt{t}.
$$

Then for any $m \in(0,1)$, for the level set of $v^{i}$, $E^{i}_{m}(t)=\{x>0, v^{i}(t, x)=m\}$ there exist constants $t_{0}>0$ and $C \geq 0$ such that

\begin{equation} \label{maxlevelsetcascading}
\max E^{i}_{m}(t) \leq c_{*}t-\frac{\left(\frac{3}{2}-k+i \right)}{\lambda_{*}}\ln t +C \text { for all } t \geq t_{0}.
\end{equation}

This inequality $\eqref{maxlevelsetcascading}$ follows directly from $\eqref{limsupmaxcascading}$ and provides us with the desired upper bound for the front of the $v^{i}$ component of the cascading system $\eqref{cascadingsystemequations}$. This completes the proof of Theorem \ref{cascadingupperbound}.

We now prove Proposition \ref{asymptoticslinearlemma}. 

\subsection{Proof of Proposition \ref{asymptoticslinearlemma}}
We begin by a standard chain of changes of variables. We set $
V^{i}(t, x)=e^{-\lambda_{*}x} z^{i}(t, x)$
and obtain the system for $z^{i}(t,x)$ using the dispersion relation $c_{*} \lambda_{*}= (\lambda_{*})^2 +f'(0)$ and $c_{*}=2 \lambda_{*}$

\begin{equation}
\begin{split}
&z^{i}_{t}+\frac{\left( \frac{3}{2}+i-k\right)}{\lambda_{*}(t+t_0)}(z^{i}_{x}-\lambda_{*}z^{i})= z^{i}_{xx}+\frac{\alpha}{(t+t_0)} z^{i+1}(t,x+\frac{\ln(t+t_0)}{\lambda_{*}}) \text{ for }  1 \leq i \leq k-1, x >0 \\
& z^{k}_{t}+\frac{\frac{3}{2}}{\lambda_{*}(t+t_0)}(z^{k}_{x}-\lambda_{*}z^{k})= z^{k}_{xx}, x>0 \\
&z^{i}(t,0)=0 \text{ for }  i \in \{1,\ldots,k\}.  \\
\end{split}
\end{equation}

We then go into self-similar variables, setting

$$
z^{i}(t,x)=p^{i}(\ln\left(t+t_{0}\right)-\ln t_{0},\frac{x}{\sqrt{t+t_{0}}})=p^{i}(\tau,\eta) \text{ for } i \in \{1,\ldots,k\}
$$

and our equations become for $
\mathcal{L}v=-v_{\eta \eta}-\frac{\eta}{2} v_{\eta}-v
$

\begin{equation}
\begin{split}
&p^{i}_{\tau}+\mathcal{L}p^{i} =\left(\frac{3}{2}+i-1-k\right)p^{i}-\varepsilon\left(\frac{3}{2}+i-k\right)e^{-\frac{\tau}{2}}p^{i}_{\eta} +\alpha p^{i+1}(\tau,\eta+\delta(\tau)), \eta>0 \\
& p^{k}_{\tau}+\mathcal{L} p^{k}= \frac{p^{k}}{2} -\frac{3}{2}\varepsilon e^{-\frac{\tau}{2}}p^{k}_{\eta}, \eta>0 \\
&p^{i}(\tau,0)=0 \text{ for }  i \in \{1,\ldots,k\}  \\
\end{split}
\end{equation} 

where we set $\delta(\tau)=\frac{\tau+\ln (t_0)}{\lambda_{*}\sqrt{t_0 e^{\tau}}}$ for the shift factor and note that $\delta(\tau) \rightarrow 0$ as  $\tau \rightarrow \infty $ exponentially fast. We also set $\varepsilon=\frac{1}{\lambda_{*}\sqrt{t_0}}$, which will be taken small later, by choosing big enough $t_0$. The corresponding term will be treated as a perturbation. We do one more change of variables to symmetrize the operator $\mathcal{L}.$

\begin{equation}
p^{i}(\tau,\eta)=w^{i}(\tau,\eta) e^{\frac{-{\eta}^2}{8}} e^{\frac{\tau}{2}} \text{ for } i\in \{1,\ldots,k\}
\end{equation}

so that our system becomes for $
M w=-w_{\eta \eta}+\left(\frac{{\eta}^{2}}{16}-\frac{3}{4}\right) w,
$

\begin{equation} \label{cascadingMoperator}
\begin{split}
&w^{i}_{\tau} +M w^{i} +(k-i)w^{i}=-\left(\frac{3}{2}+i-k \right)\varepsilon e^{-\frac{\tau}{2}}\left(w_{\eta}^{i}-\frac{\eta}{4}w^{i}\right)+\alpha w^{i+1}(\tau,\eta+\delta(\tau))e^{-\frac{\delta(\tau)^2}{8} }e^{-\frac{\eta \delta(\tau)}{4}}, \eta>0 \\
& w^{k}_{\tau}+M w^{k}= -\frac{3}{2} \varepsilon e^{-\frac{\tau}{2}}\left(w^{k}_{\eta}-\frac{\eta}{4}w^{k}\right), \eta>0 \\
&w^{i}(\tau,0)=0 \text{ for }  i \in \{1,\ldots,k\}.  \\
\end{split}
\end{equation}

We introduce the quadratic form

$$
Q(w)=\langle M w, w\rangle_{L^{2}((0,+\infty))}=\int_{0}^{+\infty}\left(w_{\eta}^{2}+\left(\frac{{\eta}^{2}}{16}-\frac{3}{4}\right) w^{2}\right) d \eta
$$
in the space $H_{0}^{1}((0,+\infty))$ with $\eta w \in L^{2}((0,+\infty))$. The operator $M$ is self-adjoint with non-negative eigenvalues $\lambda_k=k, k \in \mathbb{N}$,  and the principal eigenfunction corresponding to the zero eigenvalue is $e_{0}(\eta)=\eta e^{-{\eta}^{2} / 8} /(2 \sqrt{\pi})^{1 / 2}$. Moreover, the spectral gap gives us
$
Q(w) \geq\|w\|_{L^{2}((0,+\infty))}^{2} \text { in } e_{0}^{\perp}.
$

The Lemmas that follow will let us establish the desired asymptotics for $V^{i}$ of Proposition 
\ref{asymptoticslinearlemma}.
\begin{lemma} \label{boundedl2normlemma}
There exists a constant $C>0$ such that 
 $\|w^{i}\|_{L^{2}} \leq C $ for all $i \in \{1,\ldots,k\}$.
\end{lemma}

\begin{proof}

Take $i\in \{1,\cdots,k-1\}$. 
Suppose that there exists a constant $C>0$ such that $\|w^{i+1}(\tau)\|_{L^{2}} \leq C$.
We will show that then there exists a constant $C>0$ such that $\|w^{i}(\tau)\|_{L^{2}} \leq C$.

Multiplying the equation \eqref{cascadingMoperator} by $w^{i}$ and integrating gives us

\small
\begin{equation}
\begin{split}
& \frac{d}{d \tau}\|w^{i}\|_{L^{2}}^{2}+2 Q(w^{i})+ 2(k-i)\|w^{i}\|_{L^{2}}^{2} =2 \varepsilon \left( \frac{3}{2}+i-k \right) e^{-\frac{\tau}{2}} \int_{0}^{+\infty} \frac{\eta}{4} (w^{i})^{2} d \eta \\
& +2 \alpha e^{-\frac{\delta(\tau)^{2}}{8} } \int_{0}^{+\infty} w^{i}w^{i+1}(\tau,\eta+\delta(\tau))  e^{-\frac{\eta \delta(\tau)}{4}}  d \eta \\
\end{split}
\end{equation}
\normalsize

We first work out the case $i=k-1$, which is identical to the system \eqref{mysystem}. We have $\frac{3}{2}+i-k = \frac{1}{2}>0 $. 

\small
\begin{equation}
\begin{split}
& \int_{0}^{+\infty} \frac{\eta}{4} (w^{k-1})^{2} d \eta= \int_{0}^{+\infty} 2 \frac{\eta w^{k-1}}{4} \frac{w^{k-1}}{2} d \eta \leq \int_{0}^{+\infty}  \frac{\eta^2 (w^{k-1})^{2}}{16} +\frac{(w^{k-1})^{2}}{4} d \eta \\
&\leq \int_{0}^{+\infty} \left( \frac{\eta^2}{16} -\frac{3}{4}\right) (w^{k-1})^{2} +(w^{k-1})^{2} d \eta \leq Q(w^{k-1})+ \|w^{k-1}\|_{L^{2}}^{2} 
\end{split}
\end{equation}
\normalsize
so that
\small
\begin{equation}
\begin{split}
&\frac{d}{d \tau}\|w^{k-1}\|_{L^{2}}^{2}+2\left(1-\frac{1}{2}\varepsilon e^{-\tau / 2}\right) Q(w^{k-1}) +2\left(1-\frac{1}{2}\varepsilon e^{-\tau / 2}\right)\|w^{k-1}\|_{L^{2}}^{2} \\
&\leq 2 \alpha e^{-\frac{\delta(\tau)^2}{8}} \int_{0}^{+\infty} w^{k-1} w^{k}(\tau,\eta+\delta(\tau))  e^{-\frac{\eta \delta(\tau)}{4}} d \eta \leq \|w^{k-1}\|_{L^{2}}^{2}  + \|w^{k}\|_{L^{2}}^{2} \alpha^2 \\
\end{split}
\end{equation}
\normalsize
We consider $\varepsilon< \frac{1}{2}$. We have assumed that 
$
\|w^{k}\|_{L^{2}}$ is bounded, say
$
\|w^{k}\|_{L^{2}}^{2} \leq C $ for some $C>0$ that only depends on the initial data, and so we get 

$$
\frac{d}{d \tau}\|w^{k-1}\|_{L^{2}}^{2}+  \frac{1}{2}\|w^{k-1}\|_{L^{2}}^{2} \leq   C \alpha^2
$$

We can thus conclude that 
$
\|w^{k-1}\|_{L^{2}}$ is bounded by some constant $C$ that only depends on the initial data and $\alpha$.

Now we work out the cases $w^{i}$ with $i=1,\cdots,k-2$.
For those $i$, we have $\frac{3}{2}+i-k <0$ and the equations

\small
\begin{equation}
\begin{split}
& \frac{d}{d \tau}\|w^{i}\|_{L^{2}}^{2}+2 Q(w^{i})+ 2(k-i)\|w^{i}\|_{L^{2}}^{2} =2 \varepsilon \left( \frac{3}{2}+i-k \right) e^{-\frac{\tau}{2}} \int_{0}^{+\infty} \frac{\eta}{4} (w^{i})^{2} d \eta  \\
& +2 \alpha e^{-\frac{\delta(\tau)^2}{8}} \int_{0}^{+\infty} w^{i}w^{i+1}(\tau,\eta+\delta(\tau))  e^{-\frac{\eta \delta(\tau)}{4}} d \eta \\
\end{split}
\end{equation}
\normalsize

so we get the bound
\small
$$
\frac{d}{d \tau}\|w^{i}\|_{L^{2}}^{2}+2 Q(w^{i})+ 2(k-i)\|w^{i}\|_{L^{2}}^{2} \leq 2 \alpha e^{-\frac{\delta(\tau)^2}{8} } \int_{0}^{+\infty} w^{i}w^{i+1}(\tau,\eta+\delta(\tau))  e^{-\frac{\eta \delta(\tau)}{4}} d \eta
$$
\normalsize

and thus 
$$
\frac{1}{2}\frac{d}{d \tau}\|w^{i}\|_{L^{2}}^{2}+ (k-i)\|w^{i}\|_{L^{2}}^{2} \leq  \alpha \int_{0}^{+\infty} w^{i}w^{i+1}(\tau,\eta+\delta(\tau)) d \eta \leq \frac{\|w^{i}\|_{L^{2}}^{2}}{2} + \frac{\|w^{i+1}\|_{L^{2}}^{2} \alpha^2}{2}
$$
We have supposed that 
$
\|w^{i+1}\|_{L^{2}}^{2}$ is bounded for all times $\tau$, so we get 
$$
\frac{d}{d \tau}\|w^{i}\|_{L^{2}}^{2}+  (2k-2i-1)\|w^{i}\|_{L^{2}}^{2} \leq   C \alpha^2.
$$
Since  $2k-2i-1>1$ we get that  $
\|w^{i}(\tau)\|_{L^{2}}$ is bounded for all times $\tau$ by a constant that depends on the initial data and $\alpha$. 

We know from Lemma 2.2 in \cite{short_proof} that $\|w^{k}\|_{L^{2}}$ is bounded. By induction, we get that  $\|w^{i}\|_{L^{2}}$ is bounded for all $i \in \{1,\ldots,k-1\}$.
\end{proof}

\begin{lemma}
\label{lemma:Qwi-bounded}
There exists a constant $C>0$ such that 
\[
Q(w^{i}(\tau))
\le C
\quad\text{for all } \tau\ge0 \text{ and all } i\in\{1,\ldots,k\}.
\]
\end{lemma}

\begin{proof}
We first treat the last component $w^{k}$, and then $w^{i}$ for $i<k$.

\[
\frac{d}{d\tau}Q(w^{k})
=2\langle M w^{k},w^{k}_{\tau}\rangle
=-2\|Mw^{k}\|_{L^{2}}^{2}
-3\varepsilon e^{-\tau/2}\langle M w^{k},\,w^{k}_{\eta}-\tfrac{\eta}{4}w^{k}\rangle.
\]
We have, 
\[
\varepsilon \big|\langle M w^{k},\,w^{k}_{\eta}-\tfrac{\eta}{4}w^{k}\rangle\big|
\le \tfrac12\|Mw^{k}\|_{L^{2}}^{2}+C\varepsilon \|w^{k}_{\eta}-\tfrac{\eta}{4}w^{k}\|_{L^{2}}^{2},
\]
so that
\[
\frac{d}{d\tau}Q(w^{k})
\le
-\|Mw^{k}\|_{L^{2}}^{2}+C\varepsilon \|w^{k}_{\eta}-\tfrac{\eta}{4}w^{k}\|_{L^{2}}^{2}.
\]
Expanding the last term and comparing with $Q(w^{k})$, we have
\[
\|w^{k}_{\eta}-\tfrac{\eta}{4}w^{k}\|_{L^{2}}^{2}
\le C \big(Q(w^{k})+\|w^{k}\|_{L^{2}}^{2}\big),
\]
hence
\begin{equation}
\label{eq:Qwk-ineq}
\frac{d}{d\tau}Q(w^{k})
\le
- \|Mw^{k}\|_{L^{2}}^{2}
+ C \varepsilon Q(w^{k})
+ C \|w^{k}\|_{L^{2}}^{2}.
\end{equation}
Using the spectral gap of $M$, namely
$
\|Mw\|_{L^{2}}^{2}\ge Q(w)
$
we obtain from \eqref{eq:Qwk-ineq}
\[
\frac{d}{d\tau}Q(w^{k})\le-(1-C\varepsilon)Q(w^{k})+C\|w^{k}\|_{L^{2}}^{2}.
\]
Since $\|w^{k}\|_{L^{2}}$ is uniformly bounded in $\tau$ and $1>C\varepsilon$ for $t_{0}$ large enough, we conclude that
\[
Q(w^{k}(\tau))\le C \quad \text{for all } \tau\ge0.
\]

\medskip
\noindent
For $w^{i}, i < k$ we set $$\mathcal{T}_{i}[w^{i+1}]=w^{i+1}(\tau, \eta + \delta(\tau))e^{-A(\tau)}e^{-B(\tau)\eta}$$
where $A(\tau)=\frac{\delta(\tau)^2}{8}$ and $B(\tau)=\frac{\delta(\tau)}{4}$. Then we can write the equation for $Q(w^{i})$ as
\begin{align*}
\frac{1}{2}\frac{d}{d\tau}Q(w^{i})
&=-\|Mw^{i}\|_{L^{2}}^{2}-(k-i)Q(w^{i})
-\Big(\tfrac{3}{2}+i-k\Big)\varepsilon e^{-\tau/2}
\langle M w^{i},\,w^{i}_{\eta}-\tfrac{\eta}{4}w^{i}\rangle +\alpha\,\langle M w^{i},\mathcal{T}_{i}[w^{i+1}]\rangle.
\end{align*}
The drift term is treated as before and yields
\[
-\Big(\tfrac{3}{2}+i-k\Big)\varepsilon e^{-\tau/2}
\langle M w^{i},\,w^{i}_{\eta}-\tfrac{\eta}{4}w^{i}\rangle
\le
\frac{\varepsilon}{2}\|Mw^{i}\|_{L^{2}}^{2}+C\varepsilon(Q(w^{i})+\|w^{i}\|_{L^{2}}^{2}).
\]
For the coupling term we have, 
\[
\alpha\,|\langle M w^{i},\mathcal{T}_{i}[w^{i+1}]\rangle|
\le
\tfrac14\|Mw^{i}\|_{L^{2}}^{2}
+ C \alpha^{2}\|\mathcal{T}_{i}[w^{i+1}]\|_{L^{2}}^{2}.
\]
Since $\mathcal{T}_{i}$ is a spatial shift and an exponentially decaying weight, we have
\[
\|\mathcal{T}_{i}[w^{i+1}]\|_{L^{2}}
\le C\|w^{i+1}\|_{L^{2}},
\]
and we know that $\|w^{i+1}\|_{L^{2}}$ is uniformly bounded in $\tau$. 
Collecting all terms, we obtain
\[
\frac{d}{d\tau}Q(w^{i})
\le
-\|Mw^{i}\|_{L^{2}}^{2}
+C \varepsilon Q(w^{i})
+C\|w^{i}\|_{L^{2}}^{2}
+C\|w^{i+1}\|_{L^{2}}^{2}.
\]
Using the spectral gap estimate $\|Mw^{i}\|_{L^{2}}^{2}\ge Q(w^{i})$ again, we obtain
\[
\frac{d}{d\tau}Q(w^{i})
\le
-(1-C \varepsilon)Q(w^{i})
+C\big(\|w^{i}\|_{L^{2}}^{2}+\|w^{i+1}\|_{L^{2}}^{2}\big).
\]
Since the $L^{2}$ norms of $w^{i}$ and $w^{i+1}$ are uniformly bounded, 
\[
Q(w^{i}(\tau))\le C
\quad \text{for all }\tau\ge0.
\]

\end{proof}

\begin{Corollary} \label{boundedderivative}
There exists $C>0$ such that $||w_{\eta}^{i}||_{L^2} \leq C$ for all $ i \in \{1, \cdots, k \}$.
\end{Corollary}
\begin{proof}
Using Lemmas \ref{boundedl2normlemma}, \ref{lemma:Qwi-bounded}  we get directly

$$||w_{\eta}^{i}||_{L^2}^{2}=\int_0^{+\infty} |w_{\eta}^{i}|^{2} d \eta \leq Q(w^{i})+\frac{3}{4} ||w^{i}||_{L^2}^{2} \leq C.$$
\end{proof}

We can now decompose $w^{i}$ into its principal eigenfunction component and its orthogonal complement 
$$
w^{i}=q^{i}(\tau) e_{0} +\widehat{w^{i}}
$$
where $\widehat{w^{i}}$ is orthogonal to $e_0$ and $q^{i}(\tau)=\left \langle e_0, w^{i}(\tau) \right \rangle $.

Now we find an equation for $q^{i}(\tau)$. The following hold 

\begin{enumerate}
\item{$\left \langle e_0, M w^{i}(\tau) \right \rangle=\left \langle M e_0, w^{i}(\tau) \right \rangle=0$ since $M$ is symmetric, and $Me_{0}=0$. }
\newline 
\item{$\left \langle e_0, w^{i}_{\eta}-\frac{\eta}{4}w^{i} \right \rangle=-\left \langle (e_0)_{\eta}+\frac{\eta}{4}e_0, w^{i}(\tau) \right \rangle $}
\end{enumerate}

Using these facts and taking the inner product with $e_0$ in the equation for $w^{i}$ gives us

\begin{equation} \label{equationforqi}
\begin{split}
q^{i}_{\tau}+(k-i)q^{i}(\tau)&=\left( \frac{3}{2}+i-k \right) \varepsilon e^{-\frac{\tau}{2}}\left \langle (e_0)_{\eta}+\frac{\eta}{4}e_0, w^{i}(\tau) \right \rangle +\alpha \left\langle e_0, \mathcal{T}_{i}[w^{i+1}] \right\rangle \\
q^{k}_{\tau}&= \frac{3}{2} \varepsilon e^{-\frac{\tau}{2}}\left \langle (e_0)_{\eta}+\frac{\eta}{4}e_0, w^{k}(\tau) \right \rangle\\
\end{split}
\end{equation}
Let us set
$$
r_{i}(\tau)=\left \langle (e_0)_{\eta}+\frac{\eta}{4}e_0, w^{i}(\tau) \right \rangle, h_{i}(\tau)= \left\langle e_0, \mathcal{T}_{i}[w^{i+1}]\right\rangle.
$$

to write our equations as

$$
q^{i}_{s}+(k-i)q^{i}(s)=\left( \frac{3}{2}+i-k \right)\varepsilon e^{-\frac{s}{2}}r_{i}(s) +\alpha h_{i}(s)
$$

After multiplying with $e^{(k-i)s}$ and integrating we get

$$
 q^{i}(\tau)=e^{-(k-i)\tau}q^{i}(0)+ \left(\frac{3}{2}+i-k \right) \varepsilon e^{-(k-i)\tau} \int_{0}^{\tau} e^{(k-i-\frac{1}{2})s}r_{i}(s) \mathrm{d} s +\alpha e^{-(k-i)\tau}\int_{0}^{\tau} e^{(k-i)s} h_{i}(s) \mathrm{d} s
$$
From \cite{short_proof} we know that $$ | q^{k}(\tau)-q^{k}(0) | \leq C\varepsilon $$ where $\varepsilon =\frac{1}{\sqrt{t_0}}$. We will show more generally that $q^{i}(\tau)$ is close to $\frac{\alpha^{k-i}}{(k-i)!}q^{k}(0)$ in the following Lemma.

\begin{lemma} \label{qiclosetoqk}
For $i=1,\ldots,k-1$ there exist constants $C_1,C_2$ such that 
\begin{equation*}
| q^{i}(\tau)-\frac{\alpha^{k-i}}{(k-i)!} q^{k}(0) | \leq C_1 \varepsilon + C_2 (1+\tau) e^{-\frac{\tau}{2}}  
\end{equation*}
\end{lemma}

\begin{proof}
    
We already know $| q^{k}(\tau)-q^{k}(0) | \leq C\varepsilon $ from Lemma 2.2 in \cite{short_proof}, where $\varepsilon=\frac{1}{\sqrt{t_0}}$, so it holds for $i=k$. We will suppose that it holds for $q^{i+1}$ and then show that it holds for $q^{i}$. 
\newline 

We have
\begin{equation}
\begin{split}
 & |q^{i}(\tau)-\frac{\alpha^{k-i}}{(k-i)!} q^{k}(0)|\leq e^{-(k-i)\tau}|q^{i}(0)|+ |\left(\frac{3}{2}+i-k \right) \varepsilon e^{-(k-i)\tau} \int_{0}^{\tau} e^{(k-i-\frac{1}{2})s}r_{i}(s) \mathrm{d} s | \\
 & +|\alpha e^{-(k-i)\tau}\int_{0}^{\tau} e^{(k-i)s} h_{i}(s) \mathrm{d} s-\frac{\alpha^{k-i}}{(k-i)!}q^{k}(0)|
\end{split} 
\end{equation}

By Cauchy-Schwarz, since we already showed that the $L^2$ norm of $w^{i}$ is bounded we get 

$$ |r_{i}(\tau)|=|\left \langle (e_0)_{\eta}+\frac{\eta}{4}e_0, w^{i}(\tau) \right \rangle | \leq C.$$

As a result, we can bound the second term
\small

\begin{equation}
\begin{split}
&|\left(\frac{3}{2}+i-k \right)\varepsilon e^{-(k-i)\tau} \int_{0}^{\tau} e^{(k-i-\frac{1}{2})s}r_{i}(s) \mathrm{d} s| \leq C |\left(\frac{3}{2}+i-k \right)|\varepsilon e^{-(k-i)\tau} \int_{0}^{\tau} e^{(k-i-\frac{1}{2})s}\mathrm{d} s  \leq C\varepsilon e^{-\frac{\tau}{2}}
\end{split}
\end{equation}

\normalsize

Next,  we bound the third term. 

We can write

$$-\frac{\alpha^{k-i}}{(k-i)!} q^{k}(0)=-\frac{\alpha^{k-i}}{(k-i-1)!} e^{-(k-i)\tau}\int_{0}^{\tau} e^{(k-i)s} q^{k}(0) \mathrm{d} s -\frac{\alpha^{k-i}}{(k-i)!} q^{k}(0)e^{-(k-i)\tau} $$

and get 

\begin{equation}
\begin{split}
&|\alpha e^{-(k-i)\tau} \int_{0}^{\tau}  e^{(k-i)s} h_{i}(s)\mathrm{d} s -\frac{\alpha^{k-i}}{(k-i)!} q^{k}(0)|=\\
&|\alpha e^{-(k-i)\tau} \int_{0}^{\tau}  e^{(k-i)s} (h_{i}(s)-q^{i+1}(s))\mathrm{d} s +\alpha e^{-(k-i)\tau} \int_{0}^{\tau}  e^{(k-i)s} (q^{i+1}(s)-\frac{\alpha^{k-i-1}}{(k-i-1)!} q^{k}(0))\mathrm{d} s \\
&-\frac{\alpha^{k-i}}{(k-i)!} q^{k}(0)e^{-(k-i)\tau}| \\
&\leq \alpha e^{-(k-i)\tau}\int_{0}^{\tau}|h_{i}(s)-q^{i+1}(s)|e^{(k-i)s} \mathrm{d} s +\alpha e^{-(k-i)\tau}\int_{0}^{\tau}|q^{i+1}(s)-\frac{\alpha^{k-i-1}}{(k-i-1)!}q^{k}(0)|e^{(k-i)s} \mathrm{d} s  \\
&+\frac{\alpha^{k-i}}{(k-i)!}|q^{k}(0)|e^{-(k-i)\tau} \\
\end{split}
\end{equation}

Now we need to bound $|h_{i}(s)-q^{i+1}(s)|$. We have that there exists a constant $C>0$ such that
\[
|h_i(s)-q^{i+1}(s)| \le C\,\delta(s)
\quad \text{for all } s\ge 0.
\]

Indeed, we  write
\begin{equation}
\label{eq:hi-minus-qi+1-start}
h_i(s)-q^{i+1}(s)
=\int_{0}^{\infty} e_0(\eta)
\Big[
\mathcal{T}_{i}[w^{i+1}](s,\eta)
-
w^{i+1}(s,\eta)
\Big]
\,d\eta.
\end{equation}

We add and subtract $w^{i+1}(s,\eta+\delta(s))$ inside the bracket and split
\[
h_i(s)-q^{i+1}(s)=T_1(s)+T_2(s),
\]
where
\[
T_1(s)=\int_{0}^{\infty} e_0(\eta)
\Big(
w^{i+1}(s,\eta+\delta(s))-w^{i+1}(s,\eta)
\Big)\,d\eta,
\]
\[
T_2(s)=\int_{0}^{\infty} e_0(\eta)\,
w^{i+1}(s,\eta+\delta(s))
\Big(
e^{-A_s}e^{-B_s \eta}-1
\Big)\,d\eta.
\]

\medskip
\noindent
We can write
\[
w^{i+1}(s,\eta+\delta(s))-w^{i+1}(s,\eta)
=\int_{0}^{\delta(s)} \partial_{\eta} w^{i+1}(s,\eta+\theta)\,d\theta.
\]
Hence
\[
|T_1(s)|
\le
\int_{0}^{\delta(s)}
\int_{0}^{\infty} |e_0(\eta)|\,|\partial_{\eta} w^{i+1}(s,\eta+\theta)|\,d\eta\,d\theta.
\]
Using Cauchy-Schwarz in $\eta$ we get
\[
\int_{0}^{\infty} |e_0(\eta)|\,|\partial_{\eta} w^{i+1}(s,\eta+\theta)|\,d\eta
\le
\|e_0\|_{L^{2}}\,
\|\partial_{\eta} w^{i+1}(s,\cdot+\theta)\|_{L^{2}}
\le
C\,\|w^{i+1}_{\eta}(s,\cdot)\|_{L^{2}}.
\]
Therefore
\[
|T_1(s)|
\le
\delta(s)\,C\,\|w^{i+1}_{\eta}(s,\cdot)\|_{L^{2}}.
\]
From corollary \eqref{boundedderivative} we have shown that  $\|w^{i+1}_{\eta}(s,\cdot)\|_{L^{2}}$ is uniformly bounded in $s$. It follows that
\begin{equation}
\label{eq:T1bound}
|T_1(s)|\le C\,\delta(s).
\end{equation}

\medskip
\noindent
To estimate $T_{2}(s)$ we expand the exponential factor. Since
$
A_s=O(\delta(s)^2),
B_s=O(\delta(s)),
$
we write
\[
e^{-A_s}e^{-B_s \eta}-1
= -A_s - B_s \eta + R_s(\eta),
\]
where the remainder satisfies 
\[
|R_s(\eta)| \le C\,\delta(s)^2 (1+\eta^2).
\]
Then
\[
|T_2(s)|
\le
\int_{0}^{\infty} |e_0(\eta)|\,|w^{i+1}(s,\eta+\delta(s))|
\Big(
A_s + B_s \eta + C\,\delta(s)^2(1+\eta^2)
\Big)\,d\eta.
\]
Applying Cauchy-Schwarz in $\eta$,
\[
|T_2(s)|
\le
\|w^{i+1}\|_{L^{2}}
\Big\|
e_0(\eta)\big(
A_s + B_s \eta + C \delta(s)^2(1+\eta^2)
\big)
\Big\|_{L^{2}}.
\]

Since $e_0(\eta)=\eta e^{-\eta^2/8}$, 
$\eta e_0(\eta)$ and $\eta^2 e_0(\eta)$ are in $L^{2}(0,\infty)$.
Therefore
\[
\Big\|
e_0(\eta)\big(
A_s + B_s \eta + C \delta(s)^2(1+\eta^2)
\big)
\Big\|_{L^{2}}
\le
C\,\delta(s),
\]
since $A_s=O(\delta(s)^2)$ and $B_s=O(\delta(s))$.

We conclude
\begin{equation}
\label{eq:T2bound}
|T_2(s)| \le C\,\delta(s).
\end{equation}

\medskip
Combining \eqref{eq:T1bound} and \eqref{eq:T2bound} we obtain
\[
|h_i(s)-q^{i+1}(s)|
\le
|T_1(s)|+|T_2(s)|
\le
C\,\delta(s)
=
C\,\frac{s+\ln(t_0)}{\sqrt{t_0 e^{s}}},
\]
which is the desired estimate.

\small

\normalsize

We are now ready to bound the first part of the third term

\begin{equation}
\begin{split}
& \alpha e^{-(k-i)\tau}\int_{0}^{\tau}|h_{i}(s)-q^{i+1}(s)|e^{(k-i)s} \mathrm{d} s  \leq C_{2} \alpha e^{-(k-i)\tau}\int_{0}^{\tau}\frac{s+\ln(t_0)}{\sqrt{t_0 e^{s}}}e^{(k-i)s} \mathrm{d} s = \\
&=\frac{C_{2} \alpha}{\sqrt{t_0}} e^{-(k-i)\tau} \left( \int_{0}^{\tau} (s+\ln(t_0)) e^{(k-i-\frac{1}{2})s}  \mathrm{d} s \right) 
\end{split}
\end{equation}

\begin{equation}
\begin{split}
&\int_{0}^{\tau} (s+\ln(t_0)) e^{(k-i-\frac{1}{2})s}  \mathrm{d} s =\frac{1}{k-i-\frac{1}{2}}\tau e^{(k-i-\frac{1}{2})\tau}-\frac{1}{(k-i-\frac{1}{2})^{2}}\left( e^{(k-i-\frac{1}{2})\tau} -1\right)+\frac{\ln(t_0)}{k-i-\frac{1}{2}} \left( e^{(k-i-\frac{1}{2})\tau}-1\right)
\end{split}
\end{equation}

So we get

\begin{equation}
\begin{split}
&\alpha e^{-(k-i)\tau}\int_{0}^{\tau}|h_{i}(s)-q^{i+1}(s)|e^{(k-i)s} \mathrm{d} s \leq C_1 \varepsilon +C_2 (\tau+1) e^{-\frac{\tau}{2}} 
\end{split}
\end{equation}

In order to bound the second part of the third term, we use the inductive hypothesis for $q^{i+1}$

$$ | q^{i+1}(\tau)-\frac{\alpha^{k-i-1}}{(k-i-1)!} q^{k}(0) | \leq C_1 \varepsilon + C_2 (\tau+1) e^{-\frac{\tau}{2}}. $$

We write

\begin{equation}
\begin{split}
&\alpha e^{-(k-i)\tau}\int_{0}^{\tau}|q^{i+1}(s)-\frac{\alpha^{k-i-1}}{(k-i-1)!}q^{k}(0)|e^{(k-i)s} \mathrm{d} s \leq  \alpha e^{-(k-i)\tau}\int_{0}^{\tau} \left( C_1 \varepsilon + C_2 (s+1)e^{-\frac{s}{2}} \right) e^{(k-i)s} \mathrm{d} s \\
&\leq C  \varepsilon+C (\tau+1) e^{-\frac{\tau}{2}}.
\end{split}
\end{equation}

Putting these integrals together gives us the bound for the third term

\begin{equation}
\begin{split}
&|\alpha e^{-(k-i)\tau} \int_{0}^{\tau}  e^{(k-i)s} h_{i}(s)\mathrm{d} s -\frac{\alpha^{k-i}}{(k-i)!} q^{k}(0)| \leq  C_1 \varepsilon +C_2 (\tau+1) e^{-\frac{\tau}{2}}
\end{split}
\end{equation}

\normalsize

Combining everything, we get:
\small

\begin{equation}
\begin{split}
&|q^{i}(\tau)-\frac{\alpha^{k-i}}{(k-i)!} q^{k}(0)|\leq 
e^{-(k-i)\tau}|q^{i}(0)|+|\left(\frac{3}{2}+i-k \right)|\varepsilon e^{-(k-i)\tau} \int_{0}^{\tau} e^{(k-i-\frac{1}{2})s}|r_{i}(s)| \mathrm{d} s \\
&+ |\alpha e^{-(k-i)\tau}\int_{0}^{\tau} e^{(k-i)s} h_{i}(s) \mathrm{d} s-\frac{\alpha^{k-i}}{(k-i)!} q^{k}(0)| \leq C_1 \varepsilon + C_2 (\tau+1) e^{-\frac{\tau}{2}}
\end{split}
\end{equation}

where the constants $C_1,C_2$ depend on $\alpha$ and on the initial conditions.
\end{proof}
\normalsize
\begin{lemma} \label{boundedorthogonalremainderlemma}
There exists a constant $C=C(\alpha,\varepsilon)>0$ that depends on $\alpha, \varepsilon$ and the initial conditions such that $\|\widehat{w^{i}}\|_{L^2} \leq C (\tau +1)e^{-\frac{\tau}{2}}$.
\end{lemma}
\begin{proof}

We know from \cite{short_proof} that
$\|\widehat{w^{k}}(\tau)\|_{L^{2}} \leq C e^{-\tau / 2}$, so it holds for $i=k$.
\newline 
We assume that $\| \widehat{w^{i+1}}(\tau) \|_{L^2} \leq C (\tau+1)e^{-\frac{\tau}{2}}$ and show that then it holds for $i$, i.e.
$$\| \widehat{w^{i}}(\tau) \|_{L^2} \leq C (\tau+1) e^{-\frac{\tau}{2}}$$ 

for some constant $C$ that depends on $\alpha, \varepsilon$ and the initial conditions.

We use the equations \eqref{cascadingMoperator},\eqref{equationforqi} for $w^{i}, q^{i}(\tau)$, and the fact that $M \widehat{w^{i}}= M w^{i} $ as well as  

$w_{\eta}^{i}-\frac{\eta}{4}w^{i}=\widehat{w^{i}}_{\eta} -\frac{\eta}{4} \widehat{w^{i}} + q^{i}(\tau)\left( (e_0)_{\eta}
-\frac{\eta}{4} e_0 \right) $
and derive an equation for $\widehat{w^{i}}=w^{i}-q^{i}(\tau)e_0$. 

\begin{equation}
\begin{split}
&\widehat{w^{i}}_{\tau}+ M \widehat{w^{i}}= -(k-i)\widehat{w^{i}}- \left(\frac{3}{2}+i-k \right) \varepsilon  e^{-\tau / 2}\left(\left\langle\left(e_{0}\right)_{\eta}+\frac{\eta}{4} e_{0}, w^{i}\right\rangle e_{0}+q^{i}(\tau)\left(\left(e_{0}\right)_{\eta}-\frac{\eta}{4} e_{0}\right)\right)\\
& - \left(\frac{3}{2}+i-k \right)\varepsilon e^{-\tau / 2} \left( \widehat{w^{i}}_{\eta}-\frac{\eta}{4} \widehat{w^{i}}\right)+\alpha \mathcal{T}_{i}[w^{i+1}]-\alpha \left\langle e_0,\mathcal{T}_{i}[w^{i+1}] \right\rangle e_0 \\
\end{split}
\end{equation}

After multiplying with $\widehat{w^{i}}$ and integrating, using that $e_0$ and $\widehat{w}^{i}$ are orthogonal we get

\begin{equation} \label{eqtobound}
\begin{split}
&\frac{1}{2}\frac{d}{d \tau}\|\widehat{w^{i}}\|_{L^{2}}^{2}+ Q(\widehat{w^{i}})=-(k-i)\|\widehat{w^{i}}\|_{L^{2}}^{2} -\left(\frac{3}{2}+i-k \right)\varepsilon e^{-\frac{\tau}{2}} q^{i}(\tau) \left \langle (e_0)_{\eta} -\frac{\eta}{4} e_0, \widehat{w^{i}} \right \rangle \\
& + \left(\frac{3}{2}+i-k \right)\varepsilon e^{-\frac{\tau}{2}} \int_{0}^{+\infty } \frac{\eta}{4} {\widehat{w^{i}}}^{2} \mathrm{d} \eta +\alpha \left \langle \mathcal{T}_{i}[w^{i+1}], \widehat{w^{i}} \right \rangle \\
\end{split}
\end{equation}

Now let us see how we will bound each of the terms in equation $\ref{eqtobound}$.

\begin{equation}
\begin{split}
& -\left(\frac{3}{2}+i-k \right)\varepsilon e^{-\frac{\tau}{2}} q^{i}(\tau) \left \langle (e_0)_{\eta} -\frac{\eta}{4} e_0, \widehat{w^{i}} \right \rangle \leq  C \varepsilon e^{-\frac{\tau}{2}} \| \widehat{w^{i}}\|_{L^2}  \\
\end{split}
\end{equation}

\begin{equation}
\begin{split}
& \left(\frac{3}{2}+i-k \right)\varepsilon e^{-\frac{\tau}{2}} \int_{0}^{+\infty } \frac{\eta}{4} {\widehat{w^{i}}}^{2} \mathrm{d} \eta  \leq |\left(\frac{3}{2}+i-k \right)|\varepsilon e^{-\frac{\tau}{2}}\left( Q(\widehat{w^{i}}) +\|\widehat{w^{i}}\|^2_{L^2} \right) \leq C \varepsilon e^{-\frac{\tau}{2}}  Q(\widehat{w^{i}})
\end{split}
\end{equation}

\normalsize

We split the coupling term
$
\Big\langle
\mathcal{T}_{i}[w^{i+1}],
\widehat{w^{i}}
\Big\rangle
=
\Big\langle
\mathcal{T}_{i}[w^{i+1}]-w^{i+1},
\widehat{w^{i}}
\Big\rangle
+
\Big\langle
w^{i+1},
\widehat{w^{i}}
\Big\rangle.
$ The second piece is
$
\Big\langle
w^{i+1},
\widehat{w^{i}}
\Big\rangle
=
\Big\langle
q^{i+1}(\tau)e_0+\widehat{w^{i+1}},
\widehat{w^{i}}
\Big\rangle
=
\Big\langle
\widehat{w^{i+1}},
\widehat{w^{i}}
\Big\rangle,
$
since $\langle e_0,\widehat{w^{i}}\rangle=0$. By Cauchy-Schwarz and the inductive hypothesis, 
\[
\Big|
\langle
\widehat{w^{i+1}},
\widehat{w^{i}}
\rangle
\Big|
\le
\|\widehat{w^{i+1}}\|_{L^{2}}
\|\widehat{w^{i}}\|_{L^{2}}
\le
C\,(\tau+1)e^{-\tau/2}\|\widehat{w^{i}}\|_{L^{2}}
\]

For the difference term
$
\Big\langle
\mathcal{T}_{i}[w^{i+1}]-w^{i+1},\,\widehat{w^{i}}
\Big\rangle
$ we write 
$$
\mathcal{T}_{i}[w^{i+1}](\tau,\eta)-w^{i+1}(\tau,\eta)
=
\big(w^{i+1}(\tau,\eta+\delta)-w^{i+1}(\tau,\eta)\big)
+
w^{i+1}(\tau,\eta+\delta)\big(e^{-A_\tau}e^{-B_\tau\eta}-1\big)$$ 

and control each piece in $L^{2}$ exactly as in the proof of Lemma~\ref{qiclosetoqk}. For the shift part,
\[
w^{i+1}(\tau,\eta+\delta)-w^{i+1}(\tau,\eta)
=
\int_{0}^{\delta} \partial_{\eta} w^{i+1}(\tau,\eta+\theta)\,d\theta,
\]
so
\[
\|w^{i+1}(\cdot+\delta)-w^{i+1}(\cdot)\|_{L^{2}}
\le
\delta\,\|w_{\eta}^{i+1}(\tau)\|_{L^{2}}
\le
C\,\delta,
\]
since $\|w^{i+1}_{\eta}\|_{L^{2}}$ is uniformly bounded in $\tau$.

For the weight error we have 
$
e^{-A_\tau}e^{-B_\tau \eta}-1
= -A_\tau - B_\tau \eta + R_\tau(\eta),
$
where the remainder satisfies 
$
|R_\tau(\eta)| \le C\,\delta(\tau)^2 (1+\eta^2).
$ So we get 

\[
\|w^{i+1}(\tau, \eta+\delta)(e^{-A_\tau} e^{-B_\tau \eta \cdot}-1)\|_{L^{2}}
\le
C\,\delta.
\]

Combining these two pieces,
\[
\|\mathcal{T}_{i}[w^{i+1}]-w^{i+1}\|_{L^{2}}
\le
C\,\delta(\tau)
\le
C \varepsilon_{1} \,(\tau+1)e^{-\tau/2},
\]

Therefore,
\[
\Big|
\langle
\mathcal{T}_{i}[w^{i+1}]-w^{i+1},
\widehat{w^{i}}
\rangle
\Big|
\le
\|\mathcal{T}_{i}[w^{i+1}]-w^{i+1}\|_{L^{2}}
\|\widehat{w^{i}}\|_{L^{2}}
\le
C\varepsilon_{1}\,(\tau+1)e^{-\tau/2}\|\widehat{w^{i}}\|_{L^{2}}.
\]

Putting both parts together,
\[
\alpha \Big|
\langle
\mathcal{T}_{i}[w^{i+1}],
\widehat{w^{i}}
\rangle
\Big|
\le
C \alpha (\tau+1) e^{-\tau/2}||\widehat{w^{i}}||_{L^2}
\]

We can now conclude this energy estimate.
Using those bounds we obtain
\[
\frac{d}{d\tau}\|\widehat{w^{i}}\|_{L^{2}}^{2}
+ \big(2 - C\varepsilon \big)\,Q(\widehat{w^{i}})
\le
C_{\alpha} e^{-\tau/2}(\tau+1)||\widehat{w^{i}}||_{L^2}
\]
For $t_0$ large we have $\varepsilon$ small, so we may absorb constants and get after using the spectral gap inequality $\|\widehat{w^{i}}\|_{L^{2}}^{2}\le Q(\widehat{w^{i}})$,
\[
\frac{d}{d\tau}\|\widehat{w^{i}}\|_{L^{2}}^{2}
+ \frac{3}{2}\,\|\widehat{w^{i}}\|_{L^{2}}^{2}
\le
C_{\alpha} e^{-\tau/2}(\tau+1) ||\widehat{w^{i}}||_{L^2}
\]
Set $\phi(\tau)=\|\widehat{w^{i}}(\tau)\|_{L^{2}}$. Then
\[
2 \phi(\tau)\phi'(\tau)+\frac{3}{2}\,\phi^2(\tau)
\le
C_{\alpha} e^{-\tau/2}(\tau+1)\phi(\tau).
\]
and 
\[
\phi'(\tau)+\frac{3}{4}\,\phi(\tau)
\le
C_{\alpha} e^{-\tau/2}(\tau+1).
\]

and thus  
\[
\phi(\tau)
\le
C'(\alpha,\varepsilon)\,(\tau+1) e^{-\tau/2},
\]

We have thereby shown that there exists a constant $C=C(\alpha,\varepsilon)>0$ that depends on $\alpha, \varepsilon$ and the initial conditions such that $\|\widehat{w^{i}}\|_{L^2} \leq C (\tau +1)e^{-\frac{\tau}{2}}$.
\end{proof}

\begin{lemma} \label{boundedlinfnormforderivative}
For every \(i=1,\dots,k\) and every compact set $K$, there exists
\(C_K>0\) such that
\[
\|(\widehat{w^i})_\eta(\tau)\|_{L^\infty(K)}
\le C_K(\tau+1)e^{-\tau/2}.
\]
\end{lemma}

\begin{proof}
We argue again by downward induction on \(i\). For \(i=k\), in \cite{short_proof} it is shown that for every compact set $K$,
\[
\|(\widehat{w^k})_\eta(\tau)\|_{L^\infty(K)}
\le C_K e^{-\tau/2}.
\]

Now assume that the claim holds for \(i+1\) and we prove it for \(i\). We already know from
the previous energy estimate that
\[
\|\widehat{w^i}(\tau)\|_{L^2}
\le C(\tau+1)e^{-\tau/2}.
\]
We choose a slightly larger compact interval $K'$ and control the right-hand side of the equation for
\(\widehat{w^i}\) on \(K'\).

\begin{equation}
\begin{split}
&\widehat{w^{i}}_{\tau}+ (M+(k-i)) \widehat{w^{i}} +\left(\frac{3}{2}+i-k \right)\varepsilon e^{-\tau / 2} \left( \widehat{w^{i}}_{\eta}-\frac{\eta}{4} \widehat{w^{i}}\right) \\
&=-\left(\frac{3}{2}+i-k \right) \varepsilon  e^{-\tau / 2}\left(\left\langle\left(e_{0}\right)_{\eta}+\frac{\eta}{4} e_{0}, w^{i}\right\rangle e_{0}+q^{i}(\tau)\left(\left(e_{0}\right)_{\eta}-\frac{\eta}{4} e_{0}\right)\right)+\alpha \mathcal{T}_{i}[w^{i+1}]-\alpha \left\langle e_0, \mathcal{T}_{i}[w^{i+1}] \right\rangle e_0 \\
\end{split}
\end{equation}

Our right-hand side is 
\[
\begin{split}
F^i(\tau,\eta)
&=
-a_i\varepsilon e^{-\tau/2}
\left(
\left\langle (e_0)_\eta+\frac{\eta}{4}e_0,w^i\right\rangle e_0
+
q^i(\tau)\left((e_0)_\eta-\frac{\eta}{4}e_0\right)
\right)
\\
&\quad
+\alpha\left(
\mathcal T_i[w^{i+1}]
-
\left\langle e_0,\mathcal T_i[w^{i+1}]\right\rangle e_0
\right),
\end{split}
\]
where
\[
a_i=\frac32+i-k.
\]
We show that, for every compact set,
$
\|F^i(\tau)\|_{L^\infty(K)}
\le C_K(\tau+1)e^{-\tau/2}.
$

First, since \(q^i(\tau)\) is uniformly bounded we have
\[
\left\|
\varepsilon e^{-\tau/2}
q^i(\tau)\left((e_0)_\eta-\frac{\eta}{4}e_0\right)
\right\|_{L^\infty(K)}
\le C_K e^{-\tau/2}.
\]
Moreover,
\[
w^i=q^i e_0+\widehat w^i,
\]
with \(q^i\) uniformly bounded and
\[
\|\widehat w^i(\tau)\|_{L^2}
\le C(\tau+1)e^{-\tau/2}.
\]
Therefore
\[
\left|
\left\langle (e_0)_\eta+\frac{\eta}{4}e_0,w^i\right\rangle
\right|
\le C.
\]
Hence
\[
\left\|
\varepsilon e^{-\tau/2}
\left\langle (e_0)_\eta+\frac{\eta}{4}e_0,w^i\right\rangle e_0
\right\|_{L^\infty(K)}
\le C_K e^{-\tau/2}.
\]

It remains to estimate the projected coupling term. Using
\[
w^{i+1}=q^{i+1}e_0+\widehat w^{i+1}
\]
and the linearity of \(\mathcal T_i\), we write

$$
\mathcal T_i[w^{i+1}]
-
\left\langle e_0,\mathcal T_i[w^{i+1}]\right\rangle e_0
\\
=
q^{i+1}(\tau)
\left(
\mathcal T_i[e_0]
-
\left\langle e_0,\mathcal T_i[e_0]\right\rangle e_0
\right)
+
\left(
\mathcal T_i[\widehat w^{i+1}]
-
\left\langle e_0,\mathcal T_i[\widehat w^{i+1}]\right\rangle e_0
\right).
$$

For the first term, the Taylor expansion gives, uniformly for \(\eta\in K\),
\[
\left|\mathcal T_i[e_0](\tau,\eta)-e_0(\eta)\right|
\le C_K\delta(\tau).
\]
Also,
\[
\left|
\left\langle e_0,\mathcal T_i[e_0]\right\rangle-1
\right|
\le C\delta(\tau).
\]
Therefore
\[
\left\|
\mathcal T_i[e_0]
-
\left\langle e_0,\mathcal T_i[e_0]\right\rangle e_0
\right\|_{L^\infty(K)}
\le C_K\delta(\tau).
\]
Since \(q^{i+1}\) is uniformly bounded and
$
\delta(\tau)\le C(\tau+1)e^{-\tau/2},
$
we obtain
\[
\left\|
q^{i+1}(\tau)
\left(
\mathcal T_i[e_0]
-
\left\langle e_0,\mathcal T_i[e_0]\right\rangle e_0
\right)
\right\|_{L^\infty(K)}
\le C_K(\tau+1)e^{-\tau/2}.
\]

For the second term, choosing $K'$ so that
\(\eta+\delta(\tau)\in K'\) for all \(\eta\in K\) and $\tau$, we have by the induction hypothesis,
\[
\|(\widehat w^{i+1})_\eta(\tau)\|_{L^\infty(K')}
\le C_{K'}(\tau+1)e^{-\tau/2}.
\]
so that the pointwise bound for \(\widehat w^{i+1}\) is
$
|\widehat w^{i+1}(\tau,\eta)|
\le C(\eta)(\tau+1)e^{-\tau/2}.
$
Since \(\eta\in K'\), the factor \(C(\eta)\) is bounded on \(K'\). Hence
\[
\|\widehat w^{i+1}(\tau)\|_{L^\infty(K')}
\le C_{K'}(\tau+1)e^{-\tau/2}.
\]
Therefore, we obtain  on \(K\),
\[
\|\mathcal T_i[\widehat w^{i+1}](\tau)\|_{L^\infty(K)}
\le C_K(\tau+1)e^{-\tau/2}.
\]

Furthermore, we have
\[
\left|
\left\langle e_0,\mathcal T_i[\widehat w^{i+1}]\right\rangle
\right|
\le
\|e_0\|_{L^2}\,
\|\mathcal T_i[\widehat w^{i+1}]\|_{L^2}
\le
C\|\widehat w^{i+1}(\tau)\|_{L^2} \leq  C(\tau+1)e^{-\tau/2},
\]

which implies
$
\left\|
\left\langle e_0,\mathcal T_i[\widehat w^{i+1}]\right\rangle e_0
\right\|_{L^\infty(K)}
\le C_K(\tau+1)e^{-\tau/2}.
$
Consequently,
\[
\left\|
\mathcal T_i[\widehat w^{i+1}]
-
\left\langle e_0,\mathcal T_i[\widehat w^{i+1}]\right\rangle e_0
\right\|_{L^\infty(K)}
\le C_K(\tau+1)e^{-\tau/2}.
\]

Therefore, we obtain
\[
\|F^i(\tau)\|_{L^\infty(K)}
\le C_K(\tau+1)e^{-\tau/2}.
\]

Combined with Lemma \ref{boundedorthogonalremainderlemma}, we obtain by parabolic regularity

$$
\left\|(\widehat{w^{i}})_{\eta}(\tau)\right\|_{L^{\infty}(K)} \leq C_{K} (\tau+1) e^{-\frac{\tau}{2}}
$$

In particular, we get 
$$|\widehat{w^{i}}(\tau, \eta)| \leq C_{K} \eta (\tau+1) e^{-\frac{\tau}{2}}$$
\newline 
since $\widehat{w^{i}}(\tau,0)=w^{i}(\tau,0)-q^{i}(\tau)e_0(0)=0$.

\end{proof}

\medskip
We can now combine those Lemmas to conclude the proof of Proposition \ref{asymptoticslinearlemma}.

\begin{proof}
We use the following facts that we showed in the previous Lemmas
\begin{enumerate}
\item{ $w^{i}(\tau,\eta)=q^{i}(\tau)e_0(\eta)+\widehat{w^{i}}(\tau,\eta)$}
\item{
$q^{i}(\tau)=\frac{\alpha^{k-i}}{(k-i)!} q^{k}(0)+h^{i}(\tau)$ with 
$|h^{i}(\tau)| \leq C\varepsilon + C (\tau+1)e^{-\frac{\tau}{2}}   $
}
\item{
$|\widehat{w^{i}}(\tau,\eta)| \leq C_{K}\eta (\tau+1)e^{-\frac{\tau}{2}}$ for $\eta$ in any compact set. 
}
\item{
$q^{i}(0)=\left \langle w^{i}(0,\eta), e_0(\eta) \right \rangle= \frac{1}{(2\sqrt{\pi})^{\frac{1}{2}}} \int_{0}^{+\infty} \eta p^{i}(0,\eta) \mathrm{d} \eta$

}

\end{enumerate}
We get

\begin{equation} \label{equationforpi}
\begin{split}
&p^{i}(\tau,\eta)=w^{i}(\tau,\eta) e^{-\frac{\eta^2}{8}}e^{\frac{\tau}{2}}= \left(q^{i}(\tau)e_0(\eta)+\widehat{w^{i}}(\tau,\eta)\right) e^{-\frac{\eta^2}{8}}e^{\frac{\tau}{2}} \\
&=\left(\frac{\alpha^{k-i}}{(k-i)!(2\sqrt{\pi})^{\frac{1}{2}}}\int_{0}^{+\infty} \eta p^{k}(0,\eta) \mathrm{d} \eta +h^{i}(\tau)\right)e_0(\eta)e^{-\frac{\eta^2}{8}}e^{\frac{\tau}{2}}+\widehat{w^{i}}(\tau,\eta)e^{-\frac{\eta^2}{8}}e^{\frac{\tau}{2}} \\
&=\left(\frac{\alpha^{k-i}}{(k-i)!(2\sqrt{\pi})^{\frac{1}{2}}}\int_{0}^{+\infty} \eta p^{k}(0,\eta) \mathrm{d} \eta +h^{i}(\tau)\right) \frac{1}{(2\sqrt{\pi})^{\frac{1}{2}}}\eta e^{-\frac{\eta^2}{8}}e^{-\frac{\eta^2}{8}}e^{\frac{\tau}{2}}+\widehat{w^{i}}(\tau,\eta)e^{-\frac{\eta^2}{8}}e^{\frac{\tau}{2}} \\
&=e^{\frac{\tau}{2}} \eta \left( \frac{ \alpha^{k-i} e^{-\frac{\eta^2}{4}}}{(k-i)!2\sqrt{\pi}} \int_{0}^{+\infty} \eta p^{k}(0,\eta) \mathrm{d} \eta + \frac{e^{-\frac{\eta^2}{4}}}{(2 \sqrt{\pi})^{\frac{1}{2}}} h^{i}(\tau) +e^{-\frac{\eta^2}{8}}\frac{\widehat{w^{i}}(\tau,\eta)}{\eta} \right) \\
&=e^{\frac{\tau}{2}}\eta \left( \frac{ \alpha^{k-i} e^{-\frac{\eta^2}{4}}}{(k-i)!2\sqrt{\pi}} \int_{0}^{+\infty} \eta p^{k}(0,\eta) \mathrm{d} \eta + O(\varepsilon) +O\left((\tau+1)e^{-\frac{\tau}{2}}\right) \right) \\
\end{split}
\end{equation}

where $ O(\varepsilon), O\left((\tau+1) e^{-\frac{\tau}{2}}\right)$
denote functions of $\tau$ and $\eta$ which are of that order for $\tau>0$, and for $\eta$ in any fixed compact set. In the original variables, we get that there exists a constant $C$ that depends on the initial conditions and on $\alpha$ and on $t_0$ such that

\begin{equation}
\begin{split}
&V^{i}(t,x)=e^{-\lambda_{*}x}z^{i}(t,x)=e^{- \lambda_{*}x}p^{i}\left(\ln(t+t_0)-\ln(t_0), \frac{x}{\sqrt{t+t_0}}\right) \\
&=\frac{\sqrt{t+t_0}}{\sqrt{t_0}}\frac{x e^{-\lambda_{*}x}}{\sqrt{t+t_0}} \left( C e^{-\frac{x^2}{4(t+t_0)}} + h^{i}(t,x) \right)\\
&=\frac{x e^{-\lambda_{*}x}}{\sqrt{t_0}}\left( C e^{-\frac{x^2}{4(t+t_0)}} + h^{i}(t,x) \right) \\
\end{split}
\end{equation}

where for any $\sigma>0$, 

$$
\limsup _{t \rightarrow+\infty} \sup _{0 \leq x \leq \sigma \sqrt{t+1}}|h^{i}(t, x)|<\frac{C}{2}
$$

since the asymptotics above are valid for all $\tau>0$ and for $\eta=\frac{x}{\sqrt{t+t_0}}$ in any compact set. 

\end{proof}

\subsection{A lower bound for the location of the front of $u$}
Now we proceed by constructing a subsolution for the system \eqref{mysystem}, which will let us prove a lower bound on the location of the front of $u$. Combined with the upper bound we derived above, this will let us establish Theorem $\ref{levelsets}$ on the front location of the $u$ component in system \eqref{mysystem}.

We perform a similar analysis as before, setting $u(t,x)=\widetilde{u}(t, x-\xi_{u}(t)), v(t,x)=\widetilde{v}(t,x-\xi_{v}(t))$ in the frames

\begin{equation}
\xi_{u}(t)=c_{*}t+\frac{1}{\lambda_{*}}\ln \left(t+t_{0}\right), \quad \xi_{v}(t)=c_{*}t
\end{equation}

where $t_{0}>0$ will be chosen big enough later. The equations are

\begin{equation}
\begin{split}
\widetilde{u}_{t}-c_{*}\widetilde{u}_{x} -\frac{1}{\lambda_{*}(t+t_{0})}\widetilde{u}_{x} & =\widetilde{u}_{xx}+f(\widetilde{u})+\alpha \widetilde{v}(t,x+\frac{1}{\lambda_{*}}\ln (t+t_0))(1-\widetilde{u}), \\
\widetilde{v}_{t}-c_{*}\widetilde{v}_{x} & =\widetilde{v}_{xx}+f(\widetilde{v}) .
\end{split}
\end{equation}

Now we consider the linearized Dirichlet problem with zero boundary conditions imposed on these frames

\begin{equation} \label{linearizedonnewframes}
\begin{split}
&U_{t}-c_{*}U_{x} -\frac{1}{\lambda_{*}(t+t_{0})}U_{x}=U_{xx}+f'(0)U+\alpha V(t,x+\frac{1}{\lambda_{*}}\ln (t+t_0)), x>0 \\
&V_{t}-c_{*}V_{x} =V_{xx}+f'(0)V, x>0 . \\
& U(t,0)=0, V(t,0)=0 
\end{split}
\end{equation}

As before, we do a change of variables to get rid of the exponential factor

\begin{equation}
\begin{split}
& U(t, x)=e^{-\lambda_{*} x} z(t, x),V(t, x)=e^{- \lambda_{*} x} w(t, x) . \\
\end{split}
\end{equation}

and get the equations for $z,w$

\begin{equation}
\begin{split}
&z_{t}-\frac{1}{\lambda_{*}(t+t_{0})} (z_{x}-\lambda_{*}z )=z_{x x}+\frac{\alpha }{(t+t_0)} w(t,x+\frac{1}{\lambda_{*}}\ln (t+t_0)), x>0 \\
&w_{t}=w_{x x}, x>0 \\
&z(t,0)=0, w(t,0)=0
\end{split}
\end{equation}

In the self-similar variables the equations for $p,q$ become

\begin{equation}
\begin{split}
& p_{\tau}+L p=-2p+\varepsilon e^{-\tau / 2} p_{\eta} +\alpha q(\tau, \eta+\delta(\tau)), \eta>0 \\
& q_{\tau}+L q =-q, \eta>0 \\
&p(\tau,0)=0, q(\tau,0)=0.
\end{split}
\end{equation}

where we set as before $\delta(\tau)=\frac{\tau+\ln(t_0)}{ \lambda_{*}\sqrt{t_0 e^{\tau}}}$, $\varepsilon=\frac{1} { \lambda_{*} t_{0}^{1 / 2}}$. With the extra change of variables $p(\tau,\eta)=\widetilde{p}(\tau,\eta) e^{-\frac{{\eta}^2}{8}} e^{-\tau}, q(\tau,\eta)=\widetilde{q}(\tau,\eta) e^{-\frac{{\eta}^2}{8}}e^{-\tau}$ to symmetrize the operator we get the equations

\begin{equation}
\begin{split}
&\widetilde{p}_{\tau}+M \widetilde{p}+\widetilde{p}=\varepsilon e^{-\frac{\tau}{ 2}}\left(\widetilde{p}_{\eta}-\frac{\eta}{4} \widetilde{p} \right) + \alpha  \widetilde{q}(\tau,\eta+\delta(\tau)) e^{-\frac{\delta(\tau)^2}{8} } e^{-\frac{\eta \delta(\tau)}{4}}\\
&\widetilde{q}_{\tau}+M \widetilde{q}=0. \\
&\widetilde{p}(\tau,0)=0, \widetilde{q}(\tau,0)=0.
\end{split}
\end{equation}

Our goal is to show that $\widetilde{p}$ is approximately equal to the principal eigenfunction of the operator $M$. This implies then that  $ p(\tau, \eta) \sim \eta e^{-\frac{\eta^2}{4}} e^{-\tau}  $.

We will use of the following Proposition. The proof is omitted because it is almost identical to the proof of Proposition \ref{asymptoticslinearlemma}.

\begin{proposition}\label{lemmanewframes}
Let $(U(t,x),V(t,x))$ be the solution to the linearized Dirichlet problem \eqref{linearizedonnewframes}

with zero boundary conditions and compactly supported initial data $U_{0}, V_{0} \geq 0$ in $(0,+\infty),$

and $U_{0} \not \equiv 0$, $V_{0} \not \equiv 0$. There exists a constant $C>0$ that depends on the initial conditions $U_{0}, V_{0}$ and a constant $t_{0}>0$ that depends on $C$ such that

$$
U(t, x)=x e^{-\lambda_{*}x}(t+t_0)^{-\frac{3}{2}}t_0\left[C e^{-\frac{x^{2}}{4\left(t+t_{0}\right)}}+h(t, x)\right]
$$

where, for each $\sigma>0$,

$$
\limsup _{t \rightarrow+\infty} \sup _{0 \leq x \leq \sigma \sqrt{t+1}}|h(t, x)|<\frac{C}{2}.
$$

\end{proposition}

Now we are ready to prove the Proposition on the lower bound of $u$.

\begin{proposition} \label{lowerbarrier}
There holds

$$
\liminf _{t \rightarrow+\infty} u\left(t, c_{*} t-\frac{1}{2 \lambda_{*}} \ln t+y\right)>0 \quad \text { for all } y \in \mathbb{R}.
$$

and uniformly in $y$ in any compact set. 

Also, for every $\sigma>0$, there exists $\kappa>0$ such that

$$
u\left(t, c_{*} t-\frac{1}{2 \lambda_{*}} \ln t+y\right) \geq \kappa y e^{-\lambda_{*} y} \quad \text { for all } t \geq 1 \text { and } 0 \leq y \leq \sigma \sqrt{t}.
$$
\end{proposition}

\begin{proof}

Our nonlinear system \eqref{mysystem} in the moving frames $\xi_{u}(t)=c_{*}t+\frac{1}{\lambda_{*}}\ln \left(t+t_{0}\right), \xi_{v}(t)=c_{*}t$
is described by the equations
\begin{equation} \label{twotypesnewframenonlinear}
\begin{split}
\widetilde{u}_{t}-c_{*}\widetilde{u}_{x} -\frac{1}{\lambda_{*}(t+t_{0})}\widetilde{u}_{x} & =\widetilde{u}_{xx}+f(\widetilde{u})+\alpha \widetilde{v}(t,x+\frac{1}{\lambda_{*}}\ln (t+t_0))(1-\widetilde{u}), \\
\widetilde{v}_{t}-c_{*}\widetilde{v}_{x} & =\widetilde{v}_{xx}+f(\widetilde{v}) .
\end{split}
\end{equation}

The linearized problem with Dirichlet boundary conditions imposed at $\xi_{u}$ and $\xi_{v}$ is given by

\begin{equation}
\begin{split}
&U_{t}-c_{*}U_{y}-\frac{1}{\lambda_{*}(t+t_0)} U_{y}=U_{yy}+f'(0)U+\alpha V(t, y+ \frac{1}{\lambda_{*}}\ln (t+t_0)), y>0 \\
&V_{t}-c_{*} V_{y}=V_{yy}+f'(0)V, y>0, \\
&U(t,0)=0, V(t,0)=0.
\end{split}
\end{equation}

From Proposition $\ref{lemmanewframes}$ we have that $
U(t, y)=y e^{-\lambda_{*}y}(t+t_0)^{-\frac{3}{2}}t_0\left[C e^{-\frac{y^{2}}{ 4\left(t+t_{0}\right)}}+h(t, y)\right]
$

where, for each $\sigma>0$,

$$
\limsup _{t \rightarrow+\infty} \sup _{0 \leq y \leq \sigma \sqrt{t+1}}|h(t, y)|<\frac{C}{2}
$$

As a result, we have the following bounds

\begin{equation}
\begin{split}
& U(t,y) \leq C(t+t_0)^{-\frac{3}{2}} \text{ for some constant } C \\
&U(t,y) \geq y e^{-\lambda_{*}y} (t+t_0)^{-\frac{3}{2}} \text{ for all } t\geq t_1,  y \in [0, K\sqrt{t}]
\end{split}
\end{equation}

We want to turn $U(t,y)$ into 
a subsolution for the $\tilde{u}$ component of the nonlinear system \eqref{twotypesnewframenonlinear}.

Let us set
$$N(w)= w_{t}-c_{*}w_{y}-\frac{1}{\lambda_{*}(t+t_0)}w_{y}-w_{yy}-f(w) -\alpha \widetilde{v}(t, y+\frac{1}{\lambda_{*}}\ln (t+t_0)) + \alpha \widetilde{v}(t, y+ \frac{1}{\lambda_{*}} \ln (t+t_0) )w
$$
where $\widetilde{v}$ solves the nonlinear equation
$
\widetilde{v}_{t}-c_{*}\widetilde{v}_{y}=\widetilde{v}_{yy}+f(\widetilde{v}).
$

A function $w(t,y)$ is a subsolution for \eqref{twotypesnewframenonlinear} if $N(w) \leq 0$.

$$
N(U)=f'(0)U-f(U)+\alpha V(t,y+ \frac{1}{\lambda_{*}}\ln (t+t_0) ) -\alpha \widetilde{v}(t, y+\frac{1}{\lambda_{*}}\ln (t+t_0)) + \alpha  \widetilde{v}(t, y+ \frac{1}{\lambda_{*}}\ln (t+t_0) ) U
$$
where $V$ solves the linearized equation $
V_{t}-c_{*}V_{y}=V_{yy}+f'(0)V.
$

We notice that $U$ is not a subsolution, but we define 
$\underline{u}(t,y)=U(t,y)\gamma(t)$ with $\gamma(t)$ to be determined.  We compute

\begin{equation}
\begin{split}
&N(\underline{u}(t,y))=U_{t} \gamma(t) +U\gamma'(t) -c_{*}U_{y}\gamma(t)-\frac{1}{\lambda_{*}(t+t_0)}U_{y}\gamma(t)-U_{yy}\gamma(t)-f(U\gamma(t)) \\
&-\alpha \widetilde{v}(t,y+\frac{1}{\lambda_{*}}\ln (t+t_0))\left(1-U\gamma(t)\right) \\
&= U\gamma'(t) +\gamma(t)\left(U_{t}-c_{*}U_{y}-\frac{1}{\lambda_{*}(t+t_0)}U_{y}-U_{yy}-f'(0)U \right)+\left(f'(0)U\gamma(t) -f(U\gamma(t))\right) \\
&-\alpha \widetilde{v}(t, y+\frac{1}{\lambda_{*}}\ln (t+t_0)) 
+ \alpha \widetilde{v}(t, y+ \frac{1}{\lambda_{*}} \ln (t+t_0)) U \gamma(t) \\
&= U \gamma'(t) + \alpha \gamma(t) V(t, y+ \frac{1}{\lambda_{*}}\ln (t+t_0)) +\left(f'(0)U\gamma(t) -f(U\gamma(t))\right) \\
&-\alpha \widetilde{v}(t, y+\frac{1}{\lambda_{*}}\ln (t+t_0)) 
+ \alpha \widetilde{v}(t, y+ \frac{1}{\lambda_{*}} \ln (t+t_0))U \gamma(t) \\
&\leq U \gamma'(t)+ M U^2 \gamma^{2}(t) \\
&+\alpha \left( \gamma(t)V(t, y + \frac{1}{\lambda_{*}}\ln (t+t_0))+\gamma(t) U \widetilde{v}(t, y + \frac{1}{\lambda_{*}}\ln (t+t_0))  - \widetilde{v}(t, y + \frac{1}{\lambda_{*}}\ln (t+t_0)) \right) \\
&\leq U\left( \gamma'(t)+M U \gamma^{2}(t)
\right) \\
&+\alpha \left( \gamma(t)V(t, y + \frac{1}{\lambda_{*}}\ln (t+t_0))+\gamma(t) U \widetilde{v}(t, y + \frac{1}{\lambda_{*}}\ln (t+t_0))  - \widetilde{v}(t, y + \frac{1}{\lambda_{*}}\ln (t+t_0)) \right) \\
\end{split}
\end{equation}

We used that since $f \in C^{2}([0,1])$, there exists $M>0$ such that $f^{\prime}(0) s-f(s) \leq M s^{2}$ in a neighborhood of $0$, i.e. for $s \in\left[0, s_{0}\right)$ for some $s_0>0$. We will choose $\gamma(t)$ so that $N(\underline{u}(t,y)) \leq 0$. 
\newline 
Since we know that $U(t,y) \leq C (t+t_0)^{-\frac{3}{2}}$, for the first term to be negative it suffices to choose $\gamma(t)$ such that 
$$
\gamma^{\prime}(t)=-C_2(t+t_0)^{-\frac{3}{2}}\gamma^{2}, \quad t>0
$$
for some $C_2>0$ big enough. 
The solution to this ODE is given by 

$$\gamma(t)=\frac{\gamma(0)\sqrt{t+t_0}}{\sqrt{t+t_0}+2C_2 \gamma(0)\left(\frac{\sqrt{t+t_0}}{\sqrt{t_0}}-1\right)}$$ which can be bounded from above and from below. 
\newline 
Such a choice of $\gamma$ also works to ensure that the other term in the upper bound of $N(\underline{u})$ is also non-positive. Indeed, as shown in \cite{short_proof}, we have 
$
\gamma(t)V \leq \frac{\widetilde{v}}{2}
$
as long as  $(2\gamma(t))' \leq -C'(t+1)^{-\frac{3}{2}}\gamma^2$ and 
$
\gamma(t) U \widetilde{v} \leq  \frac{\widetilde{v}}{2}
$ since 
\small
$$\gamma(t)U \leq \frac{C \gamma(0)\sqrt{t+t_0}}{(t+t_0)^{\frac{3}{2}}\left(\sqrt{t+t_0}+2C \gamma(0)\left(\frac{\sqrt{t+t_0}}{\sqrt{t_0}}-1\right)\right)} \leq \frac{C \gamma(0)}{(t+t_0)\left(\sqrt{t+t_0}+2C \gamma(0)\left(\frac{\sqrt{t+t_0}}{t_0}-1\right)\right)}\leq \frac{1}{2} $$

by choosing $\gamma(0)$ small enough. Combining these inequalities we get $N(\underline{u}(t,y)) \leq 0$. 

Since we can bound $\gamma(t)$ from below after time $t_1$, we also have that
$$
\underline{u}(t,y)=U(t,y)\gamma(t) \geq c y e^{-\lambda_{*} y}(t+t_0)^{-\frac{3}{2}} \text{ for } t \geq t_1 \text{ and for } y \in [0, K \sqrt{t}] \text{ for any } K>0
$$

Since $\underline{u}$ is a subsolution for \eqref{twotypesnewframenonlinear} we get
$$
\widetilde{u}(t,y) \geq \underline{u}(t,y) 
$$

which implies that for any $K>0$ 
$$
u(t, y+c_{*}t+ \frac{1}{\lambda_{*}} \ln (t+t_0))  \geq \underline{u}(t,y) \geq c y e^{- \lambda_{*}y}(t+t_0)^{-\frac{3}{2}} \text { for all } t \geq t_1 \text { and } y \in[0, K\sqrt{t}]
$$
Fix $\sigma>0$,  for $\xi(t)=\sigma \sqrt{t}$ we get
$$
u(t, c_{*}t+\xi(t)) \geq c(\xi(t)-\frac{1}{\lambda_{*}}\ln (t+t_0)) e^{-\lambda_{*}\xi(t)}(t+t_0) (t+t_0)^{-\frac{3}{2}} \text { for all } t \geq t_1 
$$
i.e.
$$
u(t, c_{*}t+\xi(t)) \geq c(\xi(t)-\frac{1}{\lambda_{*}}\ln (t+t_0)) e^{-\lambda_{*}\xi(t)}(t+t_0)^{-\frac{1}{2}} \text { for all } t \geq t_1 . 
$$

Now we fix $\sigma>0$, and let $\xi(t)=\sigma \sqrt{t}$. We will construct an explicit subsolution for $u(t,x)$ on the interval $[0,c_{*}t+\xi(t)]$. There exists $\widetilde{\delta}>0$ and $T_{1} \geq 2$ such that

$$
u\left(t, c_{*} t+\xi(t)\right) \geq 2\widetilde{\delta} \xi(t) e^{-\lambda_{*} \xi(t)} (t+t_0)^{-\frac{1}{2}} \geq \widetilde{\delta} \sigma e^{- \lambda_{*} \xi(t)}   \text { for all } t \geq T_{1}
$$

The component $u(t,x)$ of the solution to the nonlinear problem $\eqref{mysystem}$ is a supersolution for the classic Fisher-KPP equation with initial condition $v(0,x)=u(0,x)$, since 
$N(u)=u_{t}-u_{xx}-f(u)= \alpha v(1-u)>0 $ and since we have that $v(t,x) \rightarrow 1$  as $t \rightarrow+\infty$ locally uniformly in $x \in \mathbb{R}$, we have that also $u(t,x) \rightarrow 1$  as $t \rightarrow+\infty$ locally uniformly in $x \in \mathbb{R}$. Therefore, there exists $T_{2} \geq T_{1}$ such that $u(t, 0) \geq 1 / 2$ for all $t \geq T_{2}$. 
\newline  
Now let $f_{1}$ be a $C^{1}$ function such that $f_{1} \leq f$ in $[0,1 / 2],  f_{1}(0)=f_{1}(1 / 2)=0, f_{1}^{\prime}(0)=f^{\prime}(0)$ and $f_{1}>0$ on $(0,1 / 2)$. The function $f_{1}$ then satisfies $f_{1}(u) \leq f(u) \leq f_{1}^{\prime}(0) u$ for all $u \in[0,1 / 2]$. 
We have that then for $c_{*}=2\sqrt{f_1'(0)}=2\sqrt{f'(0)}$ there exists a minimal speed travelling front $\widetilde{U}_{c_{*}}\left(x-c_{*} t\right)$ with nonlinearity $f_{1}$ instead of $f$, such that

$$
\widetilde{U}_{c_{*}}^{\prime \prime}+c_{*} \widetilde{U}_{c_{*}}^{\prime}+f_{1}\left(\widetilde{U}_{c_{*}}\right)=0
$$
which satisfies $0<\widetilde{U}_{c_{*}}<\frac{1}{2}$ in $\mathbb{R}$, $\widetilde{U}_{c_{*}}(-\infty)=1 / 2$, $\widetilde{U}_{c_{*}}(+\infty)=0$. The profile $\widetilde{U}_{c_{*}}$ is decreasing in $\mathbb{R}$ with asymptotics 

\begin{equation} \label{twaveasymptotics}
\widetilde{U}_{c_{*}}(s) \sim \widetilde{B} s e^{- \lambda_{*} s} \text { as } s \rightarrow+\infty
\end{equation}

for some constant $\widetilde{B}>0$. 
\newline 

We now take $\gamma>0$ and fix $x_{1} \in \mathbb{R}$ large enough so that $\widetilde{B}(\gamma+1) e^{- \lambda_{*}x_{1}}<\widetilde{\delta}$. There exists $T_{3} \geq T_{2}$ with $\left(\frac{1}{2 \lambda_{*}} \ln (t)+x_{1}\right)<\gamma \xi(t)$ for $t \geq T_{3}$, so it follows that

\small
\begin{equation}
\begin{split}
& \widetilde{U}_{c_{*}}\left(\frac{1}{2 \lambda_{*}} \ln t+\xi(t)+x_{1}\right) \leq  \widetilde{B} \left(\frac{1}{2 \lambda_{*}}\ln(t)+\xi(t)+x_1\right) e^{-\frac{1}{2}\ln(t)}e^{-\lambda_{*} \xi(t)}e^{- \lambda_{*} x_1} \\
& \leq \widetilde{B}(\gamma+1)\xi(t) e^{-\lambda_{*} x_1}t^{-\frac{1}{2}}e^{- \lambda_{*} \xi(t)} 
\leq \widetilde{\delta} \xi(t) e^{- \lambda_{*} \xi(t)} t^{-\frac{1}{2}}= \widetilde{\delta} \sigma e^{- \lambda_{*} \xi(t)} \text { for all } t \geq T_{3}.\\
\end{split}
\end{equation}

We have $\min _{x \in\left[0, c_{*}T_{3}+\xi\left(T_{3}\right)\right]} u\left(T_{3}, x\right)>0$, and $U_{c_{*}}$ is decreasing, $\widetilde{U}_{c_{*}}(+\infty)=0$, so there exists $x_{2} \geq x_{1}$ such that 

\begin{equation} \label{inequalityforT3}
\widetilde{U}_{c_{*}}\left(x-c_{*} T_{3}+\frac{1}{2 \lambda_{*} } \ln T_{3}+x_{2}\right) \leq \widetilde{U}_{c_{*}}\left(-c_{*} T_{3}+\frac{1}{2 \lambda_{*} } \ln T_{3}+x_{2}\right) \leq 
u\left(T_{3}, x\right) \text { for all } x \in\left[0, c_{*}T_{3}+\xi\left(T_{3}\right)\right]
\end{equation}

We will now define a subsolution $\underline{u}$ for the $u$  component of the system \eqref{mysystem},  as follows:

$$
\underline{u}(t, x)=\widetilde{U}_{c_{*}}\left(x-c_{*}t+\frac{1}{2 \lambda_{*}} \ln t+x_{2}\right) \text { for all } t \geq T_{3} \text { and } x \in\left[0, c_{*} t+\xi(t)\right]
$$

It follows from $\eqref{inequalityforT3}$ that $\underline{u}\left(T_{3}, x\right) \leq u\left(T_{3}, x\right)$ for all $x \in\left[0, c_{*}T_{3}+\xi\left(T_{3}\right)\right]$. 
\newline 

Also, since $x_{2} \geq x_{1}$, and $\widetilde{U}_{c_{*}}$ is decreasing, we have

$$
\underline{u}\left(t, c_{*}t+\xi(t)\right)=\widetilde{U}_{c_{*}}\left(\xi(t)+\frac{1}{2 \lambda_{*} } \ln t + x_2 \right) \leq \widetilde{\delta} \sigma e^{- \lambda_{*}\xi(t)}  \leq u\left(t, c_{*} t+\xi(t)\right).
$$

for all $t \geq T_{3}$.

We also have $\underline{u}(t, 0) \leq 1 / 2 \leq u(t, 0)$ for all $t \geq T_{3}$. \\

We set 
$$N(w)=w_{t}-w_{xx}-f(w)-\alpha v (1-w(t,x))$$ 

where $v$ satisfies the classic Fisher-KPP equation

$$v_{t}=v_{xx}+f(v).$$

A function $w$ is a subsolution for the $u$ component of the system \eqref{mysystem} if $N(w) \leq 0$. 

Using that $f_{1}(u) \leq f(u)$ in $[0,\frac{1}{2}]$, that the traveling wave is decreasing, $\widetilde{U}_{c_{*}}^{\prime}<0$, and satisfies the ODE 

$$
\widetilde{U}_{c_{*}}^{\prime \prime}+c_{*} \widetilde{U}_{c_{*}}^{\prime}+f_{1}\left(\widetilde{U}_{c_{*}}\right)=0
$$

we get 

\begin{equation}
\begin{split}
&\underline{u}_{t}(t, x)-\underline{u}_{x x}(t, x)-f(\underline{u}(t, x))-\alpha v(1-\underline{u}(t, x)) \leq\left(-c_{*}+\frac{1}{2 \lambda_{*} t}\right) \widetilde{U}_{c_{*}}^{\prime}(z)-\widetilde{U}_{c_{*}}^{\prime \prime}(z)-f_{1}\left(\widetilde{U}_{c_{*}}(z)\right) \\
&-\alpha v(1-\widetilde{U}_{c_{*}}(z))=\frac{1}{2 \lambda_{*} t} \widetilde{U}_{c_{*}}^{\prime}(z) -\alpha v(1-\widetilde{U}_{c_{*}}(z)) \leq 0\\
\end{split}
\end{equation}
where we set $z=x-c_{*} t+\frac{1}{2 \lambda_{*}} \ln t+x_{2}$.\\

As a result $\underline{u}$ is a subsolution of $u$ for $t \geq T_{3}$ and $x \in\left[0, c_{*} t+\xi(t)\right]$.We get

\begin{equation} \label{subsolutioninequality}
u(t, x) \geq \underline{u}(t, x)=\widetilde{U}_{c_{*}}\left(x-c_{*} t+\frac{1}{2 \lambda_{*}} \ln t+x_{2}\right) \text { for all } t \geq T_{3} \text { and } 0 \leq x \leq c_{*}t+\xi(t).
\end{equation}

The above inequality implies in particular that for any given $y \in \mathbb{R}$,

$$
u\left(t, c_{*} t-\frac{1}{2 \lambda_{*}} \ln t+y\right) \geq \underline{u}\left(t, c_{*} t-\frac{1}{2 \lambda_{*}} \ln t+y\right)=\widetilde{U}_{c_{*}}\left(y+x_{2}\right)>0
$$
for $t$ large enough. This implies

$$
\liminf _{t \rightarrow+\infty} u\left(t, c_{*}t-\frac{1}{2\lambda_{*}} \ln t+y\right)>0 \quad \text { for all } y \in \mathbb{R}
$$

and uniformly in any compact set.

Also, fixing $\sigma>0$, and taking $\xi(t)=\sigma \sqrt{t}$, the inequality $\eqref{subsolutioninequality}$ along with the traveling wave asymptotics $\eqref{twaveasymptotics}$  give us that there exists $\kappa>0$ such that

$$
u\left(t, c_{*} t-\frac{1}{2 \lambda_{*}} \ln t+y\right) \geq \widetilde{U}_{c_{*}}\left(y+x_{2}\right) \geq \kappa y e^{-\lambda_{*} y} \quad \text { for all } t \geq 1 \text { and } 0 \leq y \leq \sigma \sqrt{t}.
$$

This completes the proof of Proposition $\ref{lowerbarrier}$. 

\end{proof}

\begin{Corollary}
For any $m \in(0,1)$, and $E_{m}(t)$ the $m$ level set of the $u$ component of \eqref{mysystem}, there are constants $t_{0}>0$ and $C_{1},C_{2} \in \mathbb{R}$ such that

\begin{equation} \label{maxlevelset}
\max E_{m}(t) \leq c_{*} t-\frac{1}{2 \lambda_{*}} \ln t+C_{1} \text { for all } t \geq t_{0}
\end{equation}

\begin{equation} \label{minlevelset}
\min E_{m}(t) \geq c_{*} t-\frac{1}{2 \lambda_{*}} \ln t+C_{2} \text { for all } t \geq t_{0}
\end{equation}

\end{Corollary}

\begin{proof}
The inequality $\eqref{maxlevelset}$ was proven in \eqref{maxlevelsetcascading}, taking $k=2$ and $i=1$. The inequality $\eqref{minlevelset}$ follows from $\eqref{subsolutioninequality}$  by replicating the proof above with a nonlinearity $f_1$ that vanishes at $0$  and at $\frac{1+m}{2}$ (instead of vanishing at $0$ and $\frac{1}{2}$ as done above). 
\end{proof}
With this corollary the proof of Theorem $\ref{levelsets}$ is complete. 

\section{Convergence to the minimal speed Fisher-KPP wave for the cascading system} \label{Section 4}

In this section we will prove Theorem \ref{convergencetothewavecascading} for the cascading system 

\begin{equation} 
\begin{split}
&v^{i}_{t}= v^{i}_{xx}+v^{i}-(v^{i})^{2}+\alpha v^{i+1}\left(1-v^{i}\right) \text{ for }  i \in \{1,\ldots,k-1\}, \quad{ t>0, x \in \mathbb{R},} \\
& v^{k}_{t}= v^{k}_{xx}+v^{k}-(v^{k})^{2},  \quad{ t>0, x \in \mathbb{R}}.\\
\end{split}
\end{equation} 

As before, we go into moving frames $\xi_{i}(t)=2 t-\left(\frac{3}{2}+i-k\right) \ln(t)$. Setting $v^{i}(t,x)=\widetilde{v^{i}}(t,x-\xi_{i}(t))$ gives us 

\begin{equation} \label{logmovingframecascadingnot_0}
\begin{split}
&(\widetilde{v^{i}})_{t}-2(\widetilde{v^{i}})_{x}+\frac{\left( \frac{3}{2}+i-k\right)}{t}(\widetilde{v^{i}})_{x} = (\widetilde{v^{i}})_{xx}+\widetilde{v^{i}}-(\widetilde{v^{i}})^{2} +\alpha \widetilde{v^{i+1}}(t,x+\ln(t)) \left(1-\widetilde{v^{i}}\right)  \text{ for }  i \leq k-1 \\
& (\widetilde{v^{k}})_{t}-2(\widetilde{v^{k}})_{x}+\frac{\frac{3}{2}}{t}(\widetilde{v^{k}})_{x}= (\widetilde{v^{k}})_{xx}+\widetilde{v^{k}}-(\widetilde{v^{k}})^{2}.\\
\end{split}
\end{equation} 

We will do the same transformations as above, but now apply them directly to the nonlinear system. We get rid of exponential factor $\widetilde{v^{i}}(t, x)=e^{-x} z^{i}(t, x)$
and obtain the system for $z^{i}(t,x)$

\begin{equation} \label{z_system}
\begin{split}
&z^{i}_{t}+\frac{(\frac{3}{2}+i-k)}{t} (z^{i}_{x}-z^{i} )=z^{i}_{x x}-e^{-x}(z^{i})^2+\frac{\alpha }{t }z^{i+1}(t,x+\ln t) (1-e^{-x}z^{i}) \\
&z^{k}_{t}+\frac{3}{2 t}(z^{k}_{x}-z^{k})=z^{k}_{x x}-e^{-x}(z^{k})^2 \\
\end{split}
\end{equation}

Then we go into the self-similar variables $\tau=\ln t, \eta=\frac{x}{\sqrt{t}}$, and set $z^{i}(t,x)=e^{\frac{\tau}{2}}p^{i}(\tau, \eta)$, to get the equations for $p^{i}(\tau, \eta)$

\begin{equation} \label{selfsimilarnonlinear}
\begin{split}
& p^{i}_{\tau}+L p^{i}+ (k-i) p^{i}+\left( \frac{3}{2}+i-k \right)e^{-\frac{\tau}{2}}p^{i}_{\eta}+e^{\frac{3\tau}{2}-\eta e^{\frac{\tau}{2}}}(p^{i})^{2}= \alpha p^{i+1}\left(\tau, \eta+\tau e^{-\frac{\tau}{2}}\right)\left( 1- e^{ \frac{\tau}{2}-\eta e^{\frac{\tau}{2}}}p^{i}\right) \\
&p^{k}_{\tau}+ L p^{k} +\frac{3}{2}e^{-\frac{\tau}{2}}p^{k}_{\eta}+e^{\frac{3\tau}{2}-\eta e^{\frac{\tau}{2}}} (p^{k})^{2}=0 \\
\end{split}
\end{equation}

where $$
L v=-v_{\eta \eta}-\frac{\eta}{2} v_{\eta}-v
$$

Through this transformation we can see how the Dirichlet problem appears naturally. 

For $
\eta \ll-\tau e^{-\tau / 2},
$ the nonlinear terms in the left-hand and right-hand side of the equation for $p^{i}$ become exponentially large. On the other hand, for
$
\eta \gg \tau e^{-\tau / 2}
$ the quadratic term $e^{\frac{3\tau}{2}-\eta e^{\frac{\tau}{2}}}(p^{i})^{2}$ is very small and 
$\left( 1- e^{ \frac{\tau}{2}-\eta e^{\frac{\tau}{2}}}p^{i}\right)$  is very close to $1$, so the linearized Dirichlet problem with the forcing term $\alpha p^{i+1}$ is a good proxy.
\newline 

In the self-similar variables the linearized Dirichlet problem is

\begin{equation} \label{linearizedselfsimilardirichletforP}
\begin{split}
& P^{i}_{\tau}+L P^{i}+ (k-i) P^{i}+\left( \frac{3}{2}+i-k \right)e^{-\frac{\tau}{2}}P^{i}_{\eta}= \alpha P^{i+1}\left(\tau, \eta+\tau e^{-\frac{\tau}{2}}\right), \eta>0\\
&P^{k}_{\tau}+ L P^{k} +\frac{3}{2}e^{-\frac{\tau}{2}}P^{k}_{\eta}=0, \eta>0 \\
&P^{i}(\tau,0)=0
\end{split}
\end{equation}

Through this change of variables, we can see how a particular translation of the wave will be chosen. For the linearized Dirichlet problem in the self-similar variables, we can show that as $\tau \rightarrow+\infty$,

\begin{equation} \label{asymptoticslinearp}
P^{i}(\tau,\eta) \sim \alpha^{i}_{\infty} \eta e^{-\eta^{2} / 4}, \quad \eta>0
\end{equation}

\noindent for some $\alpha^{i}_{\infty}>0$. 

Then approximating the nonlinear problem $\ref{selfsimilarnonlinear}$  by the linearized problem $\ref{linearizedselfsimilardirichletforP}$ in the moving frame gives us 
\newline

\begin{equation} \label{uapprox}
\widetilde{v^{i}}(t, x)=e^{-x} z^{i}(t, x) \sim e^{-x} e^{\tau/2} p^{i}(\tau,\eta) \sim e^{-x} e^{\tau/2} P^{i}(\tau,\eta) \sim e^{-x} e^{\tau / 2} \alpha^{i}_{\infty} \eta e^{-\eta^{2} / 4}=\alpha^{i}_{\infty} x e^{-x} e^{-x^{2} /(4 t)},
\end{equation}

\noindent at least for $x$ of the order $O(\sqrt{t})$.

If we accept that $\widetilde{v^{i}}$ converges to a translate $x^{i}_{\infty}$ of $U_{c_{*}}$, then for large $x$ we have

\begin{equation} \label{twaveapprox}
\tilde{v^{i}}(t, x) \sim U_{c_{*}}\left(x+x^{i}_{\infty}\right) \sim x e^{-x-x^{i}_{\infty}}
\end{equation}

Comparing \eqref{uapprox} and \eqref{twaveapprox} gives us $
x^{i}_{\infty}=-\ln \alpha^{i}_{\infty}
$ and this determines the unique translation. 

However, as it stands, this argument doesn't hold because each of these asymptotics corresponds to a different range of $x$. 
The asymptotics \eqref{asymptoticslinearp} are valid for $\eta \sim O(1)$, so correspondingly for $x \sim O(\sqrt{t})$. However, the convergence to the wave happens at scales of $x \sim O(1)$ around the front, and the traveling wave asymptotics $\eqref{twaveapprox}$ assume that $x$ is large, but finite. The scale  $x \sim O(1)$ corresponds to $\eta \sim O(e^{-\frac{\tau}{2}})$, for which we have no information from the self-similar variables.  

Establishing a similar result for that scale $\eta \sim O(e^{-\frac{\tau}{2}})$ is hopeless, as the error term in the equation $\eqref{selfsimilarnonlinear}$ is then of order at least $\eta$, while our desired approximation is $\eta e^{-\frac{\eta^2}{4}}$. To overcome this issue, we prove stabilization at an intermediate scale  $x \sim O\left(t^{\gamma}\right)$ with a small $\gamma>0$. We are essentially using $x \sim O(t^{\gamma})$ for small $\gamma$ as our transition region to pass from $x \sim O(\sqrt{t})$ to $x \sim O(1)$. \\
For the stabilization at this intermediate scale we need to show that  

$$
p^{i}(\tau, \eta) \sim \alpha^{i}_{\infty} \eta e^{-\eta^{2} / 4}
$$

\noindent for $\eta \sim e^{-(\frac{1}{2}-\gamma) \tau}$ with a small $0<\gamma << \frac{1}{2}$. We will show that this stabilization is sufficient to ensure the stabilization on the scale $x \sim O(1)$ and convergence to a unique wave. This is the strategy followed in \cite{convergencetothewave}.

To show stabilization at the spatial scales $x \sim O\left(t^{\gamma}\right)$, we need the following generalization of Lemma 4.2 from \cite{convergencetothewave}.

\begin{lemma} \label{stabilization}
Suppose that $\widetilde{v^{i}}$ satisfy the equations
\begin{equation} \label{logmovingframecascading}
\begin{split}
&(\widetilde{v^{i}})_{t}-2(\widetilde{v^{i}})_{x}+\frac{\left( \frac{3}{2}+i-k\right)}{t}(\widetilde{v^{i}})_{x} = (\widetilde{v^{i}})_{xx}+\widetilde{v^{i}}-(\widetilde{v^{i}})^{2} +\alpha \widetilde{v^{i+1}}(t,x+\ln(t)) \left(1-\widetilde{v^{i}}\right)  \text{ for }  i \leq k-1 \\
& (\widetilde{v^{k}})_{t}-2(\widetilde{v^{k}})_{x}+\frac{\frac{3}{2}}{t}(\widetilde{v^{k}})_{x}= (\widetilde{v^{k}})_{xx}+\widetilde{v^{k}}-(\widetilde{v^{k}})^{2}\\
\end{split}
\end{equation} 
\noindent There exist constants $\alpha^{i}_{\infty} >0$ with the following property. For any $\gamma>0$ and all $\varepsilon>0$ we can find $T_{\varepsilon}$ so that for all $t>T_{\varepsilon}$ we have

\begin{equation}
\begin{split}
&\left| \widetilde{v^{i}}\left(t, x_{\gamma}\right)-\alpha^{i}_{\infty} x_{\gamma} e^{-x_{\gamma}} e^{-\frac{x_{\gamma}^{2}}{4 t}}\right| \leq \varepsilon x_{\gamma} e^{-x_{\gamma}} e^{-\frac{x_{\gamma}^{2}}{6 t}} \\
\end{split}
\end{equation}

with $x_{\gamma}=t^{\gamma}$.
\end{lemma}
This has already been proven for $\tilde{v^{k}}$ in \cite{convergencetothewave} and we will prove it for $\tilde{v^{i}}$, as a direct consequence of the following Lemma.

\begin{lemma} \label{nonlinearasymptotics}
Let $\{p^{i}(\tau, \eta)\}_{i=1}^k$ be the solution to the following nonlinear problem on $\mathbb{R}$, in self-similar variables

\begin{equation} \label{selfsimilarnonlinear2}
\begin{split}
& p^{i}_{\tau}+L p^{i}+ (k-i) p^{i}+\left( \frac{3}{2}+i-k \right)e^{-\frac{\tau}{2}}p^{i}_{\eta}+e^{\frac{3\tau}{2}-\eta e^{\frac{\tau}{2}}}(p^{i})^{2}= \alpha p^{i+1}\left(\tau, \eta+\tau e^{-\frac{\tau}{2}}\right)\left( 1- e^{ \frac{\tau}{2}-\eta e^{\frac{\tau}{2}}}p^{i}\right) \\
&p^{k}_{\tau}+ L p^{k} +\frac{3}{2}e^{-\frac{\tau}{2}}p^{k}_{\eta}+e^{\frac{3\tau}{2}-\eta e^{\frac{\tau}{2}}} (p^{k})^{2}=0 \\
\end{split}
\end{equation}

\noindent where $L v= -v_{\eta \eta}-\frac{\eta v_{\eta}}{2}-v$, with initial conditions $p^{i}(0, \eta)=p^{i}_{0}(\eta)$ such that $p^{i}_{0}(\eta)=0$ for all $\eta>M$, with some $M>0$, and $p^{i}_{0}(\eta)=\left(e^{\eta}\right)$ for $\eta<0$. There exist $\alpha^{i}_{\infty}>0$ and functions $h^{i}(\tau)$ such that $\lim _{\tau \rightarrow+\infty} h^{i}(\tau)=0$, and such that we have, for any $\gamma^{\prime} \in(0,1 / 2)$

$$
p^{i}(\tau, \eta)=\left(\alpha^{i}_{\infty}+h^{i}(\tau)\right) \eta_{+} e^{-\eta^{2} / 4}+R^{i}(\tau, \eta) e^{-\eta^{2} / 6}, \quad \eta \in \mathbb{R}
$$

with

$$
|R^{i}(\tau, \eta)| \leq C_{\gamma^{\prime}} e^{-\left(1 / 2-\gamma^{\prime}\right) \tau}
$$

and where $\eta_{+}=\max (0, \eta)$.

\end{lemma}

Now we show that Lemma $\ref{stabilization}$ is indeed an immediate consequence of Lemma $\ref{nonlinearasymptotics}$. Since  $
\widetilde{v^{i}}(t, x)=e^{-x} \sqrt{t} p^{i}\left(\ln t, \frac{x}{\sqrt{t}}\right)
$, Lemma \ref{nonlinearasymptotics} implies, with $x_{\gamma}=t^{\gamma}$,

$$
\begin{aligned}
e^{x_{\gamma}} \widetilde{v^{i}}\left(t, x_{\gamma}\right) & -\alpha^{i}_{\infty} x_{\gamma} e^{-x_{\gamma}^{2} /(4 t)}=\sqrt{t} p^{i}\left(\ln t, \frac{x_{\gamma}}{\sqrt{t}}\right)-\alpha^{i}_{\infty} x_{\gamma} e^{-x_{\gamma}^{2} /(4 t)} \\
& =h^{i}(\ln t) x_{\gamma} e^{-x_{\gamma}^{2} /(4 t)}+\sqrt{t} R^{i}\left(\ln t, \frac{x_{\gamma}}{\sqrt{t}}\right) e^{-x_{\gamma}^{2} /(6 t)}
\end{aligned}
$$

We now take $T_{\varepsilon}$ so that $|h^{i}(\ln t)|<\varepsilon$ for all $t>T_{\varepsilon}$. For the second term on the right side we write

$$
\left|R^{i}\left(\ln t, \frac{x_{\gamma}}{\sqrt{t}}\right)\right| \sqrt{t} e^{-x_{\gamma}^{2} /(6 t)} \leq C t^{\gamma^{\prime}} e^{-x_{\gamma}^{2} /(6 t)} \leq \varepsilon x_{\gamma} e^{-x_{\gamma}^{2} /(6 t)}
$$

for $t>T_{\varepsilon}$ sufficiently large, as soon as $\gamma^{\prime}<\gamma$. 
Therefore we obtain Lemma $\ref{stabilization}$ as a direct consequence of Lemma \ref{nonlinearasymptotics} .  

\subsection{Proof of Lemma $\ref{nonlinearasymptotics}$}

In order to prove Lemma $\ref{nonlinearasymptotics}$, we will  construct an upper and a lower barrier for $p^{i}$ with the correct behaviors.

We will need the following Lemmas.

\begin{lemma} \label{linearunperturbedasymptotics3}
Let $\{P^{i}\}_{i=1}^{k}$ be the solutions to the unperturbed system
\begin{equation} \label{selfsimilarlinearunperturbed}
\begin{split}
& P^{i}_{\tau}+L P^{i}+ (k-i) P^{i}= \alpha P^{i+1}\left(\tau, \eta \right), 1 \leq i \leq k-1, \eta>0 \\
&P^{k}_{\tau}+ L P^{k} =0, \eta>0 \\
& P^{i}(\tau,0)=0.
\end{split}
\end{equation}
Then we have that 

\begin{equation} \nonumber
\begin{split}
&P^{i}(\tau,\eta)=\eta e^{-\frac{\eta^2}{4}}\left( \frac{\alpha^{k-i}}{(k-i)! 2\sqrt{\pi}}\int_{0}^{+\infty} \eta P^{k}(0,\eta) d \eta +\frac{C^{i}_0 e^{-\frac{\tau}{2}}}{(2\sqrt{\pi})^{\frac{1}{2}}} \right)+ e^{-\frac{\tau}{4}} e^{-\frac{\eta^2}{6}}R^{i}(\tau,\eta).
\end{split}
\end{equation}

as $\tau \rightarrow \infty$, with $C^{i}_0$ a constant that depends on the initial conditions, 
and $R^{i}(\tau,\eta)$ is bounded. 
\end{lemma}

\begin{lemma} \label{differentialinequalities}
Consider $\omega \in(0,\frac{1}{2})$ and $G_{\overline{z^{i}}}(\tau, \eta)$ smooth, bounded, and compactly supported on $\mathbb{R}_{+}$. Let $z^{i}(\tau,\eta)$ satisfy

$$
\begin{aligned}
&| z^{i}_{\tau}+L z^{i} +(k-i)z^{i} - \alpha z^{i+1}(\tau, \eta + \tau e^{-\frac{\tau}{2}})| \leq \varepsilon e^{-\omega \tau}\left(\left|z^{i}_{\eta}\right|+|z^{i}|+G_{z^{i}}(\tau, \eta)\right), 1 \leq i \leq k-1, \eta>0\\
&|z^{k}_{\tau}+ L z^{k}| \leq   \varepsilon e^{-\omega \tau}\left(\left|z^{k}_{\eta}\right|+|z^{k}|+G_{z^{k}}(\tau, \eta)\right), \eta>0, \\
& z^{i}(\tau,0)=0 \\
\end{aligned}
$$

\noindent with initial conditions $z^{i}_{0}(\eta)$ such that $z^{i}_{0}(\eta) e^{\eta^{2} / 6} \in L^{2}\left(\mathbb{R}_{+}\right)$. 

\medskip
There exists $\varepsilon_{0}>0$,$C^{i}>0$ (depending on the initial conditions) such that, for all $0<\varepsilon<\varepsilon_{0}$, we have for $1 \leq i \leq k$
\smaller
$$
\begin{aligned}
& z^{i}(\tau, \eta)=\eta \left[e^{-\frac{\eta^2}{4}}\left( \frac{\alpha^{k-i}}{(k-i)! 2\sqrt{\pi}}\int_{0}^{+\infty} \eta z^{k}(0,\eta) d \eta +C^{i} (\tau+1) e^{-\tau/2} +\varepsilon R_{1}^{i}(\tau,\eta) \right)
 + e^{-\omega \tau} R_{2}^{i}(\tau, \eta) e^{-\frac{\eta^{2}}{6}} \right] \\
\end{aligned}
$$

where $\left\|R^{i}_{1,2}(\tau, \cdot)\right\|_{C^{3}}$ are bounded.
\end{lemma}

The proofs of these two lemmas are omitted, as they follow from modifications of the proof of Proposition \ref{asymptoticslinearlemma}.

\medskip
We need to deduce the approximate Dirichlet conditions. We recall that $p^{i}$ are related to the solution $\widetilde{v^{i}}$ of the original system in the log shifted moving frame via  

\begin{equation}
p^{i}(\tau, \eta)=\widetilde{v^{i}}\left(e^{\tau}, \eta e^{\frac{\tau}{2}}\right) e^{-\frac{\tau}{2}+\eta e^{\frac{\tau}{2}}} 
\end{equation}

The trivial a priori bounds $0<\widetilde{v^{i}}(t, x)<1 $ then imply that we have

\begin{equation}
0<p^{i}(\tau, \eta)< e^{-\frac{\tau}{2}+\eta e^{\frac{\tau}{2}}} \\
\end{equation}

and in particular, we have

\begin{equation}
0<p^{i}\left(\tau,-e^{-(1 / 2-\gamma) \tau}\right) \leq e^{-e^{\gamma \tau}} \\
\end{equation}
Moreover, since
\begin{equation}
\begin{split}
& p^{i}_{\tau}(\tau, \eta)=\widetilde{v^{i}}_{t}\left(e^{\tau}, \eta e^{\frac{\tau}{2}}\right) e^{\frac{\tau}{2}+\eta e^{\frac{\tau}{2}}}+\frac{\eta}{2} \widetilde{v^{i}}_{x}\left(e^{\tau}, \eta e^{\frac{\tau}{2}}\right) e^{\eta e^{\frac{\tau}{2}}} \\
&+\left(\frac{\eta}{2} e^{\frac{\tau}{2}}-\frac{1}{2}\right) \widetilde{v^{i}}\left(e^{\tau}, \eta e^{\frac{\tau}{2}}\right) e^{-\frac{\tau}{2}+\eta e^{\frac{\tau}{2}}} \\
\end{split}
\end{equation}

$$
\begin{aligned}
p^{i}_{\tau}\left(\tau,-e^{-(1 / 2-\gamma) \tau}\right) & =\widetilde{v^{i}}_{t}\left(e^{\tau},-e^{\gamma \tau}\right) e^{\tau / 2-e^{\gamma \tau}}-\frac{1}{2} e^{-(1 / 2-\gamma) \tau} \widetilde{v^{i}}_{x}\left(e^{\tau},-e^{\gamma \tau}\right) e^{-e^{\gamma \tau}} \\
& -\frac{1}{2}\left(e^{\gamma \tau}+1\right) \widetilde{v^{i}}\left(e^{\tau},-e^{\gamma \tau}\right) e^{-\tau / 2-e^{\gamma \tau}} \\
& =O\left(e^{-\gamma e^{\gamma \tau}}\right) \\
\end{aligned}
$$

for $\gamma>0$ sufficiently small. Thus, the solution to $\eqref{selfsimilarnonlinear}$ satisfies

\begin{equation} \label{boundaryconditions}
\begin{split}
0<p^{i}\left(\tau,-e^{-(1 / 2-\gamma) \tau}\right) & \leq e^{-e^{\gamma \tau}} \\
\left|p^{i}_{\tau}\left(\tau,-e^{-(1 / 2-\gamma) \tau}\right)\right| & \leq C e^{-\gamma e^{\gamma \tau}}.\\
\end{split}
\end{equation}

\subsection{Construction of an upper barrier}

We will construct an upper barrier by considering the solution $\overline{p^{i}}$ to: 

\begin{equation} \label{selfsimilarupperbarrier}
\begin{split}
& \overline{p^{i}}_{\tau}+L \overline{p^{i}}+ (k-i) \overline{p^{i}}+\left( \frac{3}{2}+i-k \right)e^{-\frac{\tau}{2}}\overline{p^{i}}_{\eta}= \alpha \overline{p^{i+1}}\left(\tau, \eta+\tau e^{-\frac{\tau}{2}}\right), \eta> -e^{-(1 / 2-\gamma) \tau} \\
&\overline{p^{k}}_{\tau}+ L \overline{p^{k}} +\frac{3}{2}e^{-\frac{\tau}{2}}\overline{p^{k}}_{\eta}=0, \eta> -e^{-(1 / 2-\gamma) \tau} \\
& \overline{p^{i}}\left(\tau,-e^{-(1 / 2-\gamma) \tau}\right) =e^{-e^{\gamma \tau}}
\end{split}
\end{equation}

\noindent with compactly supported initial conditions $\overline{p^{i}}_0$ chosen so that $\overline{p^{i}}_0(\eta) \geq v^{i}(1,\eta)e^{\eta}$.

We will suppose that $p^{i+1}(\tau, \eta) \leq \overline{p^{i+1}}(\tau, \eta)$ and show that then it holds for $i$, i.e. $p^{i}(\tau, \eta) \leq \overline{p^{i}}(\tau, \eta).$ 

This was shown in \cite{convergencetothewave}, for $i=k$, so we will then be able to conclude that $p^{i}(\tau, \eta) \leq \overline{p^{i}}(\tau, \eta)$ for $1 \leq i \leq k$, for all $\tau>0, \eta > -e^{-(\frac{1}{2}-\gamma)\tau}.$

We check that this system gives us an upper barrier for $p^{i}$. Indeed, if we set  

$$N(w)=w_{\tau}+Lw+(k-i)w+\left( \frac{3}{2}+i -k \right)e^{-\frac{\tau}{2}}w_{\eta}+e^{\frac{3\tau}{2}-\eta e^{\frac{\tau}{2}}}w^2-\alpha p^{i+1}(\tau, \eta+\tau e^{-\frac{\tau}{2}})(1-e^{\frac{\tau}{2}-\eta e^{\frac{\tau}{2}}}w)$$

we have that $N(p^{i})=0$, while 

$$N(\overline{p^{i}})=e^{\frac{3\tau}{2}-\eta e^{\frac{\tau}{2}}}\overline{p^{i}}^2+\alpha \overline{p^{i+1}}(\tau, \eta+\tau e^{-\frac{\tau}{2}})-\alpha p^{i+1}(\tau, \eta + \tau e^{-\frac{\tau}{2}})+ \alpha p^{i+1}(\tau, \eta +\tau e^{-\frac{\tau}{2}}) e^{\frac{\tau}{2}-\eta e^{\frac{\tau}{2}}} \overline{p^{i}} \geq 0$$

since $$p^{i+1}(\tau, \eta +\tau e^{-\frac{\tau}{2}}) \leq \overline{p^{i+1}}(\tau, \eta+ \tau e^{-\frac{\tau}{2}}).$$

Combined with \eqref{boundaryconditions} we obtain that $\overline{p^{i}}$ is an upper barrier for $p^{i}$ and we must have $p^{i}(\tau,\eta) \leq \overline{p^{i}}(\tau,\eta)$ for all $\tau>0, \eta > - e^{-(\frac{1}{2}-\gamma)\tau}.$

We make a change of variables to shift the moving boundary to a fixed boundary at $0$

\begin{equation}
\begin{split}
& \overline{p^{i}}(\tau, \eta)=\overline{z^{i}}\left(\tau, \eta+e^{-(1 / 2-\gamma) \tau}\right)+e^{-e^{\gamma \tau}} g^{i}\left(\eta+e^{-(1 / 2-\gamma) \tau}\right) \\
\end{split}
\end{equation}

where $g^{i}(\eta)$ are smooth monotonic functions such that $g^{i}(\eta)=1$ for $0 \leq \eta<1$ and $g^{i}(\eta)=0$ for $\eta>2$. 

For the functions $\overline{z^{i}}$ we have the calculations

\begin{equation} \nonumber
\begin{split}
& \overline{p^{i}}_{\tau}=\overline{z^{i}}_{\tau} -(\frac{1}{2}-\gamma)e^{-(\frac{1}{2}-\gamma)\tau}\overline{z^{i}}_{\eta}-\gamma e^{\gamma \tau}e^{-e^{\gamma \tau}}g^{i}-e^{-e^{\gamma \tau}}(\frac{1}{2}-\gamma)e^{-(\frac{1}{2}-\gamma)\tau}(g^{i})' \\
& \overline{p^{i}}_{\eta}=\overline{z^{i}}_{\eta}+e^{-e^{\gamma \tau}}(g^{i})' \\
&\overline{p^{i}}_{\eta \eta} = \overline{z^{i}}_{\eta \eta} + e^{-e^{\gamma \tau}}(g^{i})'' \\
&\overline{p^{i+1}}(\tau, \eta + \tau e^{-\frac{\tau}{2}})=\overline{z^{i+1}}(\tau, \eta + \tau e^{-\frac{\tau}{2}}+e^{-(\frac{1}{2}-\gamma)\tau})+ e^{-e^{\gamma \tau}}g^{i+1}(\eta+\tau e^{-\frac{\tau}{2}}+e^{-(\frac{1}{2}-\gamma)\tau})
\end{split}
\end{equation}

so the equations for $\overline{z^{i}}$ are

$$
\begin{aligned}
& \overline{z^{i}}_{\tau}+L \overline{z^{i}}+(k-i)\overline{z^{i}}+\left( \left(\frac{3}{2}+i-k \right)e^{-\frac{\tau}{2}}+\gamma e^{-(\frac{1}{2}-\gamma)\tau}\right)\overline{z^{i}}_{\eta}=\alpha \overline{z^{i+1}}(\tau, \eta + \tau e^{-\frac{\tau}{2}})+ e^{-e^{\gamma \tau}}G_{\overline{z^{i}}}(\tau,\eta), \eta>0 \\
&\overline{z^{k}}_{\tau}+ L \overline{z^{k}} + \left(\frac{3}{2}e^{-\frac{\tau}{2}}+\gamma e^{-(\frac{1}{2}-\gamma)\tau} \right) \overline{z^{k}}_{\eta}  = e^{-e^{\gamma \tau}}G_{\overline{z^{k}}}(\tau,\eta), \eta>0 \\
& \overline{z^{i}}(\tau,0)=0 \\
\end{aligned}
$$

where $G_{\overline{z^{i}}},G_{\overline{z^{k}}}$ are compactly supported and smooth. We have set 

\begin{equation}
\begin{split}
& G_{\overline{z^{i}}}(\tau,\eta)= \gamma e^{\gamma \tau}g^{i}-\gamma e^{-(\frac{1}{2}-\gamma)\tau}(g^{i})'+\frac{\eta}{2}(g^{i})'+(g^{i})''-(k-i)g^{i}-\left(\frac{3}{2}+i-k\right)e^{-\frac{\tau}{2}}(g^{i})'+\alpha g^{i+1}(\eta+\tau e^{-\frac{\tau}{2}}) \\
& G_{\overline{z^{k}}}(\tau,\eta)= \gamma e^{\gamma \tau}g^{i+1}-\gamma e^{-(\frac{1}{2}-\gamma)\tau}(g^{i+1})'+\frac{\eta}{2}(g^{i+1})'+(g^{i+1})''-\frac{3}{2}e^{-\frac{\tau}{2}}(g^{i+1})' \\
\end{split}
\end{equation}

We can now apply Lemma \ref{differentialinequalities} as follows. 
 
 For any $\varepsilon>0$, we may choose $T$ sufficiently large, and $\omega \in(0,\frac{1}{2}-\gamma)$ so that

$$
\begin{aligned}
& |\overline{z^{i}}_{\tau}+L \overline{z^{i}}+(k-i)\overline{z^{i}}-\alpha \overline{z^{i+1}}(\tau, \eta + \tau e^{-\frac{\tau}{2}})| \leq \varepsilon e^{-\omega(\tau-T)}(|\overline{z^{i}}_{\eta}|+|G_{\overline{z^{i}}}|)  \\
&|\overline{z^{k}}_{\tau}+ L \overline{z^{k}}| \leq  \varepsilon e^{-\omega(\tau-T)}(|\overline{z^{k}}_{\eta}|+|G_{\overline{z^{k}}}|)  \\
& \overline{z^{i}}(\tau,0)=0 \\
\end{aligned}
$$

Lemma \ref{differentialinequalities} then gives us for $\tau>T$, 

$$
\begin{aligned}
& \overline{z^{i}}(\tau, \eta)=\eta \left(e^{-\frac{\eta^2}{4}}\left( \frac{\alpha^{k-i}}{(k-i)! 2\sqrt{\pi}}\int_{0}^{+\infty} \eta \overline{z^{k}}(T,\eta) d \eta +C e^{-(\tau-T)} +\varepsilon R_{1}^{i}(\tau,\eta) \right)+ e^{-\frac{\omega(\tau-T)}{2}} e^{-\frac{\eta^2}{6}}R_2^{i}(\tau,\eta) \right) \\
&\text{ where }  \left\|R^{i}_{1,2}(\tau, \cdot)\right\|_{C^{3}} \leq C'.
\end{aligned}
$$

It was shown in \cite{convergencetothewave} that with a suitable choice of $\overline{z^{k}}(0,\eta)$ (perhaps after increasing it further), we have

$$
\int_{0}^{\infty} \eta \overline{z^{k}}(\tau, \eta) d \eta \geq 1, \text { for all } \tau>0.
$$

Then it follows that that there exists a sequence $\tau_{n} \rightarrow+\infty, C>0$ and functions $\overline{P^{i}}_{\infty}(\eta)$ such that for $1 \leq i \leq k$

$$
C^{-1} \eta e^{-\eta^{2} / 4} \leq \overline{P^{i}}_{\infty}(\eta) \leq C \eta e^{-\eta^{2} / 4}
$$

and

$$
\lim _{n \rightarrow+\infty} e^{\eta^{2} / 8}\left| \overline{z^{i}}\left(\tau_{n}, \eta\right)-\overline{P^{i}}_{\infty}(\eta)\right|=0
$$

uniformly in $\eta$ on the half-line $\eta \geq 0$. The same follows for the function $\overline{p^{i}}(\tau, \eta)$  

$$
\lim _{n \rightarrow+\infty} e^{\eta^{2} / 8}\left|\overline{p^{i}}\left(\tau_{n}, \eta\right)-\overline{P^{i}}_{\infty}(\eta)\right|=0
$$

also uniformly in $\eta$ on the half-line $\eta \geq 0$.

\subsection{Construction of a lower barrier}

Using the upper barrier that we constructed in the previous section for $p^{i}(\tau, \eta)$ we get  

$$
e^{3 \tau / 2-\eta \exp (\tau / 2)} p^{i}(\tau, \eta) \leq C_{\gamma}  e^{-\exp (\frac{\gamma \tau }{2})}
$$
as soon as $\eta \geq e^{-(1 / 2-\gamma) \tau} $ with $\gamma \in(0,1 / 2)$, with $C_{\gamma}>0$ chosen sufficiently large.

So we can choose $C_{\gamma}>0$ large enough to have for $1 \leq i \leq k $ 
\begin{equation} \label{inequalities}
\begin{split}
&e^{3 \tau / 2-\eta \exp (\tau / 2)} (p^{i})^{2}(\tau, \eta) \leq \frac{C_{\gamma}}{2} e^{-\exp (\gamma \tau / 2)} p^{i}(\tau,\eta) \\
& \alpha e^{\tau / 2-\eta \exp (\tau / 2)} p^{i+1}(\tau, \eta +\tau e^{-\frac{\tau}{2}}) p^{i}(\tau,\eta) \leq \frac{C_{\gamma}}{2} e^{-\exp (\gamma \tau / 2)} p^{i}(\tau,\eta) \\
\end{split}
\end{equation}
Thus, a lower barrier can be defined as the solution $\underline{p^{i}}$ of

\begin{equation}
\begin{split}
& \underline{p^{i}}_{\tau}+L \underline{p^{i}}+(k-i)\underline{p^{i}}+\left( \frac{3}{2}+i-k \right)e^{-\tau / 2} \underline{p^{i}}_{\eta}+C_{\gamma} e^{-e^{\frac{\gamma \tau}{2}}} \underline{p^{i}}=\alpha \underline{p^{i+1}}(\tau, \eta+\tau e^{-\frac{\tau}{2}}), \eta>e^{-(\frac{1}{2}-\gamma) \tau}, 1 \leq i \leq k-1 \\
& \underline{p^{k}}_{\tau}+L \underline{p^{k}}+\frac{3}{2} e^{-\tau / 2} \underline{p^{k}}_{\eta}+C_{\gamma} e^{-e^{\frac{ \gamma \tau}{2}}} \underline{p^{k}}=0, \quad \eta>e^{-(\frac{1}{2}-\gamma) \tau} \\
& \underline{p^{i}}\left(\tau, e^{-(\frac{1}{2}-\gamma) \tau}\right)=0 \\
& \underline{p^{k}}\left(\tau, e^{-(\frac{1}{2}-\gamma) \tau}\right)=0
\end{split}
\end{equation}

with initial conditions $\underline{p^{i}}_{0}=p^{i}_{0}$.
\medskip

We will suppose that $\underline{p^{i+1}}$ is a lower barrier for $p^{i+1}$, i.e. that $\underline{p^{i+1}}(\tau,\eta) \leq p^{i+1}(\tau, \eta) $ for all $\tau>0, \eta > e^{-(\frac{1}{2}-\gamma)\tau}$  and show that then this holds for $i$. In \cite{convergencetothewave} it is shown that $\underline{p^{k}}$ is a lower barrier for $p^{k}$, so we will then be able to conclude that $\underline{p^{i}}$ is a lower barrier for $p^{i}$ for $1 \leq i \leq k$.

We set  $$N(w)=w_{\tau}+Lw+(k-i)w+\left( \frac{3}{2}+i-k \right) e^{-\frac{\tau}{2}}w_{\eta}+C_{\gamma} e^{-e^{\frac{\gamma \tau}{2}}} w-\alpha \underline{p}^{i+1}(\tau, \eta+\tau e^{-\frac{\tau}{2}}) $$

we have that $N(\underline{p}^{i})=0$, while

$$N(p^{i})=\alpha p^{i+1}(\tau, \eta+\tau e^{-\frac{\tau}{2}})-\alpha \underline{p}^{i+1}(\tau, \eta + \tau e^{-\frac{\tau}{2}})+ C_{\gamma}e^{-e^{\frac{\gamma \tau}{2}}}p^{i}-e^{\frac{3\tau}{2}-\eta e^{\frac{\tau}{2}}}(p^{i})^{2}-\alpha p^{i+1}(\tau, \eta +\tau e^{-\frac{\tau}{2}}) e^{\frac{\tau}{2}-\eta e^{\frac{\tau}{2}}}p^{i} \geq 0$$ 

where we used that $p^{i+1}(\tau, \eta +\tau e^{-\frac{\tau}{2}}) \geq \underline{p}^{i+1}(\tau, \eta+ \tau e^{-\frac{\tau}{2}})$ and $(\eqref{inequalities})$.

As a result, $\underline{p}^{i}$ is a lower barrier for $p^{i}$ and we must have $p^{i}(\tau,\eta) \geq \underline{p}^{i}(\tau,\eta)$ for $\tau>0, \eta >e^{-(\frac{1}{2}-\gamma)\tau}.$

We make again a change of variables to shift the moving boundary to a fixed boundary at $0$

\begin{equation}
\begin{split}
& \underline{p}^{i}(\tau, \eta)=\underline{z}^{i}\left(\tau, \eta-e^{-(1 / 2-\gamma) \tau}\right) \\
\end{split}
\end{equation}

For the functions $\underline{z}^{i}$ we have the calculations

\begin{equation} \nonumber
\begin{split}
& \underline{p}^{i}_{\tau}=\underline{z}^{i}_{\tau} +(\frac{1}{2}-\gamma)e^{-(\frac{1}{2}-\gamma)\tau}\underline{z}^{i}_{\eta} \\
& \underline{p}^{i+1}(\tau, \eta + \tau e^{-\frac{\tau}{2}})= \underline{z}^{i+1}(\tau, \eta+\tau e^{-\frac{\tau}{2}} - e^{-(\frac{1}{2}-\gamma)\tau}) 
\end{split}
\end{equation}

and translating from $  \eta-e^{-(1 / 2-\gamma) \tau}$ back to $\eta$ we will get a cancellation of $\frac{1}{2} e^{-(\frac{1}{2}-\gamma)\tau}\underline{z}_{\eta}^{i}$, since 
$$
Lz=-z_{\eta \eta}-\frac{1}{2}(\eta+e^{-(1 / 2-\gamma) \tau})z_{\eta}-z
$$
so the equations for $\underline{z}^{i}$ are 

\begin{equation}
\begin{split}
& \underline{z}^{i}_{\tau}+L \underline{z}^{i}+(k-i)\underline{z}^{i}+\left(-\gamma e^{-(\frac{1}{2}-\gamma) \tau}+\left( \frac{3}{2}+i-k \right) e^{-\frac{\tau}{2}}\right) \underline{z}^{i}_{\eta}+C_{\gamma} e^{-e^\frac{\gamma \tau}{2}} \underline{z}^{i}=\alpha \underline{z}^{i+1}(\tau, \eta +\tau e^{-\frac{\tau}{2}}), 1 \leq i \leq k-1  \\
& \underline{z}^{k}_{\tau}+L \underline{z}^{k}+\left(-\gamma e^{-(\frac{1}{2}-\gamma) \tau}+\frac{3}{2} e^{-\frac{\tau}{2}}\right) \underline{z}^{k}_{\eta}+C_{\gamma} e^{-e^\frac{\gamma \tau}{2}} \underline{z}^{k}=0, \quad \eta>0\\
&\underline{z}^{i}(\tau, 0)=0, 1 \leq i \leq k.
\end{split}
\end{equation}

For any $\varepsilon>0$, we may choose $T$ sufficiently large, and $\omega \in(0,\frac{1}{2}-\gamma)$ so that

\begin{equation}
\begin{split}
&\left|\underline{z}^{k}_{\tau}+L \underline{z}^{k}  \right| \leq \varepsilon e^{-\omega(\tau-T)}\left(\left|\underline{z}^{k}_{\eta}\right|+|\underline{z}^{k}|\right), \quad \tau>T, \eta>0, \quad \underline{z}^{k}(\tau, 0)=0 \\
& \left|\underline{z}^{i}_{\tau}+L \underline{z}^{i} +(k-i)\underline{z}^{i} -\alpha \underline{z}^{i+1}(\tau, \eta+\tau e^{-\frac{\tau}{2}}) \right| \leq \varepsilon e^{-\omega(\tau-T)}\left(\left|\underline{z}^{i}_{\eta}\right|+|\underline{z}^{i}|\right), \tau>T, \eta>0, \underline{z}^{i}(\tau, 0)=0, 1 \leq i \leq k-1.\\
\end{split}
\end{equation}

Then, applying the Lemma \ref{differentialinequalities} for $\tau>T$, gives us 

$$
\begin{aligned}
& \underline{z^{i}}(\tau, \eta)=\eta \left(e^{-\frac{\eta^2}{4}}\left( \frac{\alpha^{k-i}}{(k-i)! 2\sqrt{\pi}}\int_{0}^{+\infty} \eta \underline{z^{k}}(T,\eta) d \eta +C (\tau-T)e^{-(\tau-T)} +\varepsilon R_{1}^{i}(\tau,\eta) \right)+ e^{-\frac{\omega(\tau-T)}{2}} e^{-\frac{\eta^2}{6}}R_2^{i}(\tau,\eta) \right) \\
& \text{ where } \left\|R_{1,2}^{i}(\tau, \cdot)\right\|_{C^{3}} \leq C'.
\end{aligned}
$$

It was shown in \cite{convergencetothewave} that there exists a constant $c_0$, which depends on the initial condition, such that

$$
\int_{0}^{\infty} \eta \underline{z}^{k}(\tau, \eta) d \eta \geq c_0>0, \text { for all } \tau>0.
$$

\noindent So as in the study of the upper barrier, we obtain the uniform convergence of (possibly a subsequence of) $\underline{p}^{i}\left(\tau_{n}, \cdot\right)$ on the half-line $\eta \geq e^{-(1 / 2-\gamma) \tau}$ to a function $\underline{P}_{\infty}^{i}(\eta)$ which satisfies

$$
C^{-1} \eta e^{-\eta^{2} / 4} \leq \underline{P}_{\infty}^{i}(\eta) \leq C \eta e^{-\eta^{2} / 4}
$$

and such that

$$
\lim _{n \rightarrow+\infty} e^{\eta^{2} / 8}\left|\underline{p}^{i}\left(\tau_{n}, \eta\right)-\underline{P}_{\infty}^{i}(\eta)\right|=0, \quad \eta>0.
$$

\subsection*{Proof of Lemma $\ref{nonlinearasymptotics}$}

We denote by \( X \) the space of bounded uniformly continuous functions \( u(\eta) \) on \( \mathbb{R}_+ \) such that \( e^{\eta^2/8} u(\eta) \) is also bounded and uniformly continuous. From the convergence of the upper and lower barriers for \( p^{i}(\tau, \eta) \) that solve system \eqref{selfsimilarnonlinear2} we deduce the existence of a sequence \( \tau_n \to +\infty \) for which \( p^{i}(\tau_n, \cdot) \) converges in \( X \) to a limit \( P^{i}_\infty \in X \), satisfying \( P^{i}_\infty \equiv 0 \) on \( \mathbb{R}_- \) and \( P^{i}_\infty(\eta) > 0 \) for all \( \eta > 0 \).

With a bootstrap argument as in \cite{convergencetothewave} we can upgrade this to full convergence as \( \tau \to +\infty \). We consider the translated functions
\[
p^{i}_n(\tau, \eta) := p^{i}(\tau + \tau_n, \eta),
\]
which converge in \( X \), uniformly on compact time intervals, to \( p^{i}_\infty(\tau, \eta) \) that solve the linear system
\[
\begin{aligned}
& (\partial_\tau + L) p^{i}_\infty + (k-i)p^{i}_\infty = \alpha p^{i+1}_\infty, \eta > 0, 1 \leq i \leq k-1 \\
& (\partial_\tau + L) p^{k}_\infty = 0,\eta > 0, \\
& p^{i}_\infty(\tau, 0) = 0, 1 \leq i \leq k\\
& p^{i}_\infty(0, \eta) = P^{i}_\infty(\eta), 1 \leq i \leq k. \\
\end{aligned}
\]

We have that \( p^{i}_\infty(\tau, \eta) \) is the solution to the unperturbed problem \eqref{selfsimilarlinearunperturbed} with initial conditions $P^{i}_{\infty}(\eta)$. Thus by Lemma \ref{linearunperturbedasymptotics3}, we have that $p^{i}_{\infty}$ converges in $ X$ as $\tau \to +\infty$  to 
$
\bar{\psi^{i}}(\eta) := \alpha^{i}_\infty \eta e^{-\eta^2/4},
$

where $\alpha^{i}_\infty=\frac{\alpha^{k-i}}{(k-i)! 2\sqrt{\pi}}\int_{0}^{+\infty} \eta P^{k}_{\infty}(0,\eta) d \eta. $

For any \( \varepsilon > 0 \), there exists \( T_\varepsilon > 0 \) such that
\[
\left| p^{i}_\infty(\tau, \eta) - \alpha^{i}_\infty \eta e^{-\eta^2/4} \right| \leq \varepsilon \eta e^{-\eta^2/8}, \quad \text{for all } \tau > T_\varepsilon, \ \eta > 0.
\]

Fix such a \( T_\varepsilon \). Then for all \( n > N_\varepsilon \) (for some sufficiently large \( N_\varepsilon \)), we have
\[
\left| p^{i}\left(T_\varepsilon + \tau_n, \eta + e^{-(1/2 - \gamma) T_\varepsilon} \right) - p^{i}_\infty(T_\varepsilon, \eta) \right| \leq \varepsilon \eta e^{-\eta^2/8}.
\]
It follows that for $1 \leq i \leq k$
\begin{equation} \label{boundinitialconditions}
\alpha^{i}_\infty \eta e^{-\eta^2/4} - 2\varepsilon \eta e^{-\eta^2/8}
\leq p^{i}\left(\tau_{N_\varepsilon} + T_\varepsilon, \eta + e^{-(1/2 - \gamma) T_\varepsilon} \right)
\leq \alpha^{i}_\infty \eta e^{-\eta^2/4} + 2\varepsilon \eta e^{-\eta^2/8}.
\end{equation}

We can now construct upper and lower barriers for the system satisfied by 

$
p^{i}\left(\tau + \tau_{N_\varepsilon} + T_\varepsilon, \eta + e^{-(1/2 - \gamma) T_\varepsilon} \right), 1 \leq i \leq k
$ as we did previously and using the bound \eqref{boundinitialconditions} for the initial conditions. Applying Lemma~\ref{differentialinequalities} to these barriers shows that any limit point \( \phi^{i} \in X \) of \( p^{i}(\tau, \cdot) \) as \( \tau \to +\infty \) satisfies
\[
\left( \alpha^{i}_\infty - C\varepsilon \right) \eta e^{-\eta^2/4}
\leq \phi^{i}(\eta)
\leq \left( \alpha^{i}_\infty + C\varepsilon \right) \eta e^{-\eta^2/4}.
\]
Since \( \varepsilon > 0 \) is arbitrary, we conclude that
\[
p^{i}(\tau, \eta) \to \bar{\psi^{i}}(\eta) = \alpha^{i}_\infty \eta e^{-\eta^2/4} \quad \text{in } X \text{ as } \tau \to +\infty.
\]

Finally, applying Lemma~\ref{differentialinequalities} to the upper and lower barriers of $p^{i}(\tau,\eta)$, constructed from any time $\tau > 0$,  completes the proof of Lemma~\ref{nonlinearasymptotics}.

\subsection{Proof of Theorem \ref{convergencetothewavecascading}}

Having proven the stabilization Lemma \ref{stabilization}, we now proceed with the proof of Theorem \ref{convergencetothewavecascading}. Fix $\gamma \in(0,\frac{1}{3})$, and take any $\varepsilon>0$. We go back to $\widetilde{v^{i}}$, the solution of the original nonlinear problem in the logarithmic frames \eqref{logmovingframecascading}.
 Following Lemma $\ref{stabilization}$ there exist constants $\alpha^{i}_{\infty}>0$ and $T_{\varepsilon}$ so that for all $t>T_{\varepsilon}$ we have

 \begin{equation}
\begin{split}
&\left| \widetilde{v^{i}}\left(t, x_{\gamma}\right)-\alpha^{i}_{\infty} x_{\gamma} e^{-x_{\gamma}} e^{-\frac{x_{\gamma}^{2}}{4 t}}\right| \leq \varepsilon x_{\gamma} e^{-x_{\gamma}} e^{-\frac{x_{\gamma}^{2}}{6 t}} \\
\end{split}
\end{equation}

with $x_{\gamma}=t^{\gamma}$.

We consider the solution $\widetilde{v^{\alpha^{i}_{\infty}+\varepsilon}}$ for $t > T_{\varepsilon}$ to the following moving boundary problem  with $x \leq x_{\gamma}(t)$
\small
\begin{equation} \label{utheta}
\begin{split}
&(\widetilde{v^{\alpha^{i}_{\infty}+\varepsilon}})_{t}+\left(-2+\frac{\left( \frac{3}{2}+i-k\right)}{t}\right)(\widetilde{v^{\alpha^{i}_{\infty}+\varepsilon}})_{x} = (\widetilde{v^{\alpha^{i}_{\infty}+\varepsilon}})_{xx}+\widetilde{v^{\alpha^{i}_{\infty}+\varepsilon}}-(\widetilde{v^{\alpha^{i}_{\infty}+\varepsilon}})^{2}+\alpha \widetilde{v^{i+1}}(t,x+\ln(t)), 1 \leq i \leq k-1 \\
& (\widetilde{v^{\alpha^{k}_{\infty}+\varepsilon}})_{t}+ \left( -2+\frac{\frac{3}{2}}{t} \right)(\widetilde{v^{\alpha^{k}_{\infty}+\varepsilon}})_{x}= (\widetilde{v^{\alpha^{k}_{\infty}+\varepsilon}})_{xx}+\widetilde{v^{\alpha^{k}_{\infty}+\varepsilon}}-(\widetilde{v^{\alpha^{k}_{\infty}+\varepsilon}})^{2}, \quad x \leq x_{\gamma}(t)=t^{\gamma}\\
& \widetilde{v^{\alpha^{i}_{\infty}+\varepsilon}}\left(t, t^{\gamma}\right)=(\alpha^{i}_{\infty}+\varepsilon) t^{\gamma} e^{-t^{\gamma}-\frac{t^{2 \gamma-1}}{6}} \\
\end{split}
\end{equation}
\normalsize
with initial conditions $\widetilde{v^{\alpha^{i}_{\infty}+\varepsilon}}\left(T_{\varepsilon}, x\right)=\widetilde{v^{i}}\left(T_{\varepsilon}, x\right)$.

We have that for all $1 \leq i \leq k$, $\widetilde{v^{\alpha^{i}_{\infty}+\varepsilon}}(t,x) \geq \widetilde{v^{i}}(t,x)$ for all $x \leq t^{\gamma}, t > T_{\varepsilon}.$

The corresponding $z^{\alpha^{i}_{\infty}+\varepsilon}(t, x)=e^{x} \widetilde{v^{\alpha^{i}_{\infty}+\varepsilon}}(t, x)$ solve 
$$
\begin{aligned}
&z^{\alpha^{i}_{\infty}+\varepsilon}_{t}+\frac{\left(\frac{3}{2}+i-k \right)}{t} (z^{\alpha^{i}_{\infty}+\varepsilon}_{x}-z^{\alpha^{i}_{\infty}+\varepsilon} )=z^{\alpha^{i}_{\infty}+\varepsilon}_{x x}-e^{-x}(z^{\alpha^{i}_{\infty}+\varepsilon})^2+\frac{\alpha }{t }z^{i+1}(t,x+\ln t)  , 1 \leq i \leq k-1, x \leq t^{\gamma} \\
&z^{\alpha^{k}_{\infty}+\varepsilon}_{t}+\frac{3}{2 t}(z^{\alpha^{k}_{\infty}+\varepsilon}_{x}-z^{\alpha^{k}_{\infty}+\varepsilon})=z^{\alpha^{k}_{\infty}+\varepsilon}_{x x}-e^{-x}(z^{\alpha^{k}_{\infty}+\varepsilon})^2, \quad x \leq t^{\gamma} \\
& z^{\alpha^{i}_{\infty}+\varepsilon}\left(t, t^{\gamma}\right)= (\alpha^{i}_{\infty}+\varepsilon) t^{\gamma} e^{-\frac{t^{2 \gamma-1}}{6}} \\
\end{aligned}
$$

where $z^{i}$ solve the system \eqref{z_system}.
\noindent We choose the translation of the wave that exactly matches the behavior of $\widetilde{v^{\alpha^{i}_{\infty}+\varepsilon}}$ at the boundary $x=t^{\gamma}$.

We set

$$
\psi^{i}(t, x)=e^{x} U_{c_{*}}(x+\zeta^{i}(t))
$$

where $U_{c_{*}}(x)$ is the minimal speed traveling wave. We look for a function $\zeta^{i}(t)$ such that

$$
\psi^{i}\left(t, t^{\gamma}\right)=z^{\alpha^{i}_{\infty}+\varepsilon}\left(t, t^{\gamma}\right)
$$

We recall that the traveling wave $U_{c_{*}}(\xi)$ has the asymptotics

$$
U_{c_{*}}(\xi)=(\xi+k) e^{-\xi}+O\left(e^{-\left(1+\omega_{0}\right) \xi}\right), \quad \xi \rightarrow+\infty
$$

Based on this expansion, we should have, with some $\omega_{0}>0$ :

$$
e^{-\zeta^{i}(t)}\left(t^{\gamma}+\zeta^{i}(t)+k\right)+O\left(e^{-\omega_{0} t^{\gamma}}\right)=(\alpha^{i}_{\infty}+\varepsilon) t^{\gamma} e^{-\frac{1}{c t^{1-2 \gamma}}}
$$

which implies, for $\gamma \in(0, \frac{1}{2})$:

$$
e^{-\zeta^{i}(t)}\left(1+\zeta^{i}(t)t^{-\gamma}+k t^{-\gamma}\right)+O\left( t^{-\gamma} e^{-\omega_{0} t^{\gamma}}\right)=(\alpha^{i}_{\infty}+\varepsilon) e^{-\frac{1}{ c t^{1-2 \gamma}}}
$$

Taking the logarithm on both sides and keeping the leading order terms gives us
$$
-\zeta^{i}(t)+ \ln \left(1+\zeta^{i}(t)t^{-\gamma}+k t^{-\gamma}\right)
= \ln (\alpha^{i}_{\infty}+\varepsilon) 
$$

and using the approximation $ \ln (1 +z) \sim z$ we get 

$$
\zeta^{i}(t)(1-t^{-\gamma})=k t^{-\gamma} - \ln (\alpha^{i}_{\infty}+\varepsilon) 
$$

so that

$$
\zeta^{i}(t)=\left(\frac{k}{t^{\gamma}} -  \ln (\alpha^{i}_{\infty}+\varepsilon) \right)\left( 1+\frac{1}{t^{\gamma}}\right)
$$

and thus

$$
\zeta^{i}(t)=-\ln (\alpha^{i}_{\infty}+\varepsilon) -(\ln(\alpha^{i}_{\infty}+\varepsilon) -k) t^{-\gamma}+O\left(t^{-2 \gamma}\right)
$$

As a result, we also have  

$$
|(\zeta^{i})^{\prime}(t)| \leq \frac{C}{t^{1+\gamma}}
$$

We have set $\psi^{i}(t,x)=e^{x}U_{c_{*}}(x+\zeta^{i}(t))$, so 

$$
\begin{aligned}
&\psi^{i}_{t}=e^{x}U_{c_{*}}'(x+\zeta^{i}(t))(\zeta^{i}) '(t) \\
&\psi^{i}_{x}=e^{x}U_{c_{*}}(x+\zeta^{i}(t))+e^{x}U_{c_{*}}'(x+\zeta^{i}(t))=\psi^{i}(t,x) + e^{x}U_{c_{*}}'(x+\zeta^{i}(t)) \\
&\psi^{i}_{xx}=\psi^{i}(t,x)+2e^{x}U_{c_{*}}'(x+\zeta^{i}(t))+e^{x}U_{c_{*}}''(x+\zeta^{i}(t))
\end{aligned}
$$

and thus

$$
\psi^{i}_{t}=( \psi^{i}_{x}-\psi^{i})(\zeta^{i})'(t)
$$

The traveling wave $U_{c_{*}}$ satisfies the ODE
$$
2U_{c_{*}}'+U_{c_{*}}''+U_{c_{*}}-U_{c_{*}}^2=0
$$

As a result, we get the following equation for the function $\psi^{i}$ 

\small
$$
\begin{aligned}
&\psi^{i}_{t}-\psi^{i}_{x x}+\frac{\frac{3}{2}+i-k}{t}\left(\psi^{i}_{x}-\psi^{i}\right)+e^{-x} (\psi^{i})^{2}=(\psi^{i}_{x} -\psi^{i})(\zeta^{i})'-e^{x}(U_{c_{*}}+2U_{c_{*}}'+U_{c_{*}}'')+\frac{\frac{3}{2}+i-k}{ t}\left(\psi^{i}_{x}-\psi^{i}\right)+e^{x}U_{c_{*}}^2=\\
&= (\psi^{i}_{x} -\psi^{i})(\zeta^{i})'-e^{x}(U_{c_{*}}+2U_{c_{*}}'+U_{c_{*}}''-U_{c_{*}}^2)+\frac{\frac{3}{2}+i-k}{t}\left(\psi^{i}_{x}-\psi^{i}\right)=
(\psi^{i}_{x} -\psi^{i})(\zeta^{i})'+\frac{\frac{3}{2}+i-k}{ t}\left(\psi^{i}_{x}-\psi^{i}\right)= \\
&=O\left(\frac{x}{t}\right)=O\left(t^{-1+\gamma}\right),|x|<t^{\gamma}
\end{aligned}
$$

while $z^{\alpha^{i}_{\infty}+\varepsilon}(t,x)$ satisfies

$$
\begin{aligned}
&z^{\alpha^{i}_{\infty}+\varepsilon}_{t}+\frac{\left(\frac{3}{2}+i-k \right)}{t} (z^{\alpha^{i}_{\infty}+\varepsilon}_{x}-z^{\alpha^{i}_{\infty}+\varepsilon} )=z^{\alpha^{i}_{\infty}+\varepsilon}_{x x}-e^{-x}(z^{\alpha^{i}_{\infty}+\varepsilon})^2+\frac{\alpha }{t }z^{i+1}(t,x+\ln t) , 1 \leq i \leq k-1, x \leq t^{\gamma} \\
&z^{\alpha^{k}_{\infty}+\varepsilon}_{t}+\frac{3}{2 t}(z^{\alpha^{k}_{\infty}+\varepsilon}_{x}-z^{\alpha^{k}_{\infty}+\varepsilon})=z^{\alpha^{k}_{\infty}+\varepsilon}_{x x}-e^{-x}(z^{\alpha^{k}_{\infty}+\varepsilon})^2, \quad x \leq t^{\gamma} \\
& z^{\alpha^{i}_{\infty}+\varepsilon}\left(t, t^{\gamma}\right)= (\alpha^{i}_{\infty}+\varepsilon) t^{\gamma} e^{-\frac{t^{2 \gamma-1}}{6}} \\
&z^{\alpha^{i}_{\infty}+\varepsilon} \left(t, -t^{\gamma}\right)=O \left(e^{-t^{\gamma}} \right) \\
\end{aligned}
$$

\noindent In addition, $\psi^{i}(t,x)$  is exponentially small for $x<-t^{\gamma}$ because of the exponential factor.
\newline 
In particular $\psi^{i}(t, -t^{\gamma})=e^{-t^{\gamma}}U_{c_{*}}(-t^{\gamma}+\zeta^{i}(t))=O\left(e^{-t^{\gamma}}\right)$. Hence, the difference $s^{i}(t, x)=z^{\alpha^{i}_{\infty}+\varepsilon}(t, x)-\psi^{i}(t, x)$ satisfies for $|x| \leq t^{\gamma}$

$$
\begin{aligned}
&s^{i}_{t}-s^{i}_{x x}+\frac{\left( \frac{3}{2}+i-k \right)}{t}\left(s^{i}_{x}-s^{i}\right)+e^{-x}(\psi^{i}+z^{\alpha^{i}_{\infty}+\varepsilon}) s^{i}=\frac{\alpha }{t }z^{i+1}(t,x+\ln t) -(\psi^{i}_{x} -\psi^{i})\left((\zeta^{i})'+\frac{\frac{3}{2}+i-k}{ t}\right), 1 \leq i \leq k-1\\
&s^{k}_{t}-s^{k}_{x x}+\frac{3}{2t}\left(s^{k}_{x}-s^{k}\right)+e^{-x}(\psi^{k}+z^{\alpha^{k}_{\infty}+\varepsilon}) s^{k}=-(\psi^{k}_{x} -\psi^{k})\left((\zeta^{k})'+\frac{3}{ 2t}\right)=O\left(t^{-1+\gamma}\right) \\
& s^{i}\left(t,-t^{\gamma}\right)=O\left(e^{-t^{\gamma}}\right), \quad s^{i}\left(t, t^{\gamma}\right)=0, 1 \leq i \leq k.
\end{aligned}
$$
\normalsize
We will show the following proposition.
\begin{proposition} \label{sconvergence}
    
For any $\gamma \in(0,1 / 3)$, we have for $s^{i}(t,x)=z^{\alpha^{i}_{\infty}+\varepsilon}(t,x)-\psi^{i}(t,x)$, $1 \leq i \leq k$, 

$$
\lim _{t \rightarrow+\infty} \sup _{|x| \leq t^{\gamma}}|s^{i}(t, x)|=0
$$
\end{proposition}
\begin{proof}
Proposition 4.3 from \cite{convergencetothewave} shows this for $s^{k}(t,x)=z^{\alpha^{k}_{\infty}+\varepsilon}(t,x)-\psi^{k}(t,x)$. We suppose that the statement holds for $i+1$ and show that then it holds for $i$, in order to conclude that it holds for all $1 \leq i \leq k$.

We fix $0<\gamma < \frac{1}{3}$ and take $\gamma'$ with $\gamma < \gamma' < \frac{1}{3}$ such that $t^{\gamma}+\ln(t) < t^{\gamma'}$ for all $t$ large enough.
We have that then for all $|y| \leq t^{\gamma'}$, the corresponding solution to the moving boundary problem $z^{\alpha_{\infty}^{i}+\varepsilon}$ satisfies

$$
z^{\alpha_{\infty}^{i+1}+\varepsilon}(t,y)=\psi^{i+1}(t,y)+o(1), \text{ as } t \rightarrow \infty
$$

In particular for all $|x| < t^{\gamma}$ we have $t^{\gamma}+\ln(t) < t^{\gamma'}$ for $t$ large enough, so 
$$\frac{\alpha }{t }z^{i+1}(t,x+\ln t) =O\left(t^{-1+\gamma'}\right) , \text{ for }\quad|x|<t^{\gamma} $$

As a result we get for such $t> T_{\varepsilon}$ (potentially taking it larger so that the above hold),

\begin{equation} \label{sequation}
\begin{split}
&s^{i}_{t}-s^{i}_{x x}+\frac{\left( \frac{3}{2}+i-k \right)}{t}\left(s^{i}_{x}-s^{i}\right)+e^{-x}(\psi^{i}+z^{\alpha^{i}_{\infty}+\varepsilon}) s^{i}=\frac{\alpha }{t }z^{i+1}(t,x+\ln t)\\
&-(\psi^{i}_{x} -\psi^{i})\left((\zeta^{i})'+\frac{\frac{3}{2}+i-k}{ t}\right)=O\left(t^{-1+\gamma'}\right) , \quad|x| \leq t^{\gamma} \\
& s^{i}\left(t,-t^{\gamma}\right)=O\left(e^{-t^{\gamma}}\right), \quad s^{i}\left(t, t^{\gamma}\right)=0, 1 \leq i \leq k.
\end{split}
\end{equation}

 Following the strategy of \cite{convergencetothewave} which is due to the principal Dirichlet eigenvalue for the Laplacian in $\left(-t^{\gamma}, t^{\gamma}\right)$ being $\frac{\pi^{2}}{4 t^{2 \gamma}}$, we will have a supersolution of the form

$$
\overline{s^{i}}(t, x)=t^{-\lambda} \cos \left(\frac{x}{t^{\gamma+\delta}}\right)
$$

\noindent where $\lambda$ and $\delta$ will be chosen to be small enough.
\newline 
We can choose $\delta$ small enough so that there exist $c_1>0,c_2>0$ such that $c_1 t^{-\lambda} \leq \overline{s^{i}}(t, x) \leq c_2 t^{-\lambda}$ for all $|x| \leq t^{\gamma}$. We also have that

$$
\begin{aligned}
& \overline{s^{i}}_{x}=-t^{-\lambda-\gamma-\delta}\sin \left(\frac{x}{t^{\gamma+\delta}}\right),-\overline{s^{i}}_{x x}=t^{-(2 \gamma+2 \delta)} \overline{s^{i}}(t, x) \\
& \overline{s^{i}}_{t}=-\frac{\lambda}{t} \overline{s^{i}}+t^{-\lambda}\sin\left(\frac{x}{t^{\gamma+\delta}}\right)x t^{-\gamma-\delta-1}(\gamma+\delta)
\end{aligned}
$$

We set 
$g(t,x)=t^{-\lambda}\sin\left(\frac{x}{t^{\gamma+\delta}}\right)x t^{-\gamma-\delta-1}(\gamma+\delta)$
and have that
$$|g(t,x)| \leq C |x| t^{-\lambda-\gamma -\delta-1 } 
\leq Ct^{-\delta-1} \overline{s^{i}}(t,x) \text{ for all } |x| \leq t^{\gamma}.$$

We get

\begin{equation}
\begin{split}
&|\frac{\left(\frac{3}{2}+i-k\right)}{ t}\left(\overline{s^{i}}_{x}-\overline{s^{i}}\right)(t, x)| =|\frac{\left(\frac{3}{2}+i-k\right)}{t} \left( -t^{-\lambda-\gamma-\delta}\sin (\frac{x}{t^{\gamma+\delta}})-\overline{s^{i}}(t,x)\right)| \\
& \leq \frac{|\frac{3}{2}+i-k|}{t}| t^{-\lambda-\gamma-\delta}\sin (\frac{x}{t^{\gamma+\delta}})+\overline{s^{i}}(t,x)| \\
& \leq \frac{|\frac{3}{2}+i-k|}{t} \left(| t^{-\lambda-\gamma-\delta}\sin (\frac{x}{t^{\gamma+\delta}})|+|\overline{s^{i}}(t,x)| \right) \leq \frac{|\frac{3}{2}+i-k|}{t} \left( t^{-\lambda-\gamma-\delta}+|\overline{s^{i}}(t,x)| \right) \\
&\leq \frac{|\frac{3}{2}+i-k|}{t} \overline{s^{i}}(t,x) \left( \frac{1}{c_1}t^{-\gamma -\delta} +1 \right) \leq \frac{C_{i}}{t} \overline{s^{i}}(t,x)
\end{split}
\end{equation}
for some constant $C_{i}>0$.
As a result we get that there exists $q_{i}>0$ such that, for $t$ large enough:
\small
\begin{equation}
\begin{split}
& \left(\partial_{t}-\partial_{x x}+\frac{\frac{3}{2}+i-k}{t}\left(\partial_{x}-1\right)\right) \overline{s^{i}}(t, x) \geq \left(-\lambda t^{-1}-Ct^{-\delta-1}+t^{-(2\gamma+2\delta)}-C_{i} t^{-1}\right) \overline{s^{i}}(t,x)  \\
& \geq q_{i} t^{-(2 \gamma+2 \delta)} \overline{s^{i}}(t,x) \geq q_{i} c_1 t^{-(2 \gamma+2 \delta +\lambda)} \geq O\left(\frac{1}{t^{1-\gamma'}}\right)
\end{split}
\end{equation}
\normalsize

\noindent as long as $\delta$ and $\lambda$ are small enough, since $\gamma,\gamma' \in(0,1 / 3)$. 
\newline 
Since the right-hand side does not depend on $\overline{s^{i}}$, the inequality extends to all $t \geq T_{\varepsilon}$ by replacing $\overline{s^{i}}$ by $A \overline{s^{i}}$ with $A$ large enough. As a result, we have that $\overline{s^{i}}$ is a supersolution to $\eqref{sequation}$ and thus $s^{i}(t,x) \leq \overline{s^{i}}(t,x) $ for all $|x| \leq t^{\gamma}$ and we get $\lim_{t \rightarrow \infty}\sup_{ |x| \leq t^{\gamma}}|s^{i}(t,x)|=0.$ We showed the statement for $i$, so we can conclude that this holds for $1 \leq i \leq k$ by closing the induction. 

\end{proof}

Similarly, we can consider the solution $\widetilde{v^{\alpha^{i}_{\infty}-\varepsilon}}$ for $t > T_{\varepsilon}$ to the moving boundary problem where we drop the coupling term, $x \leq x_{\gamma}(t)=t^{\gamma}$
\small
\begin{equation} \label{utheta}
\begin{split}
&(\widetilde{v^{\alpha^{i}_{\infty}-\varepsilon}})_{t}+\left(-2+\frac{\left( \frac{3}{2}+i-k\right)}{t}\right)(\widetilde{v^{\alpha^{i}_{\infty}-\varepsilon}})_{x} = (\widetilde{v^{\alpha^{i}_{\infty}-\varepsilon}})_{xx}+\widetilde{v^{\alpha^{i}_{\infty}-\varepsilon}}-(\widetilde{v^{\alpha^{i}_{\infty}-\varepsilon}})^{2}, x \leq x_{\gamma}(t), 1 \leq i \leq k \\
& \widetilde{v^{\alpha^{i}_{\infty}-\varepsilon}}\left(t, t^{\gamma}\right)=(\alpha^{i}_{\infty}-\varepsilon) t^{\gamma} e^{-t^{\gamma}-\frac{t^{2 \gamma-1}}{4}} \\
\end{split}
\end{equation}
\normalsize
with initial conditions $\widetilde{v^{\alpha^{i}_{\infty}-\varepsilon}}\left(T_{\varepsilon}, x\right)=\widetilde{v^{i}}\left(T_{\varepsilon}, x\right)$.
We have that for all $1 \leq i \leq k$, $\widetilde{v^{\alpha^{i}_{\infty}-\varepsilon}}(t,x) \leq \widetilde{v^{i}}(t,x)$ for all $x \leq t^{\gamma}, t > T_{\varepsilon}.$ 

The corresponding $z^{\alpha^{i}_{\infty}-\varepsilon}(t, x)=e^{x} \widetilde{v^{\alpha^{i}_{\infty}-\varepsilon}}(t, x)$ solve 
$$
\begin{aligned}
&z^{\alpha^{i}_{\infty}-\varepsilon}_{t}+\frac{\left(\frac{3}{2}+i-k \right)}{t} (z^{\alpha^{i}_{\infty}-\varepsilon}_{x}-z^{\alpha^{i}_{\infty}-\varepsilon} )=z^{\alpha^{i}_{\infty}-\varepsilon}_{x x}-e^{-x}(z^{\alpha^{i}_{\infty}-\varepsilon})^2 , x \leq t^{\gamma},  1 \leq i \leq k, \\
& z^{\alpha^{i}_{\infty}-\varepsilon}\left(t, t^{\gamma}\right)= (\alpha^{i}_{\infty}-\varepsilon) t^{\gamma} e^{-\frac{t^{2 \gamma-1}}{4}} \\
\end{aligned}
$$
As before we will consider $$
\psi^{i}(t, x)=e^{x} U_{c_{*}}(x+\zeta^{i}(t))
$$

where $U_{c_{*}}(x)$ is the minimal speed traveling wave profile and look for a function $\zeta^{i}(t)$ such that

$$
\psi^{i}\left(t, t^{\gamma}\right)=z^{\alpha^{i}_{\infty}-\varepsilon}\left(t, t^{\gamma}\right)
$$

Proceeding as we did earlier, we get $$
\zeta^{i}(t)=\left(\frac{k}{t^{\gamma}} -  \ln (\alpha^{i}_{\infty}-\varepsilon) \right)\left( 1+\frac{1}{t^{\gamma}}\right)
$$
A direct modification of Proposition 4.3 from \cite{convergencetothewave} gives us that for any $\gamma \in(0,1 / 3)$, we have that for $s^{i}(t,x)=z^{\alpha^{i}_{\infty}-\varepsilon}(t,x)-\psi^{i}(t,x)$, $1 \leq i \leq k$

$$
\lim _{t \rightarrow+\infty} \sup _{|x| \leq t^{\gamma}}|s^{i}(t, x)|=0.
$$

Thus, for any $t\geq T_{\varepsilon}$ we have 

$$
\widetilde{v^{\alpha^{i}_{\infty}-\varepsilon}}(t,x) \leq \widetilde{v^{i}}(t, x) \leq \widetilde{v^{\alpha^{i}_{\infty}+\varepsilon}}(t,x) \text{ 
for all } x \leq t^{\gamma}$$.

and moreover, 

$$
e^{x}\left[\widetilde{v^{\alpha^{i}_{\infty} \pm \varepsilon}}(t,x)-U_{c_{*}}\left(x+\zeta^{i}_{ \pm}(t)\right)\right]=o(1), \text { as } t \rightarrow+\infty
$$

uniformly in $x \in\left(-t^{\gamma}, t^{\gamma}\right)$, with

$$
\zeta^{i}_{ \pm}(t)=- \ln \left(\alpha^{i}_{\infty} \pm  \varepsilon\right)+O\left(t^{- \gamma}\right)
$$

Because $\varepsilon>0$ is arbitrary, we have

\begin{equation} \label{convergencetothewaveequation}
\lim _{t \rightarrow+\infty}\left( \tilde{v^{i}}(t, x)-U_{c_{*}}\left(x+x_{\infty}^{i}\right)\right)=0
\end{equation}

with $x_{\infty}^{i}=-\ln(\alpha^{i}_{\infty})$, uniformly on compact sets. This completes the proof of Theorem \ref{convergencetothewavecascading}.


\subsection{Connecting Theorem \ref{convergencetothewavecascading} with cascading branching Brownian motion}
Note that in particular, we have found the constant $$\alpha^{i}_\infty=  \frac{\alpha^{k-i}}{(k-i)!2\sqrt{\pi}}\int_{0}^{+\infty} \eta P^{k}_{\infty}(0,\eta) d \eta= \lim_{n \rightarrow \infty} \frac{\alpha^{k-i}}{(k-i)!2\sqrt{\pi}}\int_{0}^{+\infty} \eta p^{k}(\tau_{n},\eta) d \eta.$$

In fact, $\alpha^{k}_\infty= \lim_{n \rightarrow \infty} \frac{1}{2\sqrt{\pi}}\int_{0}^{+\infty} \eta p^{k}(\tau_{n},\eta) d \eta:=C^{*}$ is the constant for the single Fisher-KPP equation (by taking $\tau_n$ to be a further subsequence of the sequence $\tau_n$ of the $k$-th component in the proof), for which the front then has asymptotics $2t-\frac{3}{2}\ln(t)+\ln(\alpha^{k}_{\infty})$ as shown in \cite{convergencetothewave}.

We have thereby shown that $\alpha^{i}_{\infty}=\frac{\alpha^{k-i}}{(k-i)!}\alpha^{k}_{\infty}$. In particular, connecting this to the cascading branching Brownian motion through \eqref{cdfcascadingbbmequation}, this proves the conjecture in \cite{belloumpaper}, regarding the convergence in distribution of the centered maximum particle. We denote the maximum particle location of the cascading branching Brownian motion by $M_{t}=\max _{u \in \mathcal{N}_{t}} X_{u}(t)$, and $m(t)=2t-\frac{3}{2}\ln(t)+(k-1)\ln(t)$.

As shown by Lalley and Sellke in \cite{lalleysellke}, the probabilistic representation of the traveling wave is given by 
$
U_{c*}(x)=1-\mathbb{E} \exp \left(-e^{-x} Z_{\infty}\right) 
$, where $Z_{\infty}$ is the limit of the derivative martingale

$$
Z_{t}=\sum_{u \in N_{t}}\left(2 t-X_{u}(t)\right) e^{X_{u}(t)-2 t}.
$$

We have 
$$
\begin{aligned}
&U_{c_{*}}(x+x^{i}_{\infty})=1-\mathbb{E} \exp \left(-e^{-x-x^{i}_{\infty}} Z_{\infty}\right) = 1-\mathbb{E} \exp \left(-e^{-x+\ln(\alpha^{i}_{\infty})} Z_{\infty}\right)=  1-\mathbb{E} \exp \left(-e^{-x} \alpha^{i}_{\infty} Z_{\infty}\right)= \\
&=1-\mathbb{E} \exp \left(-e^{-x} \frac{\alpha^{k-i}}{(k-i)!}C^{*}Z_{\infty}\right) \text{ where } C^{*}=\alpha^{k}_{\infty} \text{ is the single Fisher-KPP equation constant.}
\end{aligned}
$$

Thus the convergence of $\tilde{v^{1}}$ to the traveling wave in \eqref{convergencetothewaveequation}, combined with \eqref{cdfcascadingbbmequation} implies for $M_{t}=\max _{u \in \mathcal{N}_{t}} X_{u}(t)$, the maximum particle of the cascading system:

$$\lim _{t \rightarrow \infty} \mathbb{P}\left(M_{t} \leq m(t) +x\right)=\mathbb{E} \exp \left(-e^{-x} \frac{\alpha^{k-1}}{(k-1)!}C^{*}Z_{\infty}\right)$$ for all $x \in \mathbb{R}$.

We thus obtain as conjectured in \cite{belloumpaper}, that $M_{t}-m(t)$ converges in distribution to a randomized Gumbel scaled by the constant $\frac{\alpha^{k-1}}{(k-1)!}C^{*}$ where $C^{*}$ is the constant of the single Fisher-KPP equation.

\section*{Acknowledgment}
The author is grateful to Lenya Ryzhik for his constant guidance and encouragement, and thanks Montie Avery, Cole Graham, and Matt Holzer for helpful discussions.

\printbibliography

\end{document}